\documentclass[11pt,reqno]{amsart}
\usepackage{}
\usepackage{esint}
\usepackage{bbm}
\usepackage{bm}
\usepackage{amssymb}
\usepackage{mathrsfs}
\usepackage{amsfonts}
\usepackage{amsfonts,amssymb,amsmath,amsthm}
\usepackage{url}
\usepackage{enumerate}
\usepackage[pdftex,bookmarksnumbered]{hyperref}
\usepackage{graphicx,xcolor}
\usepackage{pgfplots}
\usepackage{float}
\usepackage{caption}
\usepackage{ulem}
\usepackage{amsmath}
\usepackage{mathtools}
\makeatletter %for caption
\def\@captype{figure} %for caption
\makeatother  %for caption

\usetikzlibrary{patterns,arrows,positioning,calc,fadings,shapes,decorations.markings}
\usepackage{tikz}

\usepackage{multicol}%column
\newcommand{\ds}{\displaystyle}
\newtheorem{theorem}{Theorem}[section]
\newtheorem{lemma}[theorem]{Lemma}

\newtheorem{corollary}[theorem]{Corollary}
\theoremstyle{definition}
\newtheorem{definition}[theorem]{Definition}
\newtheorem{remark}[theorem]{Remark}

\numberwithin{equation}{section}

\newtheorem{example}{Example}

\usepackage[left=2.0cm, right=2.0cm, top=2.0cm, bottom=2.0cm]{geometry}

\usepackage{graphicx}%
\usepackage{color}
\usepackage{cite}
\usepackage{hyperref}
\hypersetup{
	colorlinks = true,%
	citecolor = blue,%[rgb]{0.65,0.0,0.0},
	filecolor=red,%
	linkcolor = [rgb]{0.65,0.0,0.0},%
	anchorcolor = red,
	pagecolor = red,
	urlcolor= [rgb]{0.65,0.0,0.0}
	linktocpage=true,
	pdfpagelabels=true,
	bookmarksnumbered=true,
}

\makeatother%用于连接公式编号

\makeatletter
\@namedef{subjclassname@2020}{\textup{2020} Mathematics Subject Classification}
\makeatother

\DeclareMathOperator{\m}{\mathfrak{m}}
\DeclareMathOperator{\dm}{d\mathfrak{m}}

\DeclareMathOperator{\vol}{vol}

\DeclareMathOperator{\dd}{d}

\DeclareMathOperator{\supp}{supp}

\DeclareMathOperator{\dvol}{dvol}
\DeclareMathOperator{\dii}{div}

\DeclareMathOperator{\B}{\mathsf{B}}

\DeclareMathOperator{\FMMM}{\mathsf{FMMM}}

\newcommand{\Lip}{\mathsf{Lip}}
\newcommand{\dil}{\mathsf{dil}}

\author{Benling Li}
\address{
School of Mathematics and Statistics\\
Ningbo University\\
315211 Ningbo, China}
\email{libenling@nbu.edu.cn}

\author{Wei Zhao}
\address{
School of Mathematics\\
East China University of Science and Technology\\
200237 Shanghai, China}
\email{szhao\underline{ }wei@yahoo.com}

\date{\today}


\thanks{B. Li is supported by Natural Science Foundation of Ningbo (No.~2024J017). W. Zhao
is supported by Natural Science Foundation of China (No.~12471045). }

\begin{document}

\title[Unexpected Analytic Phenomena on Finsler Manifolds]{Unexpected Analytic Phenomena on Finsler Manifolds}

\begin{abstract}In the Riemannian setting, every flat Cartan--Hadamard manifold is isometric to Euclidean space, the canonical model that  underlies the theory of Sobolev spaces and guarantees  the sharpness/rigidity of the Hardy inequality, the uncertainty principle, and the Caffarelli--Kohn--Nirenberg (CKN) inequality.
In this paper, we show that on a flat Finsler Cartan--Hadamard manifold --- Berwald's metric space --- the classical picture alters radically: the Nash embedding theorem fails, the Sobolev space becomes nonlinear,  and the Hardy and uncertainty inequalities break down completely, whereas the CKN inequality exhibits a sharp threshold in its validity depending on a parameter. By contrast, on Funk metric spaces --- a class of Finsler Cartan--Hadamard manifolds of constant negative curvature --- this threshold behaviour disappears, although all the other non-Riemannian features persist. We trace this divergence to the lower bound of the $S$-curvature. As a consequence, the failure of the aforementioned functional inequalities is established for a broad class of Finsler manifolds that includes both Berwald's and the Funk metric spaces.

\end{abstract}
\maketitle

\section{Introduction}
In 1996, S. S. Chern remarked that Finsler geometry may be regarded, in essence, as Riemannian geometry freed from the quadratic constraint \cite{Ch1}.
This liberation from the quadratic mold is precisely what unlocks a far richer geometric landscape, wherein profound distinctions from the Riemannian setting arise as natural and inherent features.

A quintessential non-Riemannian feature afforded by this greater generality is the asymmetry of Finsler metrics --- a property that leads to a variety of unexpected phenomena.
A striking illustration is provided by the Funk metric on the Euclidean unit ball $\mathbb{B}^n_1(\mathbf{0}):=\{x\in \mathbb{R}^n\,:\,|x|<1\}$, defined by
\begin{align}\label{Funkex1intro}
 \mathsf{F}(x,y) = \frac{ \sqrt{(1-|x|^2)|y|^2 + \langle x, y \rangle^2}}{1-|x|^2} + \frac{\langle x, y \rangle}{1-|x|^2}, \qquad \forall (x,y)\in T\mathbb{B}^n_1(\mathbf{0}).
 \end{align}
where $|\cdot|$ and $\langle\cdot,\cdot\rangle$ denote the $n$-dimensional Euclidean norm and inner product, respectively (cf.~Funk \cite{Funk1}). On the one hand,  $(\mathbb{B}^n_1(\mathbf{0}),\mathsf{F})$ is a  Finsler Cartan--Hadamard % (i.e., simply connected forward complete)
manifold of constant negative curvature $-1/4$, thus serving as a Finslerian analogue of the Poincar\'e disc. On the other hand, it exhibits extreme asymmetry: for the radial direction $\mathsf{y}:=x/|x|$,
\begin{equation*}%\label{Finerlasym}
\lim_{|x|\rightarrow 1^-}\mathsf{F}(x,\mathsf{y})=+\infty,\qquad \lim_{|x|\rightarrow 1^-}\mathsf{F}(x,-\mathsf{y})=2^{-1}.
\end{equation*}
Consequently, the Nash embedding theorem fails in this setting; that is,
$(\mathbb{B}^n_1(\mathbf{0}),\mathsf{F})$ admits no isometric embedding into any Minkowski space.
%A direct consequence is that the distance function induced by $\mathsf{F}$ satisfies
Moreover, the asymmetry is also reflected in the associated distance function, which satisfies
\[
 \lim_{|x|\rightarrow 1^-}d_{\mathsf{F}}(\mathbf{0},x)=+\infty,\qquad \lim_{|x|\rightarrow 1^-}d_{\mathsf{F}}(x,\mathbf{0})=\ln2.
 \]
These consequences of asymmetry further extend to functional-analytic structures: the associated Sobolev spaces become nonlinear, as shown by Farkas--Krist\'aly--Varga \cite{FKV} and Krist\'aly--Rudas \cite{KR}.
%A similar phenomenon was later observed by Zhang--Zhao \cite{ZZ} in the corresponding Wasserstein spaces.

A second fundamental distinction is the absence of a canonical measure on a Finsler manifold. Unlike the Riemannian setting, various natural choices exist, and their analytic properties can differ substantially (see, e.g., \'Alvarez Paiva--Berck \cite{AB} and \'Alvarez Paiva--Thompson \cite{AT}).
This is vividly illustrated by sharp functional inequalities.
On any Riemannian Cartan--Hadamard manifold endowed with its canonical Riemannian measure, the Heisenberg--Pauli--Weyl (HPW) uncertainty principle, the Caffarelli--Kohn--Nirenberg (CKN) interpolation  inequality, and the Hardy inequality are all sharp (cf.~Krist\'aly \cite{Kris2}).
Recently, in contrast, Krist\'aly--Li--Zhao \cite{KLZ} showed that all three inequalities fail on the Finsler Cartan--Hadamard manifold $(\mathbb{B}^n_1(\mathbf 0),\mathsf F)$ equipped with the Busemann--Hausdorff measure. Since on this manifold the Busemann--Hausdorff measure coincides with the standard Lebesgue measure
 $\mathscr{L}^n$
  up to a multiplicative constant, one has
 %the standard Lebesgue measure $\mathscr{L}^n$ {\color{red}(which is equivalent to the Busemann--Hausdorff measure)}: %Concretely,
%for certain ranges of the parameters, the best constants in the HPW/CKN and Hardy inequalities vanish (see \eqref{uncertainfalyonB}--\eqref{uncertainfalyonB2} for precise statements).
\begin{itemize}
\item HPW uncertainty principle/CKN interpolation inequality: for $p=2$ and $q=0$, or for $0<q<2<p$ with $2<n<2(p-q)/(p-2)$,
\begin{equation}\label{uncertainfalyonB}
\inf_{u\in C^\infty_0(\mathbb{B}^n_1(\mathbf{0}))\backslash\{0\}}\frac{\left( \int_{\mathbb{B}^n_1(\mathbf{0})}  \mathsf{F}^2(\nabla u) {\dd}x \right) \left(  \int_{\mathbb{B}^n_1(\mathbf{0})}  \frac{|u|^{2p-2}}{d_{\mathsf{F}}(\mathbf{0},\cdot)^{2q-2}}  {\dd}x \right)}{\left(\int_{\mathbb{B}^n_1(\mathbf{0})} \frac{|u|^p}{d_{\mathsf{F}}(\mathbf{0},\cdot)^q}{\dd}x\right)^2}=0;
\end{equation}

\item Hardy inequality: for any $p\in (1,n)$,
\begin{equation}\label{uncertainfalyonB2}
\inf_{u\in C^\infty_0(\mathbb{B}^n_1(\mathbf{0}))\backslash\{0\}}\frac{  \int_{\mathbb{B}^n_1(\mathbf{0})}  \mathsf{F}^p(\nabla u) {\dd}x    }{\int_{\mathbb{B}^n_1(\mathbf{0})} \frac{|u|^p}{d_{\mathsf{F}}(\mathbf{0},\cdot)^p}{\dd}x}=0.
\end{equation}
\end{itemize}
The phenomenon, in fact, extends to the whole class of Funk metric spaces. Recall that a Finsler manifold $(\Omega,\mathsf{F})$ is called a Funk metric space if $\Omega$ is a bounded domain in $\mathbb{R}^n$ and
\begin{equation*}%\label{funkdef}
x + \frac{y}{\mathsf{F}(x,y)}  \in \partial \Omega, \qquad \forall  x\in \Omega, \  y \in T_x \Omega\backslash\{0\}.
\end{equation*}
Such spaces have constant curvature $-1/4$. When $\Omega=\mathbb{B}^n_1(\mathbf{0})$,
 the metric $\mathsf{F}$
  reduces to the one given in \eqref{Funkex1intro}. The underlying reason for the failure of these inequalities
is that the $S$-curvature (associated with the Lebesgue measure, cf.~Shen \cite{ShenAdv,ShenLecture}) --- a quantity with no Riemannian analogue --- satisfies the following dominance condition over the Ricci curvature:
%The underlying reason is a genuinely  non-Riemannian feature --- the $S$-curvature induced by the Lebesgue measure (cf. Shen \cite{ShenAdv,ShenLecture}) --- which satisfies a  dominance condition over the Ricci curvature:
%Indeed, both \eqref{uncertainfalyonB} and \eqref{uncertainfalyonB2} remain valid for all kinds of Funk metric spaces $(\Omega,\mathsf{F})$, because  the $S$-curvature $\mathbf{S}$ (induced by the Lebesgue measure) and the Ricci curvature satisfy
\begin{equation}\label{strongScurvature}
\inf_{(x,y)\in T\Omega\backslash\{0\}}\frac{\mathbf{S}(x,y)}{\mathsf{F}(x,y)}>\sup_{(x,y)\in T\Omega\backslash\{0\}} \sqrt{(n-1)  \max\left\{0,\frac{-\mathbf{Ric}(x,y)}{\mathsf{F}(x,y)}\right\}}.
\end{equation}
For further properties of Funk metric spaces we refer to Faifman \cite{Fai}, Li \cite{Li}, Kaj\'ant\'o--Krist\'aly \cite{Ka2} and Shen \cite{ShenSpray,Sh0,Sh1}.

The discussion thus far has focused on the setting of negative curvature. In the case of zero curvature, the Riemannian situation is rigid: every complete, simply connected, flat Riemannian manifold is isometric to Euclidean space. This canonical model underpins the theory of Sobolev spaces and the sharpness/rigidity of all the functional inequalities aforementioned (see e.g. Cazacu--Flynn--Lam--Lu\cite{CFLL}, Hardy \cite{Hardy}, Hebey \cite{He} and Krist\'aly \cite{Kris2}).
In contrast to Funk metric spaces, a natural question arises in the Finsler setting:
\begin{quote}
\textit{Does there exist a flat Finsler Cartan--Hadamard manifold that exhibits the same unexpected non-Riemannian features
under a weaker $S$-curvature condition than \eqref{strongScurvature}?}
\end{quote}
\noindent
This paper provides a complete affirmative answer.
The candidate is furnished by the classical Berwald's metric --- a projectively flat, zero-curvature Finsler metric on $\mathbb{B}^n_1(\mathbf{0})$, originally constructed by Berwald in relation to Hilbert's fourth problem \cite{Be1,Be2}. It is explicitly given by
 \begin{equation*}%\label{Berwaldmetric}
\B(x,y):= \frac{ (\sqrt{ (1-|x|^2) |y|^2 + \langle x, y \rangle^2} +\langle x, y \rangle )^2}{ (1-|x|^2)^2 \sqrt{ (1-|x|^2) |y|^2 + \langle x, y \rangle^2} }, \qquad \forall (x,y)\in T\mathbb{B}^n_1(\mathbf{0}).
\end{equation*}
Thus $(\mathbb{B}^n_1(\mathbf{0}),\B)$
serves as a flat Finslerian counterpart of the Poincar\'e disc.

In what follows, let $p'$
 denote the H\"older conjugate exponent of
$p$. Our first main result is the following.
\begin{theorem}\label{Berwmainth} Berwald's metric space $(\mathbb{B}^n_1(\mathbf{0}),\B)$ is a   flat Cartan--Hadamard manifold satisfying:
\begin{enumerate}[\rm (i)]
% \item\label{CKNFIALBE1} $\mathbf{K}=0$ and $\mathbf{S}>0$;

\item\label{CKNFIALBE1}  $(\mathbb{B}^n_1(\mathbf{0}),\B)$ admits no isometric embedding into any Minkowski space;

\item\label{CKNFIALBE2} the Sobolev space $W^{1,p}_0(\mathbb{B}^n_1(\mathbf{0}),\B,\mathscr{L}^n)$ is not a vector space for all $p\in (1,+\infty)$;

%\item\label{CKNFIALBE3}  the forward and  backward topologies of the $L^p$-Wasserstein space are incompatible for all $p\in [1,+\infty)$;

\item\label{CKNFIALBE4}
the $L^p$-Hardy inequality fails for all $p\in (1, n)$, i.e.,
		\[
		\inf_{f\in C^\infty_0({\mathbb{B}^n_1(\mathbf{0})})\backslash\{0\}}\frac{\int_{\mathbb{B}^n_1(\mathbf{0})} {\B}^p (\nabla f) {\dd}x}{\int_{\mathbb{B}^n_1(\mathbf{0})} \frac{|f|^p}{d_{\B}(\mathbf{0},\cdot)^{p}}{\dd}x}=0.
		\]
		
\item\label{CKNFIALBE5}
the generalized uncertainty principle fails, i.e., for any $-p+1<s\leq 1<p<n$,
		\[
		\inf_{f\in C^\infty_0({\mathbb{B}^n_1(\mathbf{0})})\backslash\{0\}}\frac{\left( \int_{{\mathbb{B}^n_1(\mathbf{0})}} {\B}^p (\nabla f){\dd}x \right)^{1/p}\left(  \int_{{\mathbb{B}^n_1(\mathbf{0})}}|f|^p d_{\B}(\mathbf{0},\cdot)^{p's} {\dd}x \right)^{1/{p'}} }{ \int_{{\mathbb{B}^n_1(\mathbf{0})}} |f|^p d_{\B}(\mathbf{0},\cdot)^{s-1} {\dd}x   }=0.
		\]

\end{enumerate}
\end{theorem}

Kaj\'ant\'o--Krist\'aly--Peter--Zhao \cite{KKPZ} established the sharp  generalized uncertainty principle on flat Riemannian Cartan--Hadamard manifolds. However, Theorem \ref{Berwmainth}/\eqref{CKNFIALBE5} demonstrates that it \textit{fails completely} in the flat Finsler setting; consequently, all known forms of the HPW uncertainty principle (cf.~Kombe--\"Ozaydin\cite{KO}, Kirst\'aly\cite{Kris2}, and Nguyen \cite{Ng}) and the hydrogen uncertainty principle (cf.~Cazacu--Flynn--Lam\cite{CFL} and Frank \cite{Fr}) --- both of which holds in the Riemannian setting --- also fail there.

A natural extrapolation would suggest that the generalized Caffarelli--Kohn--Nirenberg (CKN) inequality --- which holds in the flat Riemannian case \cite{KKPZ} --- should likewise fail on $(\mathbb{B}^n_1(\mathbf{0}),\B)$. Surprisingly, the actual picture is more delicate.
\begin{theorem}\label{sbeer} On Berwald's metric space  $(\mathbb{B}^n_1(\mathbf{0}),\B)$,  let  $1<p<m$ and $p(n+s-1)>m(n-p)>0$.
\begin{enumerate}[\rm (i)]

\item\label{BerwCKN_i} If  $1-p<s\leq 2$, then the generalized Caffarelli--Kohn--Nirenberg inequality fails:
		\[
		\inf_{f\in C^\infty_0({\mathbb{B}^n_1(\mathbf{0})})\backslash\{0\}}\frac{\left( \int_{{\mathbb{B}^n_1(\mathbf{0})}}{\B}^p (\nabla f) {\dd}x   \right)^{1/p}\left( \int_{{\mathbb{B}^n_1(\mathbf{0})}}  {|f|^{p'(m-1)}}{d_{\B}(\mathbf{0},\cdot)^{p's}}{\dd}x   \right)^{1/p'}  }{ \int_{{\mathbb{B}^n_1(\mathbf{0})}}  {|f|^{m}}{d_{\B}(\mathbf{0},\cdot)^{s-1}}{\dd}x  }=0.
		\]
\item\label{BerwCKN_ii} If $s>2$, then the generalized Caffarelli--Kohn--Nirenberg inequality holds: for any $f\in C^\infty_0(\mathbb{B}^n_1(\mathbf{0}))$,
\[
 {\left( \int_{{\mathbb{B}^n_1(\mathbf{0})}}{\B}^p (\nabla f) {\dd}x   \right)^{1/p}\left( \int_{{\mathbb{B}^n_1(\mathbf{0})}}  {|f|^{p'(m-1)}}{d_{\B}(\mathbf{0},\cdot)^{p's}}{\dd}x   \right)^{1/p'}  }\geq \frac{ s-2}{4m} { \int_{{\mathbb{B}^n_1(\mathbf{0})}}  {|f|^{m}}{d_{\B}(\mathbf{0},\cdot)^{s-1}}{\dd}x},
\]
with equality if and only if $f=0$.
\end{enumerate}
\end{theorem}
Note that the hypotheses $1<p<m$ and $p(n+s-1)>m(n-p)>0$ excludes the case
 $s \le 1-p$.
Thus, in contrast to the complete failure of the uncertainty principle, the generalized CKN inequality exhibits a sharp threshold behavior in the parameter $s=2$, separating regimes of validity from those of failure. This further underscores how non-Riemannian geometry can alter the very structure of classical functional inequalities.

To our knowledge, this threshold phenomenon is observed here for the first time in a Finsler setting. Its occurrence is notable because, on Finsler manifolds with extreme asymmetry (i.e., infinite reversibility), the relevant functional inequalities typically cannot be established in a global form;  see Huang--Krist\'aly--Zhao\cite{HKZ}, Kaj\'ant\'o\cite{Ka}, and Mester--Peter--Varga\cite{MPV}. For example, while the conclusions of Theorem \ref{Berwmainth} extend to all Funk metric spaces (see Section \ref{sectionFunk}), the generalized CKN inequality fails unconditionally there, as the following result shows.

%The conclusions of Theorem \ref{Berwmainth} extend to all (see Theorems \ref{MainThmFunk} and  \ref{ThmFunkInequalities}).
%Moreover, in view of \eqref{uncertainfalyonB}, the failure of a specific CKN inequality on every Funk metric space naturally raises the question: could an analogue of Theorem \ref{sbeer} hold there? The answer is no: for Funk metrics .
\begin{theorem}\label{funkcknf}
On  a Funk metric space  $(\Omega,\mathsf{F})$,  for any  $1<p<m$ and $p(n+s-1)>m(n-p)>0$,
		\[
		\inf_{f\in C^\infty_0({\mathbb{B}^n_1(\mathbf{0})})\backslash\{0\}}\frac{\left( \int_{{\mathbb{B}^n_1(\mathbf{0})}}{\mathsf{F}}^p (\nabla f) {\dd}x   \right)^{1/p}\left( \int_{{\mathbb{B}^n_1(\mathbf{0})}}  {|f|^{p'(m-1)}}{d_{\mathsf{F}}(\mathbf{0},\cdot)^{p's}}{\dd}x   \right)^{1/p'}  }{ \int_{{\mathbb{B}^n_1(\mathbf{0})}}  {|f|^{m}}{d_{\mathsf{F}}(\mathbf{0},\cdot)^{s-1}}{\dd}x  }=0.
		\]
\end{theorem}

The decisive difference lies in the $S$-curvature. On a Funk metric space it admits a uniform positive lower bound \eqref{strongScurvature}, while on Berwald's metric space it decays polynomially (of order $(1+d_{\B}(\mathbf{0},x))^{-1}$), and in particular

%The essential difference lies in the behavior of the S-curvature (associated with the Lebesgue measure): on a Funk metric space it admits a strictly positive lower bound (see \eqref{strongScurvature}), whereas on the Berwald's metric space it is only bounded below by a quantity with polynomial decay (of order $(1+d_{\B}(\mathbf{0},x))^{-1}$) and particularly (compared with \eqref{strongScurvature}),
\begin{equation}\label{strongScurvature2}
\inf_{(x,y)\in T\mathbb{B}^n_1(\mathbf{0})\backslash\{0\}}\frac{\mathbf{S}(x,y)}{\mathsf{B}(x,y)}= \sup_{(x,y)\in T\Omega\backslash\{0\}} \sqrt{(n-1)  \max\left\{0,\frac{-\mathbf{Ric}(x,y)}{\mathsf{B}(x,y)}\right\}} =0.
\end{equation}
Crucially, \eqref{strongScurvature2} alone is not enough to support the threshold behavior in Theorem \ref{sbeer}: for instance, every Minkowski space endowed with the Lebesgue measure satisfies \eqref{strongScurvature2} yet the generalized CKN inequality holds there, and none of the pathologies in Theorem \ref{Berwmainth} occur.
%It should be remarked that  \eqref{strongScurvature2} alone is insufficient to support Theorem \eqref{sbeer}. For example, for every Minkowski space endowed with the Lebesgue measure,  \eqref{strongScurvature2} is satisfied while  the generalized Caffarelli--Kohn--Nirenberg inequality holds; indeed, none of the statements in Theorem \ref{Berwmainth} holds in this case.

This contrast motivates a general investigation. We obtain the following characterization.
\begin{theorem}\label{voluemcompar22}
Let $(M,F)$ be an $n$-dimensional forward complete Finsler manifold  with $\mathbf{Ric} \geq -(n-1)k^2$ for some $k\geq 0$ and  let $\m$ be a smooth positive measure on $M$. Suppose that  there exists  a point $o\in M$ such that
\begin{equation}\label{remscurvature}
\mathbf{S}(\nabla r)\geq  (n-1)k+\frac{C}{1+r},
\end{equation}
where $r(x)=d_F(o,x)$ is the distance function from $o$ and $C$ is a constant satisfying
\[
C>1  \ \text{ if } \ k>0; \qquad C \geq n  \ \text{ if } \ k=0.
\]
 Then the following statements are true:
\begin{enumerate}[\rm (i)]

\item \label{weakHardyinB}
The $L^p$-Hardy inequality fails for all $p\in (1, n)$, i.e.,
		\[
		\inf_{f\in C^\infty_0(M)\backslash\{0\}}\frac{ \int_{M} {F}^{*p}({\rm d} f) \dm}{ \int_{M}  {|f|^p}{r^{-p}}\dm}=0.
		\]
		
\item \label{generunicertpinweakcurva}
The generalized uncertainty principle fails, i.e., for any $-p+1<s\leq 1<p<n$ ,
		\[
		\inf_{f\in C^\infty_0({M})\backslash\{0\}}\frac{\Big(  \int_{{M}} F^{*p}({\rm d} f){\dm} \Big)^{1/p}\Big(   \int_{{M}}|f|^p r^{p's} {\dm} \Big)^{1/{p'}} }{  \int_{{M}} |f|^p r^{s-1} {\dm}   }=0.
		\]

\item \label{geenerCknineq}
The generalized  Caffarelli--Kohn--Nirenberg inequality fails for certain parameters. Specifically, if $1<p<m$, $p(n+s-1)>m(n-p)>0$,
and
\[
s < \begin{cases}
C, & \text{if } k > 0, \\[4pt]
C - n + 1, & \text{if } k = 0,
\end{cases}
\]
		then
		\begin{equation}\label{ckngenifr}
		\inf_{f\in C^\infty_0({M})\backslash\{0\}}\frac{ \Big( \int_{{M}}F^{*p}({\rm d} f) {\dm}   \Big)^{1/p}\Big( \int_{{M}}  {|f|^{p'(m-1)}}{r^{p's}}{\dm}   \Big)^{1/p'}  }{ \int_{{M}}  {|f|^{m}}{r^{s-1}}{\dm}  }=0.
		\end{equation}

\item \label{generCKNin} If  the $S$-curvature is strengthened to
\[
\inf_{(x,y)\in TM\backslash\{0\}}\frac{\mathbf{S}(x,y)}{F(x,y)}> (n-1)k,
\]
 then
the generalized Caffarelli--Kohn--Nirenberg inequality fails completely, i.e., \eqref{ckngenifr} holds for  any $1<p<m$ and $p(n+s-1)>m(n-p)>0$.
\end{enumerate}
\end{theorem}

Taking
$o=\mathbf{0}$ and endowing the space with the Lebesgue measure, Berwald's metric space satisfies $\mathbf{Ric}\equiv0$ and $\mathbf{S}(\nabla r)=\frac{n+1}{1+r}$  whereas  the Funk metric space satisfies $\mathbf{Ric}\equiv-\frac{(n-1)}4$ and $\mathbf{S}\equiv\frac{(n+1)}2$. Therefore,
Theorem \ref{voluemcompar22} accounts for the failure of the functional inequalities on both settings through their curvature profiles.  Notably, the
$S$-curvature condition in Theorem \ref{voluemcompar22} is considerably weaker than that in \cite{KLZ}, yet its conclusions are stronger.
Finally,  Theorem \ref{voluemcompar22} remains valid even when the condition \eqref{remscurvature} is weakened to
\begin{equation}\label{weakSball}
\mathbf{S}(\nabla r)\geq  (n-1)k+\frac{C}{1+r} \quad \text{ in \  }   M\backslash B^+_R(o),
\end{equation}
for some $R>0$; see Remark \ref{S-weakcurvature} for details.

Since the analysis of PDEs on manifolds relies on the functional inequalities governing Sobolev spaces, our results indicate that the functional-analytic foundation for studying elliptic and parabolic equations on Finsler manifolds is sensitive to non-Riemannian geometric data. The breakdown of classical inequalities calls for new methods to establish existence, regularity, and stability of solutions to Finslerian PDEs -- a direction ripe for future work.

\smallskip

The paper is organized as follows. Section \ref{prelimain} collects the necessary preliminaries on Finsler geometry and Sobolev spaces. Section \ref{sectionBerwald} studies the geometry of Berwald's metric space and contains the proofs of Theorems \ref{Berwmainth} and \ref{sbeer}. Section \ref{sectionFunk} presents geometric analysis of Funk metric spaces, providing  results parallel to those in Theorem \ref{Berwmainth} and proving Theorem \ref{funkcknf}. In both Sections \ref{sectionBerwald} and \ref{sectionFunk} we avoid Finslerian comparison theory and give direct calculations to keep the presentation self-contained and transparent. Finally, Section \ref{RiccSboundver} treats the general case and establishes Theorem \ref{voluemcompar22}.

%The paper is organized as follows. Section \ref{prelimain} presents preliminaries on Finsler geometry and Sobolev spaces in the Finsler setting. In Section \ref{sectionBerwald}, we study the geometry of the Berwald metric space and prove Theorems \ref{Berwmainth} and \ref{sbeer}. Section \ref{sectionFunk} contains the geometric analysis of Funk metric space, providing parallel results to the statements in Theorem \ref{Berwmainth} and proving Theorem \ref{funkcknf}. In Sections \ref{sectionBerwald} and \ref{sectionFunk}, instead of using the Finslerian comparison theory, we give direct calculations to make the results more readable. Section \ref{RiccSboundver} addresses the general case and proves Theorem \ref{voluemcompar22}.

\section{Preliminaries}\label{prelimain}

\subsection{Elements from Finsler geometry} In this section, we recall some definitions and properties from Finsler geometry; for details see
Bao--Chern--Shen \cite{BCS}, Ohta \cite{Ohta1} and Shen \cite{ShenSpray,ShenLecture}. Naturally, Einstein's summation convention is used throughout the paper.

%A \textit{Finsler $n$-manifold} $(M,F)$ is an $n$-dimensional differential manifold $M$ equipped with a Finsler metric $F$ which is a nonnegative function on $TM$ satisfying the following two conditions:
%
%(1) $F$ is positively homogeneous, i.e., $F(\lambda y)=\lambda F(y)$ for any $\lambda>0$ and $y\in TM$;
%
%
%(2) $F$ is smooth on $TM\backslash\{0\}$ and the Hessian $\frac{1}{2}[F^2]_{y^iy^j}(x,y)$ is positive definite, where $F(x,y):=F(y^i\frac{\partial}{\partial x^i}|_x)$.

Let $V$ be an $n$-dimensional vector space with $n \geq 2$. A {\it Minkowski norm} on $V$ is a nonnegative function $\phi: V \rightarrow \mathbb{R}$ satisfying:
\begin{enumerate}[{\rm (i)}]
\item\label{minkownor1} $\phi(y)\geq 0$ with equality if and only if $y=0$;
\item\label{minkownor2} $\phi$ is positively $1$-homogeneous, i.e., $\phi(\lambda y) = \lambda \phi(y)$ for all $\lambda > 0$;
\item  the Hessian matrix $\left( [\phi^2]_{y^i y^j}(y) \right)$ is positive definite for all $y \in V \backslash\{0\}$.
\end{enumerate}
The pair $(V, \phi)$ is then called a   {\it Minkowski space}.

 For any Minkowski norm $\phi$ on $\mathbb{R}^n$, the following inequality holds (see \cite[(1.2.3)]{BCS}):
\begin{equation}\label{weakswaza}
y^i\frac{\partial\phi}{\partial{y^i}}(x)\leq \phi(y), \qquad \forall y=y^i\frac{\partial}{\partial y^i}, \ x\in \mathbb{R}^n\backslash\{\mathbf{0}\},
\end{equation}
 with equality if and only if $y=\lambda x$ with $\lambda\geq 0$. Here $(y^i)$ are the standard Cartesian coordinates on \( \mathbb{R}^n \).

Let $M$ be an $n$-dimensional connected smooth manifold with tangent bundle $TM$.
 A  {\it Finsler metric} $F = F(x,y)$  is a $C^{\infty}$-function defined on $TM\backslash\{0\}$ such
 that $F(x,\cdot)$ is a  Minkowski  norm on $T_{x}M$ for every $x\in M$. The pair $(M,F)$ is called a  {\it Finsler manifold}. Thus, a Minkowski space $(V,\phi)$ is the simplest model of Finsler manifolds by setting $M:=V$ and $F(x,y):=\phi(y)$.

 The \textit{fundamental tensor} $g_y = (g_{ij}(x, y))$ is defined by
\begin{align}\label{defbasictensor}
    g_{ij}(x, y) := \frac{1}{2} \frac{\partial^2 F^2}{\partial y^i\partial y^j}(x, y), \qquad \forall (x, y) \in TM \backslash \{0\}.
\end{align}
By Euler's theorem, $F^2(x,y)=g_{ij}(x,y)y^iy^j$ for all $(x,y)\in TM\backslash\{0\}$. Note that $g_{ij}$ can be viewed as a local function on $TM\backslash\{0\}$, but it cannot be defined on $y=0$ unless $F$ is Riemannian, in which case $g_{ij}(x,y)$ is independent of $y$.

%The metric $F$ is called \textit{Riemannian} if $g$ is independent of the direction $y$, and \textit{Minkowskian} if it is independent of the position $x$. As usual, $(g^{ij})$ denotes the inverse  of the matrix $(g_{ij})$.

A distinctive non-Riemannian feature is the possible asymmetry of $F$.
Its {\it reversibility}  is defined by
\begin{equation}\label{defrever}
\lambda_F(M):=\sup_{x\in M}\lambda_F(x) \ \ {\rm with}\ \ \lambda_F(x)= \sup_{y\in T_xM\setminus\{0\}} \frac{F(x,-y)}{F(x,y)},
\end{equation}
 see Rademacher \cite{Rade1,Rade}.   Hence, $\lambda_F(M) =1$ if and only $F $ is {\it reversible}, i.e., $F(x,y)=F(x,-y)$.

The {\it Legendre transformation} $\mathfrak{L} : TM \rightarrow T^*M$ is defined
by
\begin{equation}\label{defleg}
\mathfrak{L}(x,y):=\left \{
\begin{array}{lll}
 g_y(y,\cdot), & \text{ if } (x,y)\in TM\backslash\{0\},  \\[4pt]
 0, & \text{ if } (x,y)=0\in TM,%
\end{array}
\right.
\end{equation}
where $g_y(y,\cdot):=g_{ij}(x,y)y^i{\dd}x^j$.
For a smooth function $f: M \rightarrow \mathbb{R}$, its {\it gradient}  is given by $\nabla f := \mathfrak{L}^{-1}({\dd}f)$; consequently,  ${\dd}f(X) = g_{\nabla f} (\nabla f,X)$ on the set $\{ {\dd}f \neq 0\}$.

The {\it co-metric} (or {\it dual metric}) $F^*$ of $F$ is
defined by
\begin{equation}\label{coFinsler}
F^*(x,\eta):=\underset{y\in T_xM\backslash \{0\}}{\sup}\frac{\langle\eta,y\rangle}{F(x,y)}, \qquad
\forall \eta\in T_x^*M,
\end{equation}
which itself is   a Finsler metric on $T^*M$.  Here  $\langle \eta,y\rangle:=\eta(y)$
 denotes the canonical pairing between $T^*_xM$
and $T_xM$. A  Cauchy--Schwarz type inequality then holds:
\begin{equation}\label{dualff*}
\langle\eta,y\rangle\leq F^*(x,\eta) F(x,y), \qquad  \forall \eta\in T^*_xM,\ y\in T_xM,
\end{equation}
with equality if and only if $\eta=\alpha \mathfrak{L}(x,y)$ for some $\alpha>0$. In particular,
\begin{equation}\label{F*F}
F^*(\mathfrak{L}(x,y))=F(x,y), \qquad \forall (x,y)\in TM,
\end{equation}

A smooth curve $t\mapsto \gamma(t)$ in $M$ is called a (constant-speed) \textit{geodesic} if it satisfies
\[
\frac{{\dd}^2\gamma^i}{{\dd}t^2}+2G^i\left(\gamma,\frac{{\dd}\gamma}{{\dd}t}\right)=0,
\]
where the {\it geodesic coefficient} $G^i$ is  defined by
\begin{align*}
G^i(x,y):=\frac14 g^{il}(x,y)\left\{2\frac{\partial g_{jl}}{\partial x^k}(x,y)-\frac{\partial g_{jk}}{\partial x^l}(x,y)\right\}y^jy^k.
\end{align*}
In this paper, $\gamma_y(t)$ always denotes  a geodesic with $\dot{\gamma}_y(0)=y$.
A Finsler manifold $(M,F)$ is {\it forward complete} if  every geodesic $t\mapsto \gamma(t)$, $0\leq t<1$, can be extended to a geodesic defined on $0\leq t<+\infty$; similarly,  $(M,F)$ is  {\it backward complete} if  every geodesic $t\mapsto \gamma(t)$, $0< t\leq 1$, can be extended to a geodesic defined on $-\infty< t\leq 1$.

The {\it Riemannian curvature} $R_y$ of $F$ is a family of linear transformations on tangent spaces. Explicitly, it is given by
$R_y:=R^i_k(x,y)\frac{\partial}{\partial x^i}\otimes {\dd}x^k$, where
\[
R^i_{\,k}(x,y):=2\frac{\partial G^i}{\partial x^k}-y^j\frac{\partial^2G^i}{\partial x^j\partial y^k}+2G^j\frac{\partial^2 G^i}{\partial y^j \partial y^k}-\frac{\partial G^i}{\partial y^j}\frac{\partial G^j}{\partial y^k}.
\]
For a plane $\Pi:=\text{Span}\{y,v\}\subset T_xM$, the \textit{flag curvature} is defined by
\[
\mathbf{K}(y,\Pi):=\mathbf{K}(y,v):=\frac{g_y\left( R_y(v),v  \right)}{g_y(y,y)g_y(v,v)-g^2_y(y,v)}.
\]
For a Riemannian metric, the flag curvature reduces to the sectional curvature.
A {\it Cartan--Hadamard manifold} is a  simply connected, forward complete Finsler manifold
with non-positive flag curvature.

The {\it Ricci curvature} of $y\in T_xM\backslash\{0\}$ is defined by
\[
\mathbf{Ric}(x,y):=\frac{1}{F^2(x,y)}\sum_{i}\mathbf{K}(y,e_i),
\]
where $\{e_1,\ldots,e_n\}$ is a $g_y$-orthonormal basis of $T_xM$. A Finsler manifold $(M,F)$ is said to {\it satisfy $\mathbf{Ric}\geq  k$} for some $k\in \mathbb{R}$ if
\[
\mathbf{Ric}(x,y)\geq  k F(x,y),\qquad \forall  (x,y)\in TM.
\]

Let $\zeta:[0,1]\rightarrow M$ be a Lipschitz continuous path. Its \textit{length} is defined by
\[
L_F(\zeta):=\int^1_0 F({\zeta}(t),\dot{\zeta}(t))dt.
\]
The associated  {\it distance function} $d_F:M\times M\rightarrow [0,+\infty)$  is defined as
$d_F(x_1,x_2):=\inf L_F(\zeta)$,
where the infimum is taken over all
Lipschitz continuous paths $\zeta:[0,1]\rightarrow M$ with
$\zeta(0)=x_1$ and $\zeta(1)=x_2$. Generally $d_F(x_1,x_2)\neq d_F(x_2,x_1)$  unless $F$ is reversible.

The {\it forward} and {\it backward metric balls} centered at $x\in M$ with radii $R>0$ are defined respectively by
\begin{equation}\label{definball}
B^+_R(x):=\{z\in M:\, d_F(x,z)<R\},\qquad B^-_R(x):=\{z\in M:\, d_F(z,x)<R\}.
\end{equation}
The closure of a forward (resp., backward) metric ball with finite radius is compact if $(M,F)$ is forward (resp., backward) complete.

Let $\mathfrak{m}$ be a smooth positive measure on $M$; in a local coordinate system $(x^i)$ we
express ${\dd}\mathfrak{m}=\sigma(x){\dd}x^1\wedge\ldots\wedge {\dd}x^n$. In particular,
the {\it Busemann--Hausdorff measure} $\mathfrak{m}_{BH}$ is defined by
\begin{align*}
{\dd}\mathfrak{m}_{BH}:=\frac{\omega_n}{\vol(B_xM)}{\dd}x^1\wedge\ldots\wedge {\dd}x^n,
 \end{align*}
where $B_xM:=\{y\in T_xM: F(x,y)<1\}$ and  $\omega_n:=\vol(\mathbb{B}^n_1)$ is the volume of Euclidean $n$-dimensional unit ball. In the Riemannian setting, $\mathfrak{m}_{BH}$  coincides with the canonical Riemannian measure.

For convenience, a triple $(M,F,\m)$ is called a {\it Finsler metric measure manifold} ($\FMMM$) if $(M,F)$ is a forward complete Finsler manifold endowed with a smooth positive measure $\m$.

Let $(M,F,\m)$ be a $\FMMM$.
Define the \textit{distortion} as
\begin{equation*}
\tau(x,y):=\log \frac{\sqrt{\det g_{ij}(x,y)}}{\sigma(x)}, \qquad  \forall y\in T_xM\backslash\{0\},
\end{equation*}
and the \textit{$S$-curvature} as
\begin{equation}\label{Scurvaturedef}
\mathbf{S}(x,y):=\left.\frac{\dd}{{\dd}t}\right|_{t=0} \tau(\gamma_y(t),\dot{\gamma}_y(t)).
\end{equation}
Equivalently (cf.~Shen \cite{ShenLecture}),
\begin{equation}\label{Scurvature}
\mathbf{S}(x,y) = \frac{\partial G^i}{\partial  y^i }(x,y) - y^i \frac{\partial}{\partial x^i} \ln \sigma (x).
\end{equation}
 And $(M,F,\m)$ is said to  {\it satisfy $\mathbf{S}\geq  h$} for some $h\in \mathbb{R}$ if
\[
\mathbf{S}(x,y)\geq  h \,F(x,y), \qquad \forall (x,y)\in TM\backslash\{0\}.
\]

The remainder of this subsection is devoted to the polar coordinate system in the Finsler setting.
Fix a point $o\in M$ and set  $S_o M:=\{ z\in T_oM\,:\, F(o,z)=1\}$.
The {\it cut value} $i_y$ of $y \in S_oM$ and the {\it injectivity radius} $\mathfrak{i}_o$ at $o$ are defined by
\begin{align*}%\label{injectradia}
	i_y := \sup \{ t>0: \text{the geodesic }  \gamma_y|_{[0,t]}   \text{ is globally minimizing} \}, \qquad
	\mathfrak{i}_o  := \inf_{y \in S_oM} i_y>0,
\end{align*}
where $\gamma_y(t)$ is a (unit-speed) geodesic with $\dot{\gamma}_y(0)=y$.
%If $(M,F)$ is a  Cartan--Hadamard manifold, then $i_y=+\infty$ for any $y\in S_xM$.

Let $(r,y)$ be the {\it polar coordinate system} around some point $o\in M$. That is, if $x=(r,y)$, then $r=r(x)=d_F(o,x)$ and $y=y(x)\in S_oM$ such that
 \begin{equation}\label{geommeaingofr}
r\leq i_y,\quad\gamma_y(r)=(r,y),\quad \dot{\gamma}_y(r)=\nabla r|_{(r,y)}.
\end{equation}
Moreover,
it follows from  \cite[Lemma 3.2.3]{ShenLecture} that $F(\nabla r) = F^{*}({\dd}r) =1$ for  $\m$-a.e. $x\in M$.

In the polar coordinate system around $o$, the measure $\m$ decomposes as
\begin{equation*}%\label{volumeexprsso}
\dm|_{(r,y)}=:\hat{\sigma}_o(r,y)\,{\dd}r \wedge {\dd}\nu_o(y),
\end{equation*}
where ${\dd}\nu_o$ denotes the Riemannian volume form of the indicatrix $S_oM$. Consequently,
\begin{equation}\label{inffexprexx}
\int_M f \dm=\int_{S_oM} {\dd}\nu_o(y) \int_0^{i_y} f(r,y) \,\hat{\sigma}_o(r,y)\,{\dd}r,\quad \forall  f\in L^1(M,\m).
\end{equation}
For convenience,  we also introduce the
 {\it integral of distortion} (cf.~Huang--Krist\'aly--Zhao \cite{HKZ})
\begin{equation}\label{cocont}
\mathscr{I}_{\m}(o):=\int_{S_oM} e^{-\tau(o,y)}{\dd}\nu_o(y)<+\infty.
\end{equation}

To state the volume comparison result of Zhao--Shen \cite[Theorem 3.6]{ZS}, we first set
\[
 \mathfrak{s}_k(t):=\left\{
	\begin{array}{lll}
		\frac{\sin(\sqrt{k}t)}{\sqrt{k}}, && \text{ if }k>0,\\[3pt]
		\ \ \ \ t, && \text{ if }k=0,\\%[3pt]
		\frac{\sinh(\sqrt{-k }t)}{\sqrt{-k}}, && \text{ if }k<0.
	\end{array}
	\right.
\]
\begin{theorem}[\cite{ZS}]\label{bascivolurcompar} Let $(M,F,\m)$ be an $n$-dimensional $\FMMM$. If $\mathbf{Ric}\geq (n-1)k$ for some $k \in \mathbb{R}$, then for every $o\in M$ and $y\in  S_oM$, the function
\begin{equation} \label{densesim}
H_y(r):=\frac{\hat{\sigma}_o(r,y)}{ e^{-\tau\big(\gamma_y(r),\dot{\gamma}_y (r)\big)}  \mathfrak{s}_k^{n-1}(r)},\qquad \forall  r\in (0,i_y),
\end{equation}
is monotonically decreasing in r and satisfies $\lim_{r \to 0^+} H_y(r) = 1$,
where $(r,y)$ is the polar coordinate system around $o$ and $\gamma_y(t)$ is a geodesic with $\dot{\gamma}_y(0)=y$.
\end{theorem}

\begin{remark}\label{remaep} Since $\lim_{r \to 0^+} H_y(r) = 1$ and $S_oM$ is compact, there exists $\epsilon\in (0,\min\{1,\mathfrak{i}_o\})$ such that
\begin{equation}\label{limcaes}
2^{-1} e^{-\tau(o,y)}r^{n-1}\leq \hat{\sigma}_o(r,y)\leq 2 e^{-\tau(o,y)}r^{n-1}, \quad \forall r\in (0,\epsilon), \ y\in S_oM.
\end{equation}
%Because of \eqref{geommeaingofr}, for $\m$-a.e.\ $x\in M$, the gradient $\nabla r|_x$ is the velocity at $x$ of the minimal geodesic from $ o$ to $x$. %The same proof as in \cite[Theorem 3.4]{ZS} then shows that the conclusions of Theorem \ref{bascivolurcompar} remain true under the weaker curvature bound
%$\mathbf{Ric}(\nabla r)\geq (n-1) k.$
 \end{remark}

\subsection{Sobolev spaces induced by Finsler structure}

\begin{definition}\label{soboloevespace}Let $(M,F,\m)$ be  an  $\FMMM$.
For $u\in C^\infty_0(M)$ and $p \in [1,+\infty)$, define a pseudo-norm as
\begin{equation}\label{W1p}
\|u\|_{W^{1,p}_{\m}}:=\left(\int_M |u|^p \dm \right)^{1/p}+\left(\int_M F^{*p}({\dd} u)  \dm  \right)^{1/p}.
\end{equation}
The  {\it  Sobolev space} $W^{1,p}_0(M,F,\m)$ is defined as the closure of $C^\infty_0(M)$ with respect to the  backward topology induced by $\|\cdot\|_{W^{1,p}_{\m}}$, i.e.,
\[
u\in W^{1,p}_0(M,F,\m) \quad \Longleftrightarrow \quad \exists \, (u_k) \subset C^\infty_0(M) \text{ \ with }\lim_{k\rightarrow \infty}\|u-u_k\|_{W^{1,p}_{\m}}= 0.
\]
\end{definition}

\begin{remark}\label{sobonormprop}If a function $f$ is differentiable at $x\in M$, then  the definition of gradient and \eqref{F*F} imply
\begin{equation}\label{granotwoff}
F^*(x,{\dd}f)=F(x,\nabla f).
\end{equation}
Hence, $W^{1,p}_0(M,F,\m)$ admits an equivalent characterization in terms of  the gradient, i.e.,
\[
\|u\|_{W^{1,p}_{\m}}=\left(\int_M |u|^p \dm \right)^{1/p}+\left(\int_M F^{p}(\nabla u)  \dm  \right)^{1/p}.
\]
Consequently,  $W^{1,p}_0(M,F,\m)$ is the standard Sobolev space when $F$ is Riemannian.
In the general Finsler setting,
however, since the gradient is nonlinear ($\nabla(u_1+u_2)\neq \nabla u_1+ \nabla u_2$), it is therefore
 more convenient to define Sobolev spaces via the differential.
The triangle inequality for the co-metric $F^*$ implies that for any $u,v\in  W^{1,p}_0(M,F,\m)$:
 \begin{enumerate}[\rm (i)]
\item $\|u\|_{W^{1,p}_{\m}}\geq 0$ with equality if and only if $u=0$;

\item $\|\lambda \,u\|_{W^{1,p}_{\m}}= \lambda \|u\|_{W^{1,p}_{\m}}$ for any $\lambda\geq 0$;

\item $\|u+v\|_{W^{1,p}_{\m}}\leq \|u\|_{W^{1,p}_{\m}}+\|v\|_{W^{1,p}_{\m}}$.

\end{enumerate}
\end{remark}

\begin{remark}\label{sobondistance}
The pseudo-norm $\|\cdot\|_{W^{1,p}_{\m}}$ induces a ``metric" on $W^{1,p}_0(M,F,\m)$ by
$\mathsf{D}_p(u,v):=\|v-u\|_{W^{1,p}_{\m}}$.
Since $\mathsf{D}_p$ could be not symmetric, one defines  {\it forward} and  {\it backward balls} in the same way as \eqref{definball}.
The {\it forward topology} $\mathcal{T}_+$ (resp.,\ {\it backward topology} $\mathcal{T}_-$) is the topology
generated by forward balls (resp.,\ backward balls).
Consequently, a sequence $(u_k)\subset W^{1,p}_0(M,F,\m)$ is convergent to $u\in W^{1,p}_0(M,F,\m)$ under $\mathcal{T}_+$ (resp.,\ $\mathcal{T}_-$) if $ \mathsf{D}_p(u,u_k)\rightarrow 0$ (resp.,\ $\mathsf{D}_p(u_k,u)\rightarrow 0$).
%\[
%\lim_{k\rightarrow \infty} \mathsf{D}_p(u,u_k)=0  \qquad  \left(\text{resp.,} \lim_{k\rightarrow \infty} \mathsf{D}_p(u_k,u)=0 \right).
%\]
In particular, $W^{1,p}_0(M,F,\m)$  is constructed by the backward topology $\mathcal{T}_-$.
\end{remark}

\begin{definition}
Given a set $U\subset M$, a function $f:U\rightarrow \mathbb{R}$ is called {\it Lipschitz} if there exists a constant $C\geq 0$ such that
\[
|f(x_1)-f(x_2)|\leq  C\,d_F(x_1,x_2),\quad \forall x_1,x_2\in U.
\]
The minimal $C$ satisfying the above inequality is called the {\it dilatation} (or {\it Lipschitz constant}) of $f$, denoted by $\dil(f)$. %And $f:U\rightarrow \mathbb{R}$ is called {\it locally Lipschitz} if for every compact set $\mathscr{K}\subset U$, the restriction $f|_\mathscr{K}$ is Lipschitz.
\end{definition}

%\begin{remark}
%The usual definition of locally Lipschitz function is that every point has a neighbourhood on which the restriction  of the  function is Lipschitz. Since every manifold is locally compact, these two definitions are equivalent. For nonreversible Finsler manifolds, one can define other kinds of Lipschitz functions, see Krist\'aly--Ohta--Zhao \cite{KOZ} for more details.
%\end{remark}

Let $\Lip_0(U):=C_0(U)\cap \Lip(U)$ denote the collection of compactly supported Lipschitz functions on $U$.
The following approximation  result is standard; a proof can be found in Krist\'aly--Li--Zhao \cite[Section 3]{KLZ}.

%\begin{lemma}[\cite{KLZ}]\label{localtogloaLip}For every nonempty open set $U\subset M$, there holds $C^\infty_0(U)\subset \Lip_0(U)$.
%\end{lemma}

\begin{lemma}\label{lipsconverppax}
Let $(M,F,\m)$ be an $\FMMM$. For every $u\in \Lip_0(M)$,   there exists a sequence of smooth Lipschitz functions $(u_k)\subset C^\infty_0(M)\cap \Lip_0(M)$  such that
\begin{align*}
&\supp u\cup \supp u_k\subset \mathscr{K},\quad |u_k|\leq C,\quad F^*({\dd} u_k)\leq   \dil(u_k)\leq C,\\
&\quad \sup_{ M}|u_k- u|\rightarrow 0,\qquad \|u_k-u \|_{W^{1,p}_{\m}}+\|u-u_k \|_{W^{1,p}_{\m}}\rightarrow 0,
\end{align*}
where $\mathscr{K}\subset M$ is a compact set  and $C>0$ is  a constant. In particular, if $u$ is nonnegative, so are $u_k$'s.
\end{lemma}
\begin{remark} The last convergence in Lemma \ref{lipsconverppax} is equivalent to \(\|u_k-u\|_{W^{1,p}_{\m}}\to0\). Indeed,
\[
\|u_k-u\|_{W^{1,p}_{\m}}\leq
\|u_k-u\|_{W^{1,p}_{\m}}+\|u-u_k\|_{W^{1,p}_{\m}}\leq
2\lambda_F(\mathscr{K})\,\|u_k-u\|_{W^{1,p}_{\m}}.
\]
\end{remark}

%Since the distance function $d_F$ is locally equivalent to an Euclidean distance (see~\cite[(6.2.3)]{BCS}), Rademacher's theorem remains valid in the Finsler setting, i.e., every Lipschitz function is differentiable almost everywhere. Hence, we may introduce the following definition.

The local equivalence of $d_F$ to the Euclidean metric \cite[(6.2.3)]{BCS} guarantees the validity of Rademacher's theorem in Finsler geometry; thus every Lipschitz function is differentiable almost everywhere, which motivates the next definition.

\begin{definition}\label{function11}
Let  $(M,F,\m)$ be an $n$-dimensional $\FMMM$. Fix a point $o\in M$ and let $r(x):=d_F(o,x)$ denote the distance function from $o$.
Define three functionals on $\Lip_0(M)\backslash\{0\}$ as follows:
\begin{enumerate}[{\rm (i)}]
\item  for  $p\in (1, n)$, set
\[
\mathscr{H}_{o,p}(f):=\frac{ \int_M {F}^{*p}({\dd}f) \dm}{ \int_M {|f|^p}{r^{-p}}\dm};
\]

\item for  $1-p<s\leq 1<p<n$, set
\[
\mathscr{U}_{o,p,s}(f):=\frac{\Big(  \int_{M} F^{*p}({\rm d} f) \dm \Big)^{1/p}\left(  \int_{M}|f|^p r^{p's} \dm \right)^{1/{p'}} }{  \int_{M} |f|^p r^{s-1} \dm};
\]

\item  for   $1<p<m$ and $p(n+s-1)>m(n-p)>0$, set
\[
\mathscr{C}_{o,p,m,s}(f):= \frac{\Big( \int_{{M}}F^{*p}({\rm d} f) {\dm}   \Big)^{1/p}\left( \int_{{M}}  {|f|^{p'(m-1)}}{r^{p's}}{\dm}   \right)^{1/p'}  }{ \int_{{M}}  {|f|^{m}}{r^{s-1}}{\dm}  }.
\]
\end{enumerate}
\end{definition}

\begin{lemma}\label{funfiln}
Let $(M,F,\m)$ be an $n$-dimensional $\FMMM$  and let $\mathbf{J}$ denote any of  the functionals  $\mathscr{H}_{o,p}$, $\mathscr{U}_{o,p,s}$ or $\mathscr{C}_{o,p,m,s}$  from Definition \ref{function11}. Then $\mathbf{J}$ is well-defined on $\Lip_0(M)\backslash\{0\}$, and
\begin{align}\label{samecontrl}
\inf_{f\in C^\infty_0(M)\backslash\{0\}}\mathbf{J}(f)=\inf_{f\in \Lip_0(M)\backslash\{0\}}\mathbf{J}(f).
\end{align}
\end{lemma}
\begin{proof}
Given $f\in \Lip_0(M) \backslash\{0\}$, a direct calculation combined with \eqref{dualff*} yields
\begin{align*}
\limsup_{z\rightarrow x}\frac{|f(z)-f(x )|}{d_F(x ,z)}=\sup_{y\in S_xM}\left(\lim_{t\rightarrow 0^+}\frac{|f(\gamma_y(t))-f(\gamma_y(0))|}{t}\right)=\sup_{y\in S_xM}|\langle {\dd}f, y\rangle|=\max\{ F^*(x,\pm {\dd}f) \},
\end{align*}
where $\gamma_y(t)$ denotes a unit-speed geodesic from $x$ with $\dot{\gamma}_y(0)=y$.
Hence,
\[
F^*(x,{\dd}f)\leq \limsup_{z\rightarrow x}\frac{|f(z)-f(x )|}{d_F(x ,z)}\leq \dil(f), \qquad \text{for $\m$-a.e. $x\in M$},
\]
which together with the compactness of $\supp(f)$ yields  $\int_M {F}^{*p}({\dd}f) \dm<+\infty$. Choose $R>1$ such that $\supp(f)\subset B^+_R(o)$. Note that $\overline{B^+_R(o)}$ is compact since $(M,F)$ is forward complete. Choose a small $\epsilon>0$ as in Remark \ref{remaep}. Then, by  \eqref{inffexprexx}--\eqref{limcaes} we have
\begin{align}
0<\int_M {|f|^p}{r^{-p}}\dm&\leq \max|f| \int_{B^+_R(o)} r^{-p}\dm\leq \max|f| \left( \int_{B^+_{\epsilon}(o)}+\int_{B^+_R(o)\backslash B^+_{\epsilon}(o)} r^{-p}\dm   \right)\notag\\
&\leq  \max|f| \left( 2 \mathscr{I}_{\m}(o) \int^\epsilon_0 r^{n-p-1}{\dd}r  +\epsilon^{-p}\m\left[B^+_R(o)\right]  \right)<+\infty.\label{howtcollr}
\end{align}
Thus, $\mathscr{H}_{o,p}$ is well-defined; the other two functionals are treated similarly. Equality \eqref{samecontrl} follows from Lemma \ref{lipsconverppax} and an estimate analogous  to \eqref{howtcollr}.
\end{proof}

\section{ Berwald's  metric space}\label{sectionBerwald}
This section investigates the Finsler manifold introduced by Berwald \cite{Be1,Be2}, which arose from Hilbert's fourth problem. Subsection \ref{subsectionBasicBerwald} presents its basic geometric properties, while Subsection \ref{Berwald_subsec2} establishes the nonlinearity of its Sobolev spaces and the failure of several classical functional inequalities. These results collectively provide the proofs of Theorems \ref{Berwmainth} and \ref{sbeer}. All arguments in this and the following section proceed by direct computation, which yields stronger and more explicit results than would be possible by invoking advanced tools such as volume comparison theorems.
%All proofs in this and next sections are carried out by direct computation, avoiding more advanced tools such as volume comparison theorems.

\subsection{Background of Berwald's  metric space}\label{subsectionBasicBerwald}

\begin{definition}[\cite{Be1, Be2}] Let $\mathbb{B}^n_1(\mathbf{0})$ be the Euclidean unit ball centered at $\mathbf{0}$ in $\mathbb{R}^n$.  {\it  Berwald's metric} is defined by
\begin{equation}\label{Berwaldmetric}
\B(x,y):= \frac{ (\sqrt{ (1-|x|^2) |y|^2 + \langle x, y \rangle^2} +\langle x, y \rangle )^2}{ (1-|x|^2)^2 \sqrt{ (1-|x|^2) |y|^2 + \langle x, y \rangle^2} }, \qquad \forall (x,y)\in T\mathbb{B}^n_1(\mathbf{0}).
\end{equation}
The pair $(\mathbb{B}^n_1(\mathbf{0}),\B)$ is called  {\it Berwald's metric space}.
\end{definition}

\begin{remark}
According to \cite{BCS,ShenLecture}, a Finsler metric  is said to be {\it of  Berwald type} if the Chern connection can be viewed as a linear connection on the underlying manifold; such metrics form an intermediate class between Riemannian and general Finsler structures.
The metric $\B$ defined in \eqref{Berwaldmetric}, however, is not of Berwald type. In what follows, the term ``Berwald's metric" will always refer to the metric given by \eqref{Berwaldmetric}.
\end{remark}

The manifold $(\mathbb{B}^n_1(\mathbf{0}),{\B})$ is {\it projectively flat}: the image of every geodesic is a straight line (cf.~\cite{ShenSpray,Sh1}). Its geodesic coefficients are given by $G^i(x,y) = P(x,y) y^i$ with
\begin{equation}\label{BerwaldP}
P(x,y) = \frac{ \sqrt{ (1-|x|^2) |y|^2 + \langle x, y \rangle^2} +\langle x, y \rangle }{ 1-|x|^2 }.
\end{equation}

\begin{theorem} \label{Berwalddist}
Let $(\mathbb{B}^n_1(\mathbf{0}), \B, \mathscr{L}^n)$  be Berwald's  metric space endowed with the Lebesgue measure. The following statements are true:
\begin{enumerate}[\rm (i)]
\item\label{Berwaddist} the distance functions ${d}_{\B}(\mathbf{0},x)$ and ${d}_{\B}(x,\mathbf{0})$ are given by
\begin{equation}\label{Berwaldd0xdx0}
d_{\B} (\mathbf{0},x) = \frac{|x|}{1 - |x|}, \qquad   d_{\B}(x,\mathbf{0}) =  \frac{|x|}{1 + |x|};
\end{equation}

\item\label{Berwald_ii} the Finsler manifold $(\mathbb{B}^n_1(\mathbf{0}), \B)$ is forward complete (but not backward complete) with
 \[ \lambda_{\B}(x)= \left(  \frac{1+ |x|}{1-|x|}  \right)^2, \qquad
  \lambda_{\B}\left(\mathbb{B}^n_1(\mathbf{0})\right)=+\infty; \]

  \item\label{Berwald_iv}
the gradient of the distance function $ r(x): = d_{\B} (\mathbf{0},x)$  is given by
\begin{equation}\label{LemmaBerwaldGradr}
\nabla r(x) =  \frac{(1-|x|)^2}{|x|} x;
\end{equation}

\item\label{Berwald_iii}  $(\mathbb{B}^n_1(\mathbf{0}), \B, \mathscr{L}^n)$ is a Cartan--Hadamard manifold with  $\mathbf{K}=0$ and $\mathbf{S}\in (0,2(n+1))$. In particular,
\begin{equation}\label{minSber}
\mathbf{S}(x,\nabla r)=\frac{n+1}{1+r(x)}.
\end{equation}

\end{enumerate}
\end{theorem}
\begin{proof}
\textbf{(\ref{Berwaddist})}
The projective flatness implies that the geodesic starting from $\mathbf{0}$ to $x$  has the same trajectory as the radial segment  $\gamma(t) = t x$, $t \in [0,1]$.
Hence
\begin{align*}
{d}_{\B}(\mathbf{0},x)& = \int^1_0 \B(\gamma(t), \dot{\gamma}(t)) {\dd}t = \int^1_0 \frac{|x|}{(1- t |x|)^2} {\dd}t  =  \frac{|x|}{1 - |x|},\\
{d}_{\B}(x,\mathbf{0}) &= \int^1_0 \B(\gamma(t), - \dot{\gamma}(t)) {\dd}t = \int^1_0 \frac{|x|}{(1+ t |x|)^2} {\dd}t  =  \frac{|x|}{1 + |x|}.
\end{align*}

\textbf{(\ref{Berwald_ii})}
 By the definition of reversibility \eqref{defrever}, we have
\begin{equation*}
\lambda_{\B}(x)
  =\sup_{y \in T_x \mathbb{B}^n_1(\mathbf{0})  \setminus \{0\} }\left( \frac{ \sqrt{ (1-|x|^2) |y|^2 + \langle x, y \rangle^2} +\langle x, y \rangle}{ \sqrt{ (1-|x|^2) |y|^2 + \langle x, y \rangle^2} - \langle x, y \rangle }\right)^2
  = \left(  \frac{1+ |x|}{1-|x|}  \right)^2,
\end{equation*}
which yields $\lambda_{\B}(\mathbb{B}^n_1(\mathbf{0}) )=+\infty$.

Let $c(t)$, $t\in [0,1)$ be a unit-speed geodesic. The projective flatness combined with  the triangle inequality and \eqref{Berwaldd0xdx0}  yields
\begin{equation}\label{gest}
t=d_{\B}(c(0),c(t))\geq d_{\B}(\mathbf{0},c(t))- d_{\B}(\mathbf{0},c(0))=\frac{|c(t)|}{1 - |c(t)|}-\frac{|c(0)|}{1 - |c(0)|},
\end{equation}
which implies $\limsup_{t\rightarrow 1^-} |c(t)|<1$.
Set $c(1):=\lim_{t\rightarrow 1^-}c(t)$; the limit exists because $c(t)$ is a ray. The  geodesic ODE then implies that  $c(t)$ can be extended to an interval $[0,1+\varepsilon)$ for some $\varepsilon>0$ with $\lim_{t\rightarrow (1+\varepsilon)^-}|c(t)|<1$.
Inductively we can extend $c(t)$ along the direction $\dot{c}(0)$ until $|c(t)|\rightarrow 1$,
which by the estimate above in \eqref{gest} occurs only as $t \rightarrow +\infty$. Thus $(\Omega,F)$ is forward complete.
%which together with \eqref{gest} means that $c(t)$  is  defined on $[0,+\infty)$.

On the other hand, \eqref{Berwaldd0xdx0} implies  ${d}_{\B}(x,\mathbf{0})\rightarrow 1/2$ as $|x| \rightarrow 1$.
Consequently, the backward bounded closed ball $\overline{B^-_{1/2}(\mathbf{0})}=\mathbb{B}^n_1(\mathbf{0})$ is not compact,
so  backward completeness fails.
%Then $(\mathbb{B}^n_1(\mathbf{0}), \B)$  cannot be backward complete because the backward bounded closed ball $\overline{B^-_{1/2}(\mathbf{0})}=\mathbb{B}^n_1(\mathbf{0})$ is not compact.

\textbf{(\ref{Berwald_iv})}
 Let $(r,y)$ denote  the polar coordinate system of $(\mathbb{B}^n_1(\mathbf{0}),\B)$ around $\mathbf{0}$.
For $x\in \mathbb{B}^n_1(\mathbf{0})\backslash \{ \mathbf{0}\}$, write $x = (r,y)$ and let $\gamma_y(t)$ be the unit-speed  geodesic with $\dot{\gamma}_y(0)=y\in S_{\mathbf{0}}\Omega$. Then $\gamma_y(r)=x$, and $x$ is parallel to $y$ since $F$ is projectively flat.
Hence, we may write $x^i = h(r,y) y^i$ for some nonnegative function $h(r,y)$.
Substituting this into  \eqref{Berwaldd0xdx0} yields
\[
 r  = \frac{|x|}{1- |x|}=\frac{h(r,y) |y|}{1- h(r,y) |y|},
 \]
 which implies
 \[
 h(r,y) = \frac{r}{(1 +r) |y|}, \quad x^i=\frac{r}{(1 +r) |y|} y^i, \quad \gamma_y(t)=\frac{t y}{(1 + t)  |y|}.
 \]
Using \eqref{geommeaingofr} in combination with the relations above, we obtain
\[
\nabla r(x) = \dot{\gamma}_y (r)=\left.\frac{{\dd}\gamma_y^i}{{\dd}t}\right|_{t=r} \frac{\partial}{\partial x^i} = \frac{y^i}{(1 +r)^2 |y|}    \frac{\partial}{\partial x^i}= \frac{x^i}{r(1+r)}    \frac{\partial}{\partial x^i} = \frac{(1-|x|)^2}{|x|} x.
\]

\textbf{(\ref{Berwald_iii})}
 The flatness (i.e., $\mathbf{K}=0$) follows from  \cite{Be2,Sh1}. Using \eqref{Scurvature},  \eqref{BerwaldP}, and \eqref{LemmaBerwaldGradr}, we have
\begin{align*}
\frac{\mathbf{S}(x,y)}{\B(x,y)}= (n+1)\frac{P(x,y)}{\B(x,y)}  =(n+1)\sqrt{ (1-|x|^2)+ \frac{\langle x, y \rangle^2}{|y|^2}} \left( \sqrt{ (1-|x|^2)+ \frac{\langle x, y \rangle^2}{|y|^2}} -  \frac{\langle x, y \rangle}{|y|} \right),
\end{align*}
which implies $\mathbf{S}\in (0,2(n+1))$ and \eqref{minSber}.
\end{proof}

Now we study the co-metric $\B^*$ of $\B$. Given the complexity of the formula for $\B$, it seems impossible to express $\B^*$  in an explicit way. Instead, we investigate this problem from the perspective of  $(\alpha,\beta)$-metric. Recall that
a Finsler metric $F$ is called an {\it $(\alpha,\beta)$-metric}  if there exists a smooth function $\phi(s)$ such that
$F=\alpha \,\phi ( \beta/\alpha )$,
where $\alpha = \sqrt{a_{ij}(x)y^i y^j}$ is a Riemannian metric and $\beta$ is a $1$-form.
Such metrics, introduced by Matsumoto \cite{Ma}, form a rich and widely studied family in Finsler geometry.
 Berwald's metric $\B$ is a special $(\alpha, \beta)$-metric:
\begin{equation} \label{Berwaldphi}
\B = \alpha (1+s)^2, \qquad s=\frac{\beta}{\alpha},
\end{equation}
with
\begin{equation}\label{exparealphbeata}
\alpha = \alpha(x,y)= \frac{\sqrt{ (1-|x|^2) |y|^2 + \langle x, y \rangle^2}}{(1-|x|^2)^2}, \qquad \beta =\beta(x,y)= \frac{\langle x, y \rangle}{(1-|x|^2)^2}.
\end{equation}
%A metric expressed as in the form \eqref{Berwaldphi} is called a {\it squared metric}, where for general   Thus, Berwald's metric is a special squared metric.

Following the approach of Masca--Sabau--Shimada \cite{Masca}, we derive a quartic equation satisfied by $\B^{*}$. The proof is given in Appendix \ref{axiberwmetric}.
\begin{lemma}\label{lemmaBerwaldquartic}
The co-metric $\B^{*} = \B^{*}(x, \xi)$ of Berwald's metric $\B$ satisfies  the quartic equation
\begin{equation}\label{Berwaldquartic}
\begin{split}
& 16 |x|^2 (1-|x|^2)^2\Big\{ (1-|x|^2) \alpha^{*2} + \beta^{*2} \Big\}{\B}^{*4}
+ 8 \Big\{ ( 10 |x|^2 -1)(1-|x|^2) \alpha^{*2}\beta^{*} + (9 |x|^2 -1)\beta^{*3} \Big\} {\B}^{*3}
\\
&+ \Big\{  (1 -20 |x|^2 -8 |x|^4) \alpha^{*4} +6 (6|x|^2 -5)\alpha^{*2} \beta^{*2}-27 \beta^{*4} \Big\} {\B}^{*2}
+ 12 \alpha^{*4} \beta^{*} {\B}^{*} - \alpha^{*6} =0,
\end{split}
\end{equation}
where $\alpha^{*2} = \alpha^{*2}(x,\xi)=a^{ij} \xi_i \xi_j$, $\beta^{*} =\beta^*(x,\xi)= b^i \xi_i$ and
\begin{equation}\label{Berwalda_upij}
a^{ij} = (1 - |x|^2 )^3 (\delta^{ij} - x^i x^j), \qquad b^i = (1 - |x|^2)^2 x^i.
\end{equation}
\end{lemma}

\begin{example}\label{drisequalto1}
In this example, we  apply Lemma \ref{lemmaBerwaldquartic} to verify that the differential of the distance function has unit length, i.e.,
 \begin{equation} \label{BerwaldFdr}
{\B}^{*}(x,{\dd}r) = 1, \qquad \forall x\in \mathbb{B}^n_1(\mathbf{0})\backslash\{\mathbf{0} \},
\end{equation}
where $r(x) := d_{\B}(\mathbf{0},x)$.
In fact, owing to \eqref{Berwaldd0xdx0}, we get
\begin{equation}\label{drequ}
 {\dd}r = \sum_i^n \frac{x^i}{(1 - |x|)^2 |x|}  {\dd}x^i, \qquad \forall x\in \mathbb{B}^n_1(\mathbf{0})\backslash\{\mathbf{0} \}.
 \end{equation}
By the definition  of co-metric  \eqref{coFinsler}, we have
\begin{align}
 \B^*(x, {\dd}r)=\sup_{y\in T_x\mathbb{B}^n_1(\mathbf{0})\backslash\{0\}}\frac{\langle {\dd}r,y\rangle}{\B(x,y)}
 =(1+|x|)^2\sup_{|y|=1}\frac{\langle x,y\rangle}{|x|}\frac{ \sqrt{ (1-|x|^2)  + \langle x, y \rangle^2}   }{\left[\sqrt{ (1-|x|^2)   + \langle x, y \rangle^2} +\langle x, y \rangle \right]^2}.\label{Bdrcontroll}
\end{align}
Thus, it suffices to consider the case when $\langle x, y\rangle\geq 0$ and $|y|=1$.
By setting $\theta:=\angle x\mathbf{0}y$ (the angle subtended by $x$ and $y$), we have
\begin{align*}
\frac{\langle x,y\rangle}{|x|}\frac{ \sqrt{ (1-|x|^2)  + \langle x, y \rangle^2}   }{\left[\sqrt{ (1-|x|^2)   + \langle x, y \rangle^2} +\langle x, y \rangle \right]^2}=\frac{\cos \theta}{\sqrt{1-|x|^2\sin^2\theta}}  \frac{1}{ \left[  1+\frac{|x|\cos\theta}{\sqrt{1-|x|^2\sin^2\theta}}  \right]^2   }  \leq 1.
\end{align*}
Hence, \eqref{Bdrcontroll} yields $ \B^*(x, {\dd}r)\leq 2$ for $x\in\mathbb{B}^n_1(\mathbf{0})\backslash\{\mathbf{0} \}$.
 On the other hand, \eqref{Berwalda_upij} and \eqref{drequ} imply
\begin{align}\label{expressionalbeta}
\alpha^{*2}(x,  {\dd}r)
  =  (1+|x|)^4, \qquad
\beta^{*}(x, {\dd}r) =  |x|(1+ |x|)^2.
\end{align}
Substituting the above expressions into \eqref{Berwaldquartic}$\div (1+|x|)^6 $ yields
\[
({\B}^*(x,{\dd}r)-1)\Big\{ ({\B}^*(x,{\dd}r) (1 - |x|)^2 + (1 + |x|)^2 \Big\} \Big\{ 4 |x| {\B}^*(x,{\dd}r) -(1 + |x|)^2 \Big\}^2 =0,
\]
 which can be simplified into
\begin{equation}\label{BerwaldDualeq1}
\left({\B}^*(x,{\dd}r)-1 \right) \Big\{ 4 |x| {\B}^*(x,{\dd}r) -(1 + |x|)^2 \Big\} =0.
 \end{equation}
Since  ${\B}^*(x,{\dd}r)\leq 2$, by letting $|x|\rightarrow 0$ in \eqref{BerwaldDualeq1} we obtain ${\B}^*(x,{\dd}r)=1$ for small $|x|$.

Suppose that the second factor of \eqref{BerwaldDualeq1} will vanish in $\mathbb{B}^n_1(\mathbf{0})$. Let $x_0$ be such a point with the smallest Euclidean norm. Then     ${\B}^*(x,{\dd}r)=1$ for any $0<|x|<|x_0|$. By the continuity, we have two equations:
 \[
 {\B}^*(x_0,{\dd}r) = 1,\qquad 4 |x_0| {\B}^*(x_0,{\dd}r) -(1 + |x_0|)^2=0,
 \]
which gives $|x_0|=1$, contradicting $x_0\in \mathbb{B}^n_1(\mathbf{0})$.
Consequently, the first factor in \eqref{BerwaldDualeq1} vanishes identically and hence \eqref{BerwaldFdr} follows.
\end{example}

%\begin{remark}For a general Finsler manifold, the length of the differential of distance There is
%another approach to show \eqref{BerwaldFdr}. Let
%$t:=\sqrt{\frac{\cos\theta}{\sqrt{1-|x|^2\sin^2\theta}}}\in [0,1]$. Then \eqref{Bdrcontroll} becomes
%\[
%\B^*(x,{\dd}r)=(1+|x|)^2\sup_{\theta\in [0,\pi/2)}\frac{1}{(t^{-1}+|x|t)^2}= 1,
%\]
%with the extremal value point $t=1$, in which case $\theta=0$.
%\end{remark}

\subsection{Analytic properties of  Berwald's metric space}\label{Berwald_subsec2}

\begin{lemma}\label{BerwaldSobolevLemma1}
Let $(\mathbb{B}^n_1(\mathbf{0}),{\B},\mathscr{L}^n)$ be Berwald's  metric space endowed with the Lebesgue measure.
Define a function $w$ on $\mathbb{B}^n_1(\mathbf{0})$ by
\begin{equation}\label{distox}
w(x): = - \left[\ln\left(  \frac{2-|x|}{1 - |x|}\right)\right]^{-\frac{1}{n}}.
\end{equation}
  Then $w \in W^{1,p}_0(\mathbb{B}^n_1(\mathbf{0}),\B,\mathscr{L}^n)$ for every $p  \in [1, +\infty)$.
\end{lemma}
\begin{proof} The proof proceeds in two steps.

\smallskip

\textbf{Step 1.} We show $\|w\|_{W^{1,p}_{\mathscr{L}^n}}<+\infty$.

\smallskip
Let $x=(t,y)$ denote the Euclidean polar coordinates. Using $\ln(\frac{2-t}{1 - t}) \geq \ln 2$ for $t\in [0,1)$, we have
\begin{align}
\int_{\mathbb{B}^n_1(\mathbf{0})} |w|^p {\dd} x
& = \int_{\mathbb{B}^n_1(\mathbf{0})} \left[ \ln\left(\frac{2-|x|}{1 - |x|}\right) \right]^{-\frac{p}{n}} {\dd}x= n\omega_n \int_0^1 \left[ \ln\left(\frac{2-t}{1 - t}\right) \right]^{-\frac{p}{n}} t^{n-1} {\dd}t \notag\\
& \leq  n\omega_n \int_0^1 ( \ln 2 )^{-\frac{p}{n}} t^{n-1} {\dd}t =  \frac{ \omega_n}{ ( \ln 2 )^{\frac{p}{n}}} < +\infty.\label{ulpberwadl}
\end{align}
By \eqref{Berwaldd0xdx0}, note that $w(x)=- [\ln(2+r(x))]^{-\frac{1}{n}}$, where $r(x)=d_{\B}(\mathbf{0},x)=|x|/(1-|x|)$.  A direct calculation combined with \eqref{drequ} gives
\begin{align}
{\dd}w &= \frac{\partial w}{\partial r}{\dd}r=\frac{1-|x|}{n(2-|x|)} \left[\ln\left(1+ \frac{1}{1 - |x|}\right)\right]^{-\frac{1}{n}-1} \sum^n_{i =1}   \frac{x^i}{|x| (1-|x|)^2}  {\dd}x^i. \label{Berwalddu}
\end{align}
Together with \eqref{BerwaldFdr} this yields
\begin{equation*}
\begin{split}
\int_{\mathbb{B}^n_1(\mathbf{0})} {\B}^{*p}({\dd}w) {\dd}x
&= \frac{1}{n^p}\int_{\mathbb{B}^n_1(\mathbf{0})}  \left(\frac{1-|x|}{2-|x|}\right)^p \left[\ln\left(1+ \frac{1}{1 - |x|}\right)\right]^{-\frac{p}{n}-p} {\dd}x
\\
%&= \frac{\omega_n}{n^p}\int_0^1  \frac{(1-|x|)^p}{(2-|x|)^p} [\ln(1+ \frac{1}{1 - |x|})]^{-\frac{p}{n}-p} |x|^{n-1} {\dd} |x|
%\\
&= \frac{n\omega_n}{n^p}\int_0^1  \left(\frac{1-t}{2-t}\right)^p \left[\ln\left(1+ \frac{1}{1 - t}\right)\right]^{-\frac{p}{n}-p} t^{n-1} {\dd} t.
\end{split}
\end{equation*}
Since $\frac{1-t}{2-t} \leq \frac{1}{2}$ and $\ln(1+ \frac{1}{1 - t}) \geq \ln 2$ for $t \in [0, 1)$, we get
\begin{equation}\label{dulpberwad2}
\int_{\mathbb{B}^n_1(\mathbf{0})} {\B}^{*p}({\dd}w) {\dd}x
\leq  \frac{\omega_n}{2^p n^{p} }  (\ln 2)^{-\frac{p}{n}-p}< +\infty.
\end{equation}
Hence, $\|w\|_{W^{1,p}_{\mathscr{L}^n}}<+\infty$ follows by \eqref{ulpberwadl} and \eqref{dulpberwad2}.

\smallskip

\textbf{Step 2.} We prove $w \in W^{1,p}_0(\mathbb{B}^n_1(\mathbf{0}),\B,\mathscr{L}^n)$.

\smallskip

For every small $\varepsilon>0$, set
\[
w_\varepsilon:=
\begin{cases}
- [\ln(2+\varepsilon)]^{-\frac{1}{n}}, &\text{ if }r<\varepsilon,\\[4pt]
- [\ln(2+r)]^{-\frac{1}{n}}= w, &\text{ if }r\geq \varepsilon.
\end{cases}
\]
One checks directly that
\[
\lim_{\varepsilon\rightarrow 0^+}\|w-w_{\varepsilon}\|_{W^{1,p}_{\mathscr{L}^n}}=0.
\]
Now define $w_{ \varepsilon,\delta}:=\min\{0, w_{ \varepsilon} +\delta\}$ for $\delta\in (0,(\ln 2)^{-1/n})$. If we show
\begin{equation}\label{convertgelipp2}
w_{ \varepsilon,\delta}\in \Lip_0(\mathbb{B}^n_1(\mathbf{0})), \qquad \lim_{\delta\rightarrow 0^+}\|w_{\varepsilon}-w_{\varepsilon,\delta}\|_{W^{1,p}_{\mathscr{L}^n}}=0,
\end{equation}
then a diagonal argument together with Lemma \ref{lipsconverppax} implies that  $w$ can be approximated by smooth functions with compact supports;
consequently, $w\in W_0^{1,p}(\mathbb{B}^n_1(\mathbf{0}),\B,\mathscr{L}^n)$.

To verify \eqref{convertgelipp2}, note first that by projective flatness,
one can directly check $w_{ \varepsilon,\delta}\in \Lip_0(\mathbb{B}^n_1(\mathbf{0}))$ with
\begin{equation*}%\label{suppho}
\supp (w_{ \varepsilon,\delta})=\overline{\{w_{ \varepsilon,\delta}<0\}}=\overline{\{ r< \exp(\delta^{-n})-2\}}= \left\{ |x|\leq \frac{\exp(\delta^{-n})-2}{\exp(\delta^{-n})-1} \right\}.
\end{equation*}
Moreover,
a direct calculation yields
\begin{align*}%\label{deltaconvergezero1}
\lim_{\delta\rightarrow 0^+}\int_{\mathbb{B}^n_1(\mathbf{0})} |w_{\varepsilon}-w_{\varepsilon,\delta}|^p {\dd}x&=\lim_{\delta\rightarrow 0^+}\int_{\{r < \exp(\delta^{-n})-2\}}\delta^p {\dd}x +\lim_{\delta\rightarrow 0^+}\int_{\{r \geq \exp(\delta^{-n})-2\}}|w_{\varepsilon}|^p {\dd}x\notag\\
&\leq \lim_{\delta\rightarrow 0^+}\int_{\mathbb{B}^n_1(\mathbf{0})}\delta^p {\dd}x +\lim_{\delta\rightarrow 0^+}\int_{\left\{|x|\geq \frac{\exp(\delta^{-n})-2}{\exp(\delta^{-n})-1}\right\}}|w|^p {\dd}x=0.
\end{align*}
Similarly, since $ \mathsf{B}^{*}({\dd} w)\in L^p(\mathbb{B}^n_1(\mathbf{0}),\mathscr{L}^n)$, we obtain
\begin{align*}%\label{deltaconvergezero2}
\lim_{\delta\rightarrow 0^+}\int_{\mathbb{B}^n_1(\mathbf{0})} \mathsf{B}^{*p}({\dd} w_{\varepsilon}-{\dd} w_{\varepsilon,\delta})  {\dd}x \leq \lim_{\delta\rightarrow 0^+}\int_{\left\{|x|\geq \frac{\exp(\delta^{-n})-2}{\exp(\delta^{-n})-1}\right\}} \mathsf{B}^{*p}({\dd} w) {\dd}x= 0.
\end{align*}
Then \eqref{convertgelipp2} follows, which concludes the proof.
\end{proof}

Based on the above lemma, we have the following result.
\begin{theorem}\label{bersolnol}
Let $(\mathbb{B}^n_1(\mathbf{0}),{\B},\mathscr{L}^n)$ be Berwald's metric space endowed with the Lebesgue measure.
Then the Sobolev space $W^{1,p}_0(\mathbb{B}^n_1(\mathbf{0}),\mathscr{L}^n)$ is not a vector space for all $p\in (1,+\infty)$.
\end{theorem}
\begin{proof} Let $w$ be defined as in \eqref{distox}.
In view of Lemma \ref{BerwaldSobolevLemma1}, it suffices to show $\int_{\mathbb{B}^n_1(\mathbf{0})} {\B}^{*p}( -{\dd}w) {\dd}x = +\infty$.
By   \eqref{expressionalbeta} and \eqref{Berwalddu},  we have
\begin{align*}
\alpha^{*2}(x, - {\dd}w) = w_r^2 \alpha^{*2}(x, {\dd}r)
 = w_r^2 (1+|x|)^4, \quad
\beta^{*}(x, - {\dd}w) = - w_r \beta^{*}(x, {\dd}r) =- w_r  |x|(1+ |x|)^2,
\end{align*}
where
\begin{equation*}\label{urexpression}
 w_r:= \frac{\partial w}{\partial r}= \frac{1-|x|}{n(2-|x|)} \left[\ln\left(1+ \frac{1}{1 - |x|}\right)\right]^{-\frac{1}{n}-1}>0.
 \end{equation*}
Substituting the above expressions into  \eqref{Berwaldquartic}$\div [(1+|x|)^6 w_r^2] $ gives
\begin{align*}
 & 16 |x|^2 (1 - |x|)^2 \,{\B}^{*4}(x,-{\dd}w)  +  8|x| (|x|^4-10|x|^2+1)w_r \,{\B}^{*3}(x,-{\dd}w)
 \\
 &+ (1+|x|)^2 (|x|^4-50 |x|^2+1) w_r^2\, {\B}^{*2}(x,-{\dd}w) -12 |x| (1+|x|)^4  w_r^3 \,{\B}^{*}(x,-{\dd}w) -(1+|x|)^6 w_r^4 =0,
\end{align*}
which factors as
\begin{align*}
({\B}^{*}(x,-{\dd}w) + w_r) \Big\{ (1 - |x|)^2 \,{\B}^{*}(x,-{\dd}w) - (1 +|x|)^2 w_r\Big\} \Big\{ 4 |x|\, {\B}^{*}(x,-{\dd}w) + (1 +|x|)^2 w_r \Big\}^2 =0.
\end{align*}
Because ${\dd}w\neq0$, we have ${\B}^{*}(x,-{\dd}w)>0$; hence the  second factor must vanish. Consequently,
\begin{equation}\label{Bt_-du}
\B^{*}(x,-{\dd}w) = \frac{(1 +|x|)^2}{(1 - |x|)^2} w_r = \frac{(1 +|x|)^2}{n(1 - |x|)(2-|x|)} \left[\ln\left(1+ \frac{1}{1 - |x|}\right)\right]^{-\frac{1}{n}-1}.
\end{equation}
%where the last equality follows by \eqref{urexpression}.

Given $p>1$, there exists a small $\vartheta>0$ such that
\begin{equation}\label{palphscondion}
p-\vartheta p\left( \frac{1}{n}+1\right)>1.
\end{equation}
Moreover, choose a number $\delta=\delta(\vartheta)\in (0,1)$ such that for $ |x|\in (\delta,1)$,
\begin{equation*}
\left[\ln\left(1+ \frac{1}{1 - |x|}\right)\right]^{-\frac{1}{n}-1} > \left[ \frac{1}{1 - |x|} -1 \right]^{\vartheta(-\frac{1}{n}-1)}.
\end{equation*}
Combined with \eqref{Bt_-du}, these choices yield
\begin{align}\label{Bduinf}
\int_{\mathbb{B}^n_1(\mathbf{0})} {\B}^{*p}(x, -{\dd}w) {\dd}x&\geq  \frac{1}{2^pn^p} \int_{\mathbb{B}^n_1(\mathbf{0})}  \frac{(1 +|x|)^{2p}}{(1 - |x|)^p } \left[\ln\left(1+ \frac{1}{1 - |x|}\right)\right]^{-\frac{p}{n}-p} {\dd}x,\notag\\
& \geq \frac{1}{2^p n^p} \int_{\{\delta<|x|<1\}}  \frac{(1+|x|)^{2p}}{(1 - |x|)^p}\left[ \frac{1}{1 - |x|} -1 \right]^{- \vartheta p(\frac{1}{n}+1)} {\dd}x\notag\\
& = \frac{\omega_n}{2^p n^{p-1}} \int_\delta^1  \frac{(1+t)^{2p}t^{n-1-\vartheta p(\frac{1}{n}+1)}}{(1 - t)^{p-\vartheta p(\frac{1}{n}+1)}}  {\dd}t = + \infty,
\end{align}
where the last equality follows from \eqref{palphscondion}.
This completes the proof.
\end{proof}

The nonlinearity of $W^{1,p}_0(\mathbb{B}^n_1(\mathbf{0}),\B,\mathscr{L}^n)$ has its underlying reason in the incompatibility of the relevant topologies (see Remark \ref{sobondistance}).

\begin{corollary}\label{forwbacktopso} For each $p\in (1,+\infty)$, the  forward and backward topologies of
 $W^{1,p}_0(\mathbb{B}^n_1(\mathbf{0}),\B,\mathscr{L}^n)$ are incompatible. In particular, there exist  a sequence $(w_k)$ and  $w$ in $W^{1,p}_0(\mathbb{B}^n_1(\mathbf{0}),\B,\mathscr{L}^n)$  such that
 \[
 \lim_{k\rightarrow \infty}\|w-w_k\|_{W^{1,p}_{\mathscr{L}^n}}=0, \quad \|w_k-w\|_{W^{1,p}_{\mathscr{L}^n}}=+\infty, \ \forall k\in \mathbb{N}.
 \]
\end{corollary}
\begin{proof}
Let  $w$  be as in \eqref{distox}.
By Lemma \ref{BerwaldSobolevLemma1}, there exists a sequence $(w_k)\subset C^\infty_0(\mathbb{B}^n_{1}(\mathbf{0}))$ such that $\lim_{k\rightarrow \infty}\|w-w_k\|_{W^{1,p}_{\mathscr{L}^n}}=0$. It remains to show $\|w_k-w\|_{W^{1,p}_{\mathscr{L}^n}}=+\infty$ for all $k\in \mathbb{N}$.
Suppose by contradiction that  $\|w_{k_0}-w\|_{W^{1,p}_{\mathscr{L}^n}}<+\infty$ for some $k_0\in \mathbb{N}$. Then the triangle inequality of $\|\cdot\|_{W^{1,p}_{\mathscr{L}^n}}$ forces
\[
\|-w\|_{W^{1,p}_{\mathscr{L}^n}}=\|0-w\|_{W^{1,p}_{\mathscr{L}^n}}\leq \|0-w_{k_0}\|_{W^{1,p}_{\mathscr{L}^n}}+\|w_{k_0}-w\|_{W^{1,p}_{\mathscr{L}^n}} <+\infty,
\]
which contradicts the fact that  $\|-w\|_{W^{1,p}_{\mathscr{L}^n}}=+\infty$ by \eqref{Bduinf}. This concludes the proof.
\end{proof}

In the remainder of this subsection,
we demonstrate the failure of several classical functional inequalities on Berwald's metric space.
This phenomenon cannot be deduced from the results of Krist\'aly--Li--Zhao\cite{KLZ}, because here the
$S$-curvature does not admit a uniform positive lower bound (see \eqref{minSber}).

\begin{theorem}\label{MainThmBerw}
Let $(\mathbb{B}^n_1(\mathbf{0}),\B,\mathscr{L}^n)$ be Berwald's metric space endowed with the Lebesgue measure. Denote by  $r(x)=d_{\B}(\mathbf{0},x)$ the distance function from the origin.
Then the following statements are true:
\begin{enumerate}[{\rm (i)}]
\item\label{Berw4}
the $L^p$-Hardy inequality fails for every $p\in (1, n)$, i.e.,
		\[
		\inf_{f\in C^\infty_0({\mathbb{B}^n_1(\mathbf{0})})\backslash\{0\}}\frac{ \int_{\mathbb{B}^n_1(\mathbf{0})} {\B}^{*p}({\rm d} f) {\dd}x}{ \int_{\mathbb{B}^n_1(\mathbf{0})} {|f|^p}{r^{-p}}{\dd}x}=0;
		\]
		
\item\label{Berw5}
the  generalized uncertainty principle fails, i.e., for any $-p+1<s\leq 1<p<n$,
		\[
		\inf_{f\in C^\infty_0({\mathbb{B}^n_1(\mathbf{0})})\backslash\{0\}}\frac{\left(  \int_{{\mathbb{B}^n_1(\mathbf{0})}} {\B}^{*p}({\rm d} f){\dd}x \right)^{1/p}\left(   \int_{{\mathbb{B}^n_1(\mathbf{0})}}|f|^p r^{p's} {\dd}x \right)^{1/{p'}} }{  \int_{{\mathbb{B}^n_1(\mathbf{0})}} |f|^p r^{s-1} {\dd}x   }=0.
		\]
%where $p':=p/(p-1)$ is the H\"older conjugate number of $p$.
%({\color{blue} Heisenberg-Pauli-Weyl uncertainty principle for $s=1$;  Hydrogen uncertainty principle for $s = 0$ )}

\end{enumerate}

\end{theorem}
\begin{proof}  Given $\iota\in (0,1)$, define a function $u_{\iota}$ on $\mathbb{B}^n_1(\mathbf{0})$ by
\begin{equation}\label{ualphadefine}
u_{\iota}(x) := - e^{-\iota r(x)} = - \exp\left({\frac{-\iota |x|}{1-|x|}}\right).
\end{equation}
%A simlar argument as in Lemma \ref{BerwaldSobolevLemma1} forces $u_{\iota}\in W^{1,p}_0(\mathbb{B}^n_1(\mathbf{0}),\B,\mathscr{L}^n)$. Moreover, s
Since $\B^*({\dd}r)=1$, we get
\begin{align}
\int_{\mathbb{B}^n_1(\mathbf{0})} {\B}^{*p}({\dd}u_{\iota}) {\dd}x &= \iota^p \int_{\mathbb{B}^n_1(\mathbf{0})}|u_\iota|^p {\dd}x
 =  \iota^p \int_{\mathbb{B}^n_1(\mathbf{0})} e^{\frac{-\iota p |x|}{1-|x|}} {\dd}x\notag\\
&= \iota^p n\omega_n \int_0^1 e^{\frac{-\iota p t}{1-t}} t^{n-1} {\dd}t
\leq  \iota^p n\omega_n \int_0^1  t^{n-1} {\dd}t = {\iota^p \omega_n}.\label{BerwaldFstarp}
\end{align}

\textbf{(\ref{Berw4})}
 Choose a small $\varepsilon_0\in (0,1)$ such that
$ e^{\frac{- t}{1-t}} >2^{-1}$ for $ t \in (0,\varepsilon_0)$.
A direct calculation then yields
\begin{align}\label{Berwaldstandintegr}
\int_{\mathbb{B}^n_1(\mathbf{0})} \frac{|u_{\iota}|^p}{r^p}{\dd}x &= \int_{\mathbb{B}^n_1(\mathbf{0})} \frac{(1-|x|)^p}{|x|^p}  e^{\frac{-\iota p |x|}{1-|x|}}{\dd}x=n\omega_n \int_0^1  (1-t)^p t^{n-1-p} e^{\frac{-\iota p t}{1-t}}  {\dd}t\notag\\
&\geq
 \frac{n\omega_n (1-\varepsilon_0)^p}{2^{\iota p}}  \int_0^{\varepsilon_0} t^{n-1-p} {\dd}t \geq \frac{n\omega_n (1-\varepsilon_0)^p \varepsilon_0^{n-p} }{2^{  p}(n-p)}.
\end{align}

On the other hand, for every small $\epsilon>0$, set
\[
u_{\iota, \epsilon}:=
\begin{cases}
-e^{-\iota\epsilon}, &\text{ if }  r < \epsilon,\\
-e^{-\iota r}, &\text{ if }   r \geq \epsilon,
\end{cases}
\qquad \tilde{u}_{\iota, \epsilon}:=\min\{0, u_{\iota, \epsilon} +\epsilon\}.
\]
Then $\tilde{u}_{\iota, \epsilon}\in  \Lip_0(\mathbb{B}^n_1(\mathbf{0}))$.
Now a direct calculation
yields
\begin{equation}\label{Hardyifuncitonal}
\lim_{\epsilon\rightarrow 0^+} \mathscr{H}_{\mathbf{0},p}(\tilde{u}_{\iota, \epsilon})=  \mathscr{H}_{\mathbf{0},p}(u_\iota)<+\infty,
\end{equation}
where $ \mathscr{H}_{\mathbf{0},p}$ is defined in Definition \ref{function11}.
Now Lemma \ref{funfiln} combined with \eqref{BerwaldFstarp}--\eqref{Hardyifuncitonal} implies
\begin{align*}
\inf_{f\in C^\infty_0(\mathbb{B}^n_1(\mathbf{0}))} \mathscr{H}_{\mathbf{0},p}(f)=\inf_{f\in \Lip_0(\mathbb{B}^n_1(\mathbf{0}))} \mathscr{H}_{\mathbf{0},p}(f)\leq \lim_{\iota\rightarrow 0^+}\mathscr{H}_{\mathbf{0},p}(u_\iota)
\leq  \lim_{\iota\rightarrow 0^+}\frac{2^{  p} (n-p) \iota^p  }{n (1-\varepsilon_0)^p \varepsilon_0^{n-p}}= 0.
\end{align*}
Hence, the Hardy inequality fails for $p\in (1,n)$.

\smallskip

\textbf{(\ref{Berw5})}
As before,  choose a small $\varepsilon_0\in (0,1)$ such that
$ e^{\frac{- t}{1-t}} >2^{-1}$ for $ t \in (0,\varepsilon_0)$. Since $1-s\geq0$ and $n+s-1>0$, we have
\begin{align}
\int_{\mathbb{B}^n_1(\mathbf{0})} |u_{\iota}|^pr^{s-1} {\dd}x&=n\omega_n\int^1_0 e^{-\iota p\left( \frac{t}{1-t} \right)}\left( \frac{t}{1-t} \right)^{s-1}t^{n-1}{\dd}t\notag\\
& \geq     \frac{n\omega_n}{2^{\iota p}}  \int^{\varepsilon_0}_0  (1-t)^{1-s} t^{n+s-2}  {\dd}t\geq  \frac{n\omega_n(1-\varepsilon_0)^{1-s}\varepsilon_0^{n+s-1}}{2^{ p}(n+s-1)}= : C_1(s,p,n,\varepsilon_0).
\label{BerwaldhARDdenome}
\end{align}

For the numerator factor containing $r^{p's}$ we write
 \begin{align}
   \int_{{\mathbb{B}^n_1(\mathbf{0})}}|u_\iota|^p r^{p's}{\dd}x =      n\omega_n \int^1_0 e^{-\frac{\iota p t}{1-t}}\left( \frac{t}{1-t} \right)^{p's}   t^{n-1}{\dd}t.\label{BerwaldhARDdenumer}
\end{align}
For  a small positive $\delta< 1$, since $p's+n>0$,
\begin{align}\label{BerwaldhARDdenumer1}
  \int^\delta_0 e^{-\frac{\iota p t}{1-t}}\left( \frac{t}{1-t} \right)^{p's}   t^{n-1}{\dd}t\leq  \max\left\{1,  (1-\delta)^{-p's}  \right\}\int^\delta_0 t^{p's+n-1} {\dd}t=:C_2(\delta,s,p,n).
\end{align}
For the remaining part, we have
\begin{align}\label{medpartloc}
\int^1_\delta e^{-\frac{\iota p t}{1-t}}\left( \frac{t}{1-t} \right)^{p's}   t^{n-1}{\dd}t
\leq  \max\left\{1,\delta^{p's+n-1}\right\}\int^1_\delta e^{-\frac{\iota p \delta}{1-t}}(1-t)^{-p's}{\dd}t.
\end{align}
\noindent\textbf{Case 1: $p's\leq 1$.} Then we have
\begin{align}\label{1detecontr1}
\int^1_\delta e^{-\frac{\iota p \delta}{1-t}}(1-t)^{-p's}{\dd}t=&\int^1_\delta  e^{-\frac{\iota p \delta}{1-t}}(1-t)^{2-p's}{\dd}\left(\frac{1}{1-t}\right)\leq(1-\delta)^{2-p's}\int^1_\delta  e^{-\frac{\iota p \delta}{1-t}}{\dd}\left(\frac{1}{1-t}\right)\notag\\
=&\frac{(1-\delta)^{2-p's}}{\iota p \delta }e^{-\frac{\iota p \delta}{1-\delta}}\leq\frac{(1-\delta)^{2-p's}}{\iota p \delta } =:  \iota^{-1} C_3(\delta,s,p) .
\end{align}
\noindent\textbf{Case 2: $p's>1$.} With the change of variables $h=1/(1-t)$ and $z=\iota p\delta h$, we obtain
\begin{align}\label{1detecontr2}
 \int^1_\delta e^{-\frac{\iota p \delta}{1-t}}(1-t)^{-p's}{\dd}t&=\int^{+\infty}_{\frac{1}{1-\delta}} e^{-\iota p \delta h} h^{p's-2}{\dd}h\leq \frac{1}{(\iota p \delta)^{p's-1}}\int^{+\infty}_0 e^{-z} z^{p's-2}{\dd}z\notag\\
 &=\frac{\Gamma(p's-1)}{(\iota p \delta)^{p's-1}}=:\iota^{1-p's} C_4(\delta,s,p),
\end{align}
where $\Gamma(\cdot)$ denotes the Gamma function.

From \eqref{BerwaldhARDdenumer}--\eqref{1detecontr1} we have
\[
\int_{{\mathbb{B}^n_1(\mathbf{0})}}|u_\iota|^p r^{p's}{\dd}x =
\begin{cases}
n\omega_n\left[C_2(\delta,s,p,n)+\iota^{-1} \max\left\{1,\delta^{p's+n-1}\right\} C_3(\delta,s,p) \right], & \text{ if }p's\leq 1,\\[4pt]
n\omega_n\left[ C_2(\delta,s,p,n)+\iota^{1-p's} \max\left\{1,\delta^{p's+n-1}\right\} C_4(\delta,s,p) \right], & \text{ if }p's> 1.
\end{cases}
\]
Combining this estimate with \eqref{BerwaldFstarp} and \eqref{BerwaldhARDdenome}, and using the fact that $p'-p's+1>0$, we get
$\lim_{\iota\rightarrow 0^+}\mathscr{U}_{\mathbf{0},p,s}(u_{\iota})=0$,
where
 $\mathscr{U}_{\mathbf{0},p,s}$ is the functional from  Definition \ref{function11}.
The same approximation argument as in the proof of \eqref{Berw4} then yields
\begin{equation*}
\inf_{f\in C^\infty_0(\mathbb{B}^n_1(\mathbf{0}))\backslash\{0\}}\mathscr{U}_{\mathbf{0},p,s} (f)\leq \liminf_{\iota \rightarrow 0^+} \mathscr{U}_{\mathbf{0},p,s} (u_{\iota})=0,
\end{equation*}
which establishes the failure of the generalized uncertainty principle.
\end{proof}

%\begin{proof}[Proof of Theorem \ref{Berwmainth}] Statement \eqref{CKNFIALBE1} follows from Theorem \ref{Berwalddist}/\eqref{Berwald_ii} since the reversibility of every Minkowski space is finite.
%Statement \eqref{CKNFIALBE2} is proved in Theorem \ref{bersolnol}, and Statements \eqref{CKNFIALBE4},\eqref{CKNFIALBE5} in Theorem \ref{MainThmBerw} together with \eqref{granotwoff}.
%\end{proof}

\begin{proof}[Proof of Theorem \ref{Berwmainth}]
Statement \eqref{CKNFIALBE1} follows from Theorem \ref{Berwalddist}/\eqref{Berwald_ii} because every Minkowski space has finite reversibility.
Statement \eqref{CKNFIALBE2} is proved in Theorem \ref{bersolnol}, while Statements \eqref{CKNFIALBE4} and \eqref{CKNFIALBE5} are established by Theorem \ref{MainThmBerw} together with \eqref{granotwoff}.
\end{proof}

In order to show Theorem \ref{sbeer}, we need the following result.
\begin{lemma}\label{lemma_function}
$\ds \lim_{a \rightarrow 0^+} \int^{1}_0 \frac{ t^n }{( a +  t)^{n+1}}  {\dd}t = +\infty.$
\end{lemma}
\begin{proof}
A direct computation yields
\begin{align*}
\int^{1}_0\frac{ t^n }{( a +  t)^{n+1}}{\dd}t &= \int^{1}_0\frac{ (a+t-a)^n }{( a +  t)^{n+1}} {\dd}t=  \int^{1}_0\left[ \frac{1}{a+t} + \sum^{n}_{k=1} \binom{n}{k} (-a)^{k} (a+t)^{-k-1}\right] {\dd}t\\
& = \ln\left(1+a^{-1}\right) - \sum^{n }_{k=1} \frac{1}{k}\binom{n}{k}\left[ (-a)^{k} (a+1)^{-k} - (-1)^k \right].
\end{align*}
Then the proof is completed as $a \rightarrow 0^+$.
\end{proof}

\begin{proof}[Proof of Theorem \ref{sbeer}]
\textbf{(\ref{BerwCKN_i})}
 Let $r(x):=d_{\B}(\mathbf{0},x)$ and let $(t,y)$ denote the Euclidean polar coordinate system.
The following relations will be used repeatedly:
\begin{align}
t  = |x| = \frac{r}{1+r}, \qquad
{\dd}x  =  t^{n-1} {\dd}t \,{\dd}\nu_{\mathbb{S}^{n-1}}(y)
   =\frac{1 }{(r+1)^2} \left(\frac{r}{r+1} \right)^{n-1}   {\dd}r \,{\dd}\nu_{\mathbb{S}^{n-1}}(y), \label{Berwald_dxdr}
\end{align}
where ${\dd}\nu_{\mathbb{S}^{n-1}}$ is the standard measure of Euclidean unit  sphere $\mathbb{S}^{n-1}$.
%For the convenience, define
%\begin{equation*}%\label{cknfunctional}
%\mathcal {G}_{p,m,s}(f):=\frac{\left( \int_{{\mathbb{B}^n_1(\mathbf{0})}}{\B}^{*p}({\rm d} f) {\dd}x   \right)^{1/p}\left( \int_{{\mathbb{B}^n_1(\mathbf{0})}}  {|f|^{p'(m-1)}}{r^{p's}}{\dd}x   \right)^{1/p'}  }{ \int_{{\mathbb{B}^n_1(\mathbf{0})}}  {|f|^{m}}{r^{s-1}}{\dd}x  }.
%\end{equation*}
%It is needed to analysis the three integrals in the above quotient.

Define a smooth function $v_{\iota}$ on $\mathbb{B}^n_1(\mathbf{0})$ by
\begin{equation*}
v_\iota : = - e^{- \iota r^{1+\frac{\mu}{p}}},
\end{equation*}
where constants $\iota$ and $\mu$ satisfy
\[ 0<\iota< \min\left\{ \frac{1}{p' (m-1)}, \frac{1}{m} \right\}, \qquad  \mu=\frac{3}{2}. \]
As in the proof of Theorem \ref{MainThmBerw}/\eqref{Berw4} we have
\begin{equation*}%\label{controll}
\inf_{f\in C^\infty_0(\mathbb{B}^n_1(\mathbf{0}))\backslash\{0\}}\mathscr{C}_{\mathbf{0},p,m,s}(f)\leq \mathscr{C}_{\mathbf{0},p,m,s}(v_{\iota}),
\end{equation*}
where $\mathscr{C}_{\mathbf{0},p,m,s}$ is defined in Definition \ref{function11}. In view of \eqref{granotwoff}, it suffices to show
\begin{equation}\label{limconcerb}
\lim_{\iota\rightarrow 0^+}\mathscr{C}_{\mathbf{0},p,m,s}(v_{\iota})=0.
\end{equation}

Firstly, by \eqref{Berwald_dxdr},
the denominator of $\mathscr{C}_{\mathbf{0},p,m,s}(v_{\iota})$ satisfies
\begin{align} \label{v_denom}
\int_{{\mathbb{B}^n_1(\mathbf{0})}}  {|v_{\iota}|^{m}}{r^{s-1}}{\dd}x
%&=n\omega_n \int^{+\infty}_0 e^{-\iota m r^{1+\frac{\mu}{p}} }   \frac{r^{s-1}}{(1+r)^2} \left(\frac{r}{1+r} \right)^{n-1} {\dd}r \notag \\
&\geq   n\omega_n \int^{1}_{\frac{1}{2}} e^{- m\iota r^{1+\frac{\mu}{p}} }   \frac{r^{s-1}}{(1+r)^2} \left(\frac{r}{1+r} \right)^{n-1} {\dd}r \geq   \frac{n\omega_n}{  3^{n-1} e}\min\left\{2^{-2},2^{-1-s}\right\}.
\end{align}
%It is clear that $\int_{{\mathbb{B}^n_1(\mathbf{0})}}  {|f|^{m}}{r^{s-1}}{\dd}x$ has positive lower bounded.

Secondly, since $\B^*({\dd}r)=1$, a direct calculation together with  \eqref{Berwald_dxdr} yields
\begin{align}
\int_{\mathbb{B}^n_1(\mathbf{0})} {\B}^{*p}({\dd}v_{\iota}) {\dd}x
& = \iota^p \left(1+\frac{\mu}{p}\right)^p  \int_{\mathbb{B}^n_1(\mathbf{0})} r^{\mu} e^{- \iota p r^{1+\frac{\mu}{p}}} {\dd}x \notag\\
& = \iota^p n\omega_n \left(1+\frac{\mu}{p}\right)^p \int_{0}^{+\infty} e^{- \iota p r^{1+\frac{\mu}{p}}}
\frac{r^{\mu}}{(r+1)^2} \left( \frac{r}{r+1} \right)^{n-1}  {\dd}r \notag\\
& \leq \iota^p n\omega_n  \left(1+\frac{\mu}{p}\right)^p \left(1  +     \int_{1}^{+\infty} e^{- \iota p r^{1+\frac{\mu}{p}}}r^{\mu-2}  {\dd}r\right). \label{v_numer_first0}
\end{align}
To analyze  $\int_{1}^{+\infty} e^{- \iota p r^{1+\frac{\mu}{p}}}
r^{\mu-2}  {\dd}r$, let $z:= \iota p r^{1+\frac{\mu}{p}}$.
Since $\frac{p(\mu-2) - \mu}{p+\mu} > -1$, we have
\begin{align*}
\int_{1}^{+\infty} e^{- \iota p r^{1+\frac{\mu}{p}}}
r^{\mu-2}  {\dd}r
 &\leq \iota^{\frac{p(1-\mu)}{p+\mu}} \frac{p^{\frac{p(2-\mu)+\mu}{p+\mu}} }{p+\mu} \int_{0}^{+\infty} e^{- z}
z^{\frac{p(\mu-2) - \mu}{p+\mu}}  {\dd}z\\
&=\iota^{\frac{p(1-\mu)}{p+\mu}} \frac{p^{\frac{p(2-\mu)+\mu}{p+\mu}} }{p+\mu}\Gamma\left(\frac{p(\mu-2) - \mu}{p+\mu}+1\right)=:\iota^{\frac{p(1-\mu)}{p+\mu}}C_1(p).
\end{align*}
 Then it follows from \eqref{v_numer_first0} that
\begin{align}\label{v_numer_first}
\left( \int_{\mathbb{B}^n_1(\mathbf{0})} {\B}^{*p}({\dd}v_{\iota}) {\dd}x \right)^{1/p}
\leq  \iota^{ \frac{p+1}{p+\mu}} \left( \iota^{\frac{p(\mu-1)}{p+\mu}} +  C_1(p) \right)^{1/p} (n \omega_n)^{1/p} \left(1+\frac{\mu}{p}\right).
\end{align}

Thirdly,  to estimate the second factor of the numerator of  $\mathscr{C}_{\mathbf{0},p,m,s}(v_{\iota})$, we obtain
\begin{align} \label{v_numer_second1}
 \int_{{\mathbb{B}^n_1(\mathbf{0})}}  {|v_\iota|^{p'(m-1)}}{r^{p's}} {\dd}x
 &=n\omega_n \int^{+\infty}_0 e^{-\iota p'(m-1) r^{1+\frac{\mu}{p}} }   \frac{r^{p's}}{(1+r)^2} \left(\frac{r}{1+r} \right)^{n-1} {\dd}r \notag \\
 &\leq n\omega_n \left(  \int^{1}_0 e^{-\iota p'(m-1) r^{1+\frac{\mu}{p}} } r^{p's+n-1}{\dd}r + \int_{1}^{+\infty}   e^{-\iota p'(m-1) r^{1+\frac{\mu}{p}} } r^{p's-2} {\dd}r  \right).
\end{align}
Since
\begin{align*}%\label{n_p's}
p's + n =\frac{p(n+s)-n}{p-1}>\frac{p+m(n-p)-n}{p-1}=\frac{(m-1)(n-p)}{p-1}>0,
\end{align*}
we have
\begin{align}\label{v_numer_second_0delta}
\int^{1}_0 e^{-\iota p'(m-1) r^{1+\frac{\mu}{p}} } r^{p's+n-1}{\dd}r \leq   \int^{1}_0   r^{p's+n-1}{\dd}r
= \frac{1}{p's+n}.
\end{align}
The remainder of the proof is divided into three cases:  $s\leq \frac{2}{p'}$, $\frac{2}{p'}<s<2$ and $s=2$.

\smallskip

{\bf Case 1: $s\leq \frac{2}{p'}$.} In this case,
\begin{align*}%\label{v_numer_second_1}
\int_{1}^{+\infty}   e^{-\iota p'(m-1) r^{1+\frac{\mu}{p}} } r^{p's-2} {\dd}r
\leq \int_{1}^{+\infty}   e^{-\iota p'(m-1) r} {\dd}r
\leq  \frac{\iota^{-1}}{p'(m-1)}.
 \end{align*}
Together with \eqref{v_numer_first}--\eqref{v_numer_second_0delta}, it follows that
\begin{align}\label{v_numer_1}
 & \left( \int_{\mathbb{B}^n_1(\mathbf{0})} {\B}^{*p}({\dd}v_{\iota}) {\dd}x \right)^{1/{p}} \left( \int_{{\mathbb{B}^n_1(\mathbf{0})}}  {|v_\iota|^{p'(m-1)}}{r^{p's}} {\dd}x \right)^{1/{p'}} \notag\\
 \leq \  & \iota^{\frac{p+1}{p+\mu}} \left( \frac{\iota^{-1}}{p'(m-1)} +  \frac{1}{p's+n}  \right)^{1/{p'}}  \left(\iota^{\frac{p(\mu-1)}{p+\mu}} + C_1(p) \right)^{1/{p}} n\omega_n \left(1+\frac{\mu}{p}\right).
\end{align}
The lowest power of $\iota$ on the right-hand side of the above inequality is
\[ \frac{p+1}{p+\mu} - \frac{1}{p'} =  \frac{p(2-\mu) + \mu}{p(p+\mu)} > 0, \]
which combined with \eqref{v_denom} and \eqref{v_numer_1} implies \eqref{limconcerb}.

\smallskip

{\bf Case 2:  $\frac{2}{p'}<s<2$.}
Let $z:= \iota p'(m-1) r^{1+\frac{\mu}{p}}$. Since
$\frac{p(p's-2)-\mu}{p+\mu}  > -1$,
we have
\begin{align}
\int_{1}^{+\infty}   e^{-\iota p'(m-1) r^{1+\frac{\mu}{p}} } r^{p's-2} {\dd}r
& \leq  \iota^{- \frac{p(p's-1)}{p+\mu}} \left(\frac{p}{p+\mu}\right) \left[p'(m-1)\right]^{- \frac{p(p's-1)}{p+\mu}} \int^{+\infty}_0 e^{-z}  z^{ \frac{p(p's-2)-\mu}{p+\mu}}   {\dd}z\notag\\
&= \iota^{- \frac{p(p's-1)}{p+\mu}}  \left(\frac{p}{p+\mu}\right)\left[p'(m-1)\right]^{- \frac{p(p's-1)}{p+\mu}}\Gamma\left(\frac{p(p's-2)-\mu}{p+\mu} +1 \right)\notag\\
&=:\iota^{- \frac{p(p's-1)}{p+\mu}}   C_2(s, m,  p). \label{v_numer_second_deltainfty}
\end{align}
The inequalities \eqref{v_numer_first}--\eqref{v_numer_second_0delta},  and \eqref{v_numer_second_deltainfty} then yield
\begin{align}\label{v_numer}
 & \left( \int_{\mathbb{B}^n_1(\mathbf{0})} {\B}^{*p}({\dd}v_{\iota}) {\dd}x \right)^{1/p} \left( \int_{{\mathbb{B}^n_1(\mathbf{0})}}  {|v_\iota|^{p'(m-1)}}{r^{p's}} {\dd}x \right)^{1/p'} \notag\\
 \leq \  & n\omega_n \iota^{\frac{p+1}{p+\mu}} \left( \iota^{- \frac{p(p's-1)}{p+\mu}} C_2(s, m, p ) + \frac{1}{p's+n} \right)^{{1}/{p'}} \left(\iota^{\frac{p(\mu-1)}{p+\mu}}  + C_1(p) \right)^{{1}/{p}}\left(1+\frac{\mu}{p}\right).
\end{align}
The lowest power of $\iota$ on the right-hand side of the above inequality is
\begin{equation} \label{Berwald_s_range}
\frac{p+1}{p+\mu} - \frac{p(p's-1)}{p'(p+\mu)} = \frac{p(2-s)}{p+\mu} > 0,
\end{equation}
which combined with \eqref{v_denom} and \eqref{v_numer} implies \eqref{limconcerb}.

\smallskip

{\bf Case 3:  $s=2$.}
We claim that the denominator of $\mathscr{C}_{\mathbf{0},p,m,s}(v_\iota)$ diverges, i.e.,
\begin{align} \label{Specialfunc}
\lim_{\iota\rightarrow 0^+}\int_{{\mathbb{B}^n_1(\mathbf{0})}}  {|v_{\iota}|^{m}}{r^{s-1}}{\dd}x
=+\infty.
\end{align}
Indeed,
by using \eqref{Berwald_dxdr} and setting $\bar{z} : = \iota m r^{1+\frac{\mu}{p}}$, we obtain
\begin{align*}
\int_{{\mathbb{B}^n_1(\mathbf{0})}}  {|v_{\iota}|^{m}}{r^{s-1}}{\dd}x&= n\omega_n\int^{+\infty}_0 e^{-\iota m r^{1+\frac{\mu}{p}} }   \frac{r^n }{(1+r)^{n+1}}  {\dd}r
 =  n\omega_n\int^{+\infty}_0 e^{- \bar{z} } \frac{ \bar{z} ^{\frac{n p}{p+\mu}} }{\left[ (\iota m)^{\frac{p}{p+\mu}}+  \bar{z} ^{\frac{p}{p+\mu}} \right]^{n+1}}     {\dd}\bar{z}^{\frac{p}{p+\mu}} \notag \\
&\geq \frac{n\omega_n}{e}\int^{1}_0 \frac{ \bar{z} ^{\frac{n p}{p+\mu}} }{\left[ (\iota m)^{\frac{p}{p+\mu}}+  \bar{z} ^{\frac{p}{p+\mu}} \right]^{n+1}}     {\dd}\bar{z}^{\frac{p}{p+\mu}}.   \notag
\end{align*}
Put $\bar{t}: = \bar{z}^{\frac{p}{p+\mu}}$ and  $a:=(\iota m)^{\frac{p}{p+\mu}}$.
Then Lemma \ref{lemma_function}
 implies \eqref{Specialfunc}.
 For $s=2$ the numerator of   $\mathscr{C}_{\mathbf{0},p,m,s}(v_\iota)$ remains finite and positive (see \eqref{v_numer} and \eqref{Berwald_s_range}); consequently \eqref{limconcerb} follows.
This   completes the proof of Statement (\ref{BerwCKN_i}).

\textbf{(\ref{BerwCKN_ii})}
Equations \eqref{Berwaldphi} and \eqref{exparealphbeata} yield
$\B(x,y) = \frac{(\alpha+\beta)^2}{\alpha} \leq  \frac{(\alpha+\alpha)^2}{\alpha}= 4 \alpha(x,y). $
Then  \eqref{coFinsler} implies
\[
\B^{*}(x,\xi)=\sup_{y\in T_x\mathbb{B}^n_1(\mathbf{0})\backslash\{0\}} \frac{\langle \xi,y \rangle}{\B(x,y)} \geq \sup_{y\in T_x\mathbb{B}^n_1(\mathbf{0})\backslash\{0\}} \frac{\langle \xi,y \rangle}{4\alpha(x,y)}=\frac{1}{4} \alpha^*(x,\xi),
\]
where $\alpha^*(x,\xi)$ is the co-metric of the Riemannian metric $\alpha$.

For any $f\in C^{\infty}_0 (\mathbb{B}^n_1(\mathbf{0})) \backslash \{0\}$,
the above inequality together with \eqref{LemmaBerwaldGradr}, \eqref{dualff*}, and \eqref{exparealphbeata}  yields
\begin{align*}
\left| \left\langle {\dd}f, \nabla r    \right\rangle  \right| & = \frac{(1-|x|)^2}{|x|}\left| \left\langle {\dd}f, x    \right\rangle \right| \leq
\frac{(1-|x|)^2}{|x|}  \alpha^*(x,{\dd}f) \alpha(x,x)\\
 &= \frac{1}{(1+|x|)^2} \alpha^*(x,{\dd}f)
\leq    \frac{4}{(1+|x|)^2} \B^*(x,{\dd}f) \leq  4 \B^*(x,{\dd}f).
\end{align*}
Consequently,
\begin{align}
\int_{\mathbb{B}^n_1(\mathbf{0})}  {|f|^{m-1}}{r^{s}} \left| \left\langle {\dd}f, \nabla r    \right\rangle \right| {\dd}x
& \leq 4 \int_{\mathbb{B}^n_1(\mathbf{0})}  {|f|^{m-1}}{r^{s}}  {\B}^*({\dd}f) {\dd}x  \notag\\
& \leq 4 \left( \int_{\mathbb{B}^n_1(\mathbf{0})}  \B^{*p}({\dd}f) {\dd}x \right)^{1/{p}} \left( \int_{\mathbb{B}^n_1(\mathbf{0})} |f|^{p'(m-1)} r^{p's} {\dd}x \right)^{1/{p'}}. \label{Berwalds>=2numer}
\end{align}

Let $\dii$ be the Euclidean divergence in $\mathbb{R}^n$. A direct computation  combined with \eqref{LemmaBerwaldGradr} and \eqref{Berwald_dxdr}  gives
\begin{equation}\label{Berwald_Laplace}
\Delta r: =\dii\circ\nabla r=  \frac{\partial }{\partial x^i} \left[ \frac{(1-|x|)^2}{|x|} x^i \right] =  \frac{n-1}{|x|}+(n+1)|x|-2n
= \frac{n-1}{r} - \frac{n+1}{1+r}.
\end{equation}
On the one hand, the divergence theorem together with    $\langle {\dd}r,\nabla r\rangle=1$ and \eqref{Berwalds>=2numer} implies
\begin{align}\label{leftberwres}
&\quad \int_{\mathbb{B}^n_1(\mathbf{0})} {|f|^{m}}{r^{s}}\Delta r {\dd}x + s \int_{\mathbb{B}^n_1(\mathbf{0})} {|f|^{m}}{r^{s-1}} {\dd}x \notag\\
& =  - \int_{\mathbb{B}^n_1(\mathbf{0})} \left\langle {\dd}\left( {|f|^{m}}{r^{s}}\right), \nabla r    \right\rangle{\dd}x + s \int_{\mathbb{B}^n_1(\mathbf{0})} {|f|^{m}}{r^{s-1}} {\dd}x
%& = - m  \int_{\mathbb{B}^n_1(\mathbf{0})}  {|f|^{m-2} f}{r^{s}} \left\langle {\dd}f, \nabla r    \right\rangle{\dd}x - s \int_{\mathbb{B}^n_1(\mathbf{0})} {|f|^{m}}{r^{s-1}} {\dd}x \notag\\
 \leq m \int_{\mathbb{B}^n_1(\mathbf{0})}  {|f|^{m-1}}{r^{s}}\left| \left\langle {\dd}f, \nabla r    \right\rangle \right| {\dd}x
\notag\\
&\leq 4m \left( \int_{\mathbb{B}^n_1(\mathbf{0})}  \B^{*p}({\dd}f) {\dd}x \right)^{1/{p}} \left( \int_{\mathbb{B}^n_1(\mathbf{0})} |f|^{p'(m-1)} r^{p's} {\dd}x \right)^{1/{p'}}.
\end{align}
On the other hand, it follows from  \eqref{Berwald_Laplace} that
\begin{align}
& \quad \int_{\mathbb{B}^n_1(\mathbf{0})} {|f|^{m}}{r^{s}} \Delta r {\dd}x  + s \int_{\mathbb{B}^n_1(\mathbf{0})} {|f|^{m}}{r^{s-1}} {\dd}x \notag\\
&  = (s-2) \int_{\mathbb{B}^n_1(\mathbf{0})}  {|f|^{m}}{r^{s-1}}   {\dd}x  + (n+1) \int_{\mathbb{B}^n_1(\mathbf{0})} {|f|^{m}} \frac{ r^{s-1}}{1+r}  {\dd}x  \geq  (s-2) \int_{\mathbb{B}^n_1(\mathbf{0})}  {|f|^{m}}{r^{s-1}}   {\dd}x, \label{BerwaldCKNs>=2_secondpart}
\end{align}
with equality if and only if $f\equiv0$.
Now Statement \eqref{BerwCKN_ii} follows by \eqref{leftberwres}, \eqref{BerwaldCKNs>=2_secondpart}, and \eqref{granotwoff}.
\end{proof}

The optimal constant is central to the analysis of functional inequalities. For reversible Finsler manifolds, the sharp constants of many classical inequalities coincide with their Riemannian counterparts (cf.~Huang--Krist\'aly--Zhao \cite{HKZ}, Krist\'aly--Repov\v{s}\cite{KR16}, and Zhao \cite{ZhaoHardy}). In the finite-reversibility setting, Mester--Peter--Varga \cite{MPV} proved that several fundamental inequalities remain valid, while Kaj\'ant\'o \cite{Ka} showed that  their sharp constants equal the Riemannian constants divided by the reversibility.

In contrast, Berwald's metric possesses infinite reversibility, thereby placing it outside the scope of all preceding results. The threshold behavior stated in Theorem \ref{sbeer} therefore represents a new phenomenon in this singular regime. Whether the constant $\frac{2-s}{4m}$ in Theorem \ref{sbeer}/\eqref{BerwCKN_ii} is sharp remains an intriguing open question.

\section{Funk metric spaces}\label{sectionFunk}
This section investigates Funk metric spaces, a fundamental class of projectively flat Finsler manifolds with constant negative curvature. Subsection \ref{subsectionBasicFunk} collects their basic geometric properties, while Subsection \ref{SobolevFunk} establishes the nonlinearity of their Sobolev spaces and the failure of the uncertainty principle and CKN inequality. Theorem \ref{funkcknf} is established, which restates Theorem \ref{ThmFunkInequalities}/(\ref{FunkCKN}) for Funk metric spaces.
\subsection{Background of Funk metric spaces}\label{subsectionBasicFunk}

\begin{definition}\label{funkdef}
A Finsler metric $\mathsf{F}$ on a bounded domain $\Omega\subset \mathbb{R}^n$ is called a  {\it Funk metric} if
\begin{equation}\label{funkdef}
x + \frac{y}{\mathsf{F}(x,y)}  \in \partial \Omega,
\end{equation}
for  every $x\in \Omega$ and  $ y \in T_x \Omega\backslash\{0\}$. The pair $(\Omega,\mathsf{F})$ is called a  {\it Funk metric space}.
\end{definition}

In \eqref{funkdef}   both $x$ and $y/\mathsf{F}(x,y)$ are  interpreted as vectors in $\mathbb{R}^n$ with their sum emanating from  the origin.
The following characterization of Funk metrics is due to Li~\cite{Li} and Shen~\cite{Sh1}.
\begin{theorem}[\cite{Li, Sh1}]\label{funkmetrisufficnecessy}
A   pair $(\Omega,\mathsf{F})$ is a Funk metric space if and only if
there is a (unique) Minkowski norm $\phi$ on $\mathbb{R}^n$ such that
\begin{equation} \label{phiF}
	\partial\Omega=\phi^{-1}(1),\qquad \mathsf{F}(x,y) = \phi{\left(y + x \mathsf{F}(x,y)\right)}.
\end{equation}
In particular, $\Omega$ is a bounded strongly convex domain in $\mathbb{R}^n$ given by
\begin{equation}\label{omegedom}
	\Omega=\{x\in \mathbb{R}^n\,:\, \phi(x)<1\}.
\end{equation}
\end{theorem}

\begin{example} Let $\phi(x)=|x|$  be the Euclidean norm of $\mathbb{R}^n$. Then $\Omega=\mathbb{B}^n_1({\mathbf{0}})$, and the corresponding Funk metric is
\begin{align*}%\label{Funkex1}
 \mathsf{F}(x,y) = \frac{ \sqrt{(1-|x|^2)|y|^2 + \langle x, y \rangle^2}}{1-|x|^2} + \frac{\langle x, y \rangle}{1-|x|^2}, \qquad \forall (x,y)\in T\mathbb{B}^n_1(\mathbf{0}),
 \end{align*}
which coincides with the metric \eqref{Funkex1intro} discussed in the Introduction.
\end{example}

We collect several fundamental properties of Funk metric spaces that will be used repeatedly.
\begin{theorem}[\cite{ShenSpray}]\label{bascifunthm}
Let $(\Omega,\mathsf{F},\m_{BH})$  be an $n$-dimensional Funk metric space equipped with the Busemann--Hausdorff measure. Then the following statements are true:
\begin{enumerate}[\rm (i)]

\item $(\Omega,\mathsf{F})$ is a Cartan--Hadamard manifold with $\mathbf{K}=-\frac14$;

\item the S-curvature is constant, i.e., $\mathbf{S}=\frac{n+1}{2}$;

\item\label{ShenSpray}  $\m_{BH}=\sigma \mathscr{L}^n$ with
\begin{equation}\label{Vol_Funk}
 \sigma  = \frac{\vol(\mathbb{B}^n_1)}{\vol(\Omega)}=\frac{\omega_n}{\vol(\Omega)}.
\end{equation}
 \end{enumerate}
\end{theorem}

The explicit form of the co-metric for a Funk metric was given by Huang--Mo \cite{HuangMo} and Shen \cite{Sh0}.
\begin{lemma}[\cite{HuangMo,Sh0}]\label{Funkdualmetric}
Let $(\Omega,\mathsf{F})$  be a Funk metric space and $\phi$ the Minkowskian norm satisfying  \eqref{phiF}. The dual metric $\mathsf{F}^{*}$ is given by
\begin{equation}\label{co-metricrelation}
\mathsf{F}^{*}(x,\eta) = \phi^{*}(\eta) - \langle \eta, x   \rangle,
\end{equation}
where $\ds\phi^{*}(\eta):=\underset{y\in T_x\Omega\backslash \{0\}}{\sup}\frac{\eta(y)}{\phi(y)}$ and $ \langle \eta, x   \rangle=\eta(x)$ is the canonical pairing between $T_x^*\Omega$ and $T_x\Omega$.
\end{lemma}

Every  Funk metric space is projectively flat; this yields concrete formulas for distances and reversibility.
\begin{theorem}[\cite{KLZ}] \label{Funkdist}
Let $(\Omega,\mathsf{F})$  be a Funk metric space and $\phi$ the Minkowskian norm satisfying \eqref{phiF}. Then the following statements are true:
\begin{enumerate}[\rm (i)]
%\item\label{distexprseeion1} the distance function is given by
%\begin{equation}\label{phi_dpq}
%d_{\mathsf{F}}(x_1, x_2) = \ln \left[\frac{\phi(z - x_1)}{\phi(z - x_2)}\right],\quad \forall x_1,x_2\in \Omega,
%\end{equation}
%where $z$ denotes the intersection point of the ray $l_{x_1 x_2}(t):= x_1 + t (x_2 - x_1)$, $t\geq 0$, with $\partial \Omega$, i.e.,
%\[
%z=x_1 + \lambda (x_2-x_1) \in \partial\Omega, \quad \text{ for some }\lambda>1.
%\]

\item\label{distexprseeion2} the distance function  $d_{\mathsf{F}} (\mathbf{0},x)$ is given by
\begin{equation}\label{Funkd0xdx0}
d_{\mathsf{F}} (\mathbf{0},x) = -\ln [1-\phi(x)];
\end{equation}

\item\label{distexprseeion3} the Funk metric space $(\Omega,\mathsf{F})$ is forward complete (but not backward complete) with
\begin{equation*}
\lambda_{{\mathsf{F}}}(x) \in \left[\frac{1 +\phi(x)}{1-\phi(x)},\frac{\lambda_\phi +\phi(x)}{1-\phi(x)}\right],\qquad  \lambda_F(\Omega)=+\infty.
\end{equation*}
where $\lambda_\phi:=\sup_{y\neq 0}\phi(-y)/\phi(y)$ is the reversibility of $\phi$.
\end{enumerate}
\end{theorem}

\subsection{Analytic properties of Funk metric spaces}\label{SobolevFunk}
%In this subsection, we study the   topologies of   Sobolev spaces over Funk metric spaces.

\begin{lemma}\label{FunkIntegral}
Let $(\Omega,\mathsf{F})$ be  an $n$-dimensional Funk metric space   and  $\phi$ the Minkowskian norm satisfying  \eqref{phiF}.
If $f$ is a measurable function on $\mathbb{R}$ such that $f\circ\phi\in L^1(\Omega, \mathscr{L}^n)$, then
\begin{equation}\label{FunkIntegralEq}
\int_\Omega f\circ\phi \dm_{BH} = n\omega_n \int_0^1 f(t) t^{n-1} {\dd} t  = n\omega_n \int_0^{+\infty} f(1 - e^{-r})    e^{-r} (1-e^{-r})^{n-1} {\dd} r,
\end{equation}
where the variable $r$ denotes the distance from the origin given by \eqref{Funkd0xdx0}.
\end{lemma}

\begin{proof}
The Minkowski norm $\phi$ induces a Riemannian metric on $\mathbb{R}^n\backslash\{\mathbf{0}\}$ (cf.~\cite[Sections 14.2]{BCS}), i.e.,
\[
\hat{g}:= \hat{g}_{ij}(x) {\dd}x^i\otimes {\dd}x^j:=\frac12 \frac{\partial^2\phi^2}{\partial x^i \partial x^j}(x) {\dd}x^i\otimes {\dd}x^j.
\]
It should be remarked that
$\hat{g}_{ij}(x)=\hat{g}_{ij}(\lambda x)$ for any $\lambda>0$.

Let $(t,v)$ denote the polar coordinate system  around $\mathbf{0}$ induced by $\hat{g}$, i.e.,
\begin{equation}\label{polarcorphi}
t:=\phi(x),\qquad v:=\frac{x}{\phi(x)}\in S_{\mathbf{0}}:=\{y\in \mathbb{R}^n\,:\, \phi(y)=1\}.
\end{equation}
According to  \cite[Sections 14.9]{BCS}, the canonical Riemannian metric satisfies
\begin{equation}\label{volueversionp}
 {\dd}{\vol_{\hat{g}}}(x)=\sqrt{\det \hat{g}_{ij}(x)}\, {\dd}x= t^{n-1}{\dd}t \,{\dd}\nu_{\mathbf{0}}(v),
\end{equation}
where  ${\dd}\nu_{\mathbf{0}}$ is the canonical Riemannian measure of  $S_{\mathbf{0}}$ (see \cite[(1.4.8)]{BCS}), i.e.,
\begin{equation}\label{spherevolue}
{\dd}\nu_{\mathbf{0}}|_x=\sqrt{\det \hat{g}_{ij}(x)}\, \sum_{j=1}^n (-1)^{j-1} x^i {\dd}x^1\wedge \cdots \wedge {\dd}x^{j-1}\wedge {\dd}x^{j+1}\wedge \cdots  {\dd}x^n.
\end{equation}

%On the other hand, $\phi$ can be viewed as a   Finsler metric $\widetilde{F}(x,y):=\phi(y)$ on $\mathbb{R}^n$.
%Its  fundamental tensor coincides with  the one of the Riemannian metric $\hat{g}$, i.e.,
%\begin{equation}\label{basictensorcoin}
%\tilde{g}_{ij}(x,y):=\frac12 \frac{\partial^2 \widetilde{F}^2(x,y)}{\partial y^i \partial y^j} =\hat{g}_{ij}(y).
%\end{equation}
%Moreover, by \eqref{omegedom} and \eqref{Vol_Funk}, the Busemann--Hausdorff measures of $(\mathbb{R}^n,\widetilde{F})$ and $(\Omega,F)$ coincide because
%\[
%{\dd}\widetilde{\m}_{BH}=\frac{\vol(\mathbb{B}^n_1)}{\vol\left(\{y\in T_x\widetilde{M}\,:\, \widetilde{F}(x,y)<1\}\right)}{\dd}x=\frac{\omega_n}{\vol\left(\{y\in \mathbb{R}^n\,:\, \phi(y)<1\}\right)}{\dd}x=\frac{\omega_n}{\vol(\Omega)}{\dd}x=\dm_{BH}.
%\]

Set $\sigma:=\frac{\omega_n}{\vol(\Omega)}$ as in \eqref{Vol_Funk}. It is straightforward to check that
\[
\Psi(x):=\frac{\sigma }{\sqrt{\det \hat{g}_{ij}(x)}}
\]
is a scalar function on $\mathbb{R}^n\backslash\{\mathbf{0}\}$ with $\Psi(x)=\Psi(x/\phi(x))$; thus it may be regarded as a function on $S_{\mathbf{0}}$.
By  Stokes' theorem and \eqref{spherevolue}, we have
\begin{align}\label{distourinteger}
\int_{S_\mathbf{0}}\Psi(v){\dd}\nu_{\mathbf{0}}(v)
& = \sigma \int_{S_\mathbf{0}}\sum_{j=1}^n (-1)^{j-1} x^i {\dd}x^1\wedge \cdots \wedge {\dd}x^{j-1}\wedge {\dd}x^{j+1}\wedge \cdots  {\dd}x^n\notag\\
& = n\sigma \int_{\{\phi<1\}}{\dd}x=n \sigma \vol(\Omega) =n \omega_n.
\end{align}

The assumption combined with Theorem \ref{bascifunthm}/\eqref{ShenSpray} implies   $f\circ\phi\in L^1(\Omega, \m_{BH})$.
Let  $(t,v)$ be the polar coordinate system  as in \eqref{polarcorphi}. A direct calculation combined with \eqref{polarcorphi}, \eqref{volueversionp}, and \eqref{distourinteger} yields
\begin{align*}
\int_\Omega f\circ\phi  \dm_{BH}&=\int_{\Omega\backslash\{\mathbf{0}\}} f(\phi(x))\sigma {\dd}x= \int_{\Omega\backslash\{\mathbf{0}\}} f(\phi(x))\frac{\sigma}{\sqrt{\det \hat{g}(\frac{x}{\phi(x)})}}{\sqrt{\det \hat{g}(x)}} {\dd}x\\
&=\int_{\Omega\backslash\{\mathbf{0}\}} f(\phi(x))\Psi\left(\frac{x}{\phi(x)}\right)\dvol_{\hat{g}}(x)=\int_{S_\mathbf{0}} \left(\int^1_0 f(t)\Psi(v)t^{n-1}{\dd}t\right) {\dd}\nu_{\mathbf{0}}(v)\\
&=n\omega_n\int^1_0 f(t) t^{n-1}{\dd}t.
\end{align*}

Finally,   \eqref{Funkd0xdx0} implies $r = -\ln(1-t)$ with $r = d_{\mathsf{F}}(\mathbf{0}, x)$.
 Hence, we have $t = 1 - e^{-r}$, and
\[ \int_\Omega f\circ\phi  \dm_{BH} =  n\omega_n\int^1_0 f(t) t^{n-1}{\dd}t = n\omega_n \int^{+\infty}_0 f(1 - e^{-r}) e^{-r} (1 - e^{-r})^{n-1}  {\dd}r,
\]
which completes the proof.
\end{proof}

\begin{lemma}\label{FunkW1p}
Let $(\Omega,\mathsf{F})$ be  an $n$-dimensional Funk metric space  and  $\phi $ the Minkowski norm satisfying  \eqref{phiF}.
For $\iota\in (0,1)$, define
\begin{equation*}
 u_\iota(x): = -[{1-\phi(x)}]^\iota.
\end{equation*}
Then $u_\iota\in W_0^{1,p}(\Omega,\mathsf{F},\m_{BH})$ and
\begin{align}
\int_\Omega |u_{\iota}|^p \dm_{BH} &= \frac{n!\, \omega_n}{(\iota p+1)(\iota p+2)\cdot\cdot\cdot(\iota p+n)}, \label{Funkukeq}
\\
\int_\Omega \mathsf{F}^{*p}( {\dd}u_\iota) \dm_{BH} &=  \frac{\iota^p n!\, \omega_n}{(\iota p+1)(\iota p+2)\cdot\cdot\cdot(\iota p+n)}. \label{FunkFdueq}
\end{align}
\end{lemma}
\begin{proof} For simplicity, write $\m := \m_{BH}$.
By Lemma \ref{FunkIntegral}, we have
\begin{equation*}
\int_\Omega |u_{\iota}|^p \dm = n\omega_n \int_0^1 t^{n-1}(1-t)^{\iota p} {\dd}t =n\omega_n\mathfrak{B}(n,\iota p+1) =\frac{n!\, \omega_n}{(\iota p+1)(\iota p+2)\cdot\cdot\cdot(\iota p+n)},
\end{equation*}
where $\mathfrak{B}(\cdot,\cdot)$ denotes the Beta function. This gives \eqref{Funkukeq}.

On the other hand, for any $x\in \mathbb{R}^n\backslash\{\mathbf{0}\}$, by \eqref{weakswaza} we have
\begin{equation}\label{phihom}
\phi^*({\dd}\phi|_x)=\sup_{y\neq 0}\frac{\langle  {\dd}\phi|_x, y\rangle}{\phi(y)}= 1, \qquad \langle {\dd}\phi|_x , x \rangle=x^i \frac{\partial \phi}{\partial x^i}(x)=\phi(x).
\end{equation}
Then by Lemma \ref{Funkdualmetric}, we obtain
\begin{equation*}%\label{Fuaua}
\mathsf{F}^*(x,{\dd}u_\iota)=\phi^*( {\dd}u_\iota|_x  )-\langle {\dd}u_\iota|_x, x   \rangle=\iota[{1-\phi(x)}]^\iota=\iota|u_\iota|(x),
\end{equation*}
which together with \eqref{Funkukeq} yields \eqref{FunkFdueq} directly. Therefore, $\|u_\iota\|_{W^{1,p}_{\m}}$ is finite.
The proof of  $u_\iota\in W_0^{1,p}(\Omega,\mathsf{F},\m)$
follows the same pattern as Lemma \ref{BerwaldSobolevLemma1} and is therefore omitted.
\end{proof}

\begin{theorem}\label{MainThmFunk}
Let $(\Omega,\mathsf{F},\m_{BH})$ be  an $n$-dimensional Funk metric space endowed with the Busemann-Hausdorff measure.
Then the following statements are true:
\begin{enumerate}[\rm (i)]
\item \label{funksob1} $W^{1,p}_0(\Omega,\mathsf{F},\m_{BH})$ is not a vector space for each  $p\in(1,+\infty)$;

\item \label{funksob2} the forward and backward topologies of $W_0^{1,p}(\Omega,\mathsf{F},\m_{BH})$ are incompatible.
 In particular, there exist  a sequence $(u_k)$ and $u$ in  $W^{1,p}_0(\Omega,\mathsf{F},\m_{BH})$ such that
 \[
 \lim_{k\rightarrow \infty}\|u-u_k\|_{W^{1,p}_{\m_{BH}}}=0, \qquad \|u_k-u\|_{W^{1,p}_{\m_{BH}}}= + \infty, \ \forall k\in \mathbb{N}.
 \]

 \end{enumerate}

\end{theorem}
\begin{proof}For simplicity, write $\m:=\m_{BH}$
 and let $\phi$   be  the Minkowski norm  satisfying \eqref{phiF}.
 Fix $p\in (1,+\infty)$ and choose $\iota\in (0,1-1/p)$. Set
 \[
 u_\iota(x): = -[{1-\phi(x)}]^\iota, \qquad \forall x\in \Omega.
 \]
 Lemma \ref{FunkW1p} implies $u_{\iota}\in W^{1,p}_0(\Omega,\mathsf{F},\m)$.
 Now we show $-u_\iota \not\in W^{1,p}_0(\Omega,\m)$.
Indeed, the second equation in \eqref{phihom} yields
\begin{equation*}%\label{xDu}
\langle{\dd}u_{\iota}|_x, x \rangle=  \iota[1-\phi(x)]^{\iota-1}\phi(x)\geq 0.
\end{equation*}
Using the expression for the co-metric \eqref{co-metricrelation} and the integral formula \eqref{FunkIntegralEq}, we obtain
%which together with (\ref{co-metricrelation}) and \eqref{FunkIntegralEq} yields
\begin{align*}
\|-u_\iota \|^p_{W^{1,p}_{\m}}\geq &\int_\Omega \mathsf{F}^{*p}(x, -{\dd}u_\iota)\dm
=  \int_\Omega\left[\phi^{*}(-{\dd}u_\iota) + \langle {\dd}u_\iota, x  \rangle\right]^p\dm \geq \int_\Omega \langle {\dd}u_{\iota}, x  \rangle^p \dm\\
=&
 {n\omega_n\iota^p}   \int_0^1  t^{p+n-1} (1-t)^{p(\iota-1)}  {\dd}t = \mathfrak{B} \big(p+n, p(\iota-1)+1\big)=+\infty,
\end{align*}
where the last equality follows by $p(\iota-1)+1<0$.
Hence $-u_\iota \not\in W^{1,p}_0(\Omega,\mathsf{F},\m_{BH})$, proving (\ref{funksob1}).

%Statement (\ref{funksob2}) follows by the same argument as in the proof of Corollary \ref{forwbacktopso}.
Statement (\ref{funksob2}) follows from the same argument used to prove Corollary \ref{forwbacktopso}.
\end{proof}

The failure of the Hardy inequality on Funk metric spaces was proved by Krist\'aly--Li--Zhao \cite{KLZ}.
Using the test function $u_\iota(x): = -[{1-\phi(x)}]^\iota$, we shall show in
Theorem \ref{ThmFunkInequalities} that both the uncertainty principle and the generalized Caffarelli--Kohn--Nirenberg inequality fail on Funk metric spaces.
The proof relies on a basic estimate for a certain exponential integral, which we record first.

%In Section \ref{RiccSboundver}, we will provide an alternative proof based on the volume estimate and show that these phenomena appear on every Finsler metric measure manifold  satisfying  $\mathbf{Ric}\geq -(n-1)k^2$ and $\mathbf{S}\geq (n-1)h$ with $h>k\geq 0$.

\begin{lemma} \label{Lemma_integral_Funk}
For every
$b > - n$,  there exists a constant  $C(b,n) > 0$ such that
\begin{equation}
 \int_0^{+\infty} r^b e^{-r} (1 - e^{-r})^{n-1} {\dd}r  \leq  C(b,n) < +\infty.
\end{equation}
\end{lemma}
\begin{proof}
Note that
$1 - e^{-r} \leq r$ for all $r\in[0, +\infty)$. This together with $b+n> 0$ yields
\begin{align*}
 \int_0^{+\infty} r^b e^{- r} (1 - e^{-r})^{n-1} {\dd}r
& \leq  \int_0^{+\infty}  e^{- r} r^{b+n -1} {\dd}r=\Gamma(b+n)=:C(b,n)<+\infty,
\end{align*}
which completes the proof.
\end{proof}

%The failure of the Hardy inequality on Funk metric spaces is established in  Krist\'aly--Li--Zhao \cite{KLZ}. Furthermore, we have the following result.

%The failure of Hardy inequality on Funk metric spaces has been proven in Krist\'aly--Li--Zhao \cite{KLZ}. Moreover,  we have the following.

\begin{theorem}\label{ThmFunkInequalities}
Let $(\Omega,\mathsf{F},\m_{BH})$ be  an $n$-dimensional Funk metric space endowed with the Busemann--Hausdorff measure and set  $r(x):=d_{\mathsf{F}}(\mathbf{0},x)$.
Then the following statements are true:
\begin{enumerate}[{\rm (i)}]
%\item\label{FunkHardy}
%the $L^p$-Hardy inequality fails for all $p\in (1, n)$, i.e.,
%		\[
%		\inf_{f\in C^\infty_0{\Omega}\backslash\{0\}}\frac{\ds\int_{\Omega} {\mathsf{F}}^{*p}({\rm d} f) {\dd}{\m_{BH}} }{\ds\int_{\Omega} {|f|^p}{r^{-p}}{\dd}{\m_{BH}} }=0;\]
		
\item\label{FunkHPW}
the generalized uncertainty principle fails, i.e., for any $-p+1<s\leq 1<p<n$ ,
		\[
		\inf_{f\in C^\infty_0{\Omega}\backslash\{0\}}\frac{\Big( \int_{{\Omega}} {\mathsf{F}}^{*p}({\rm d} f){\dd}{\m_{BH}} \Big)^{1/p}\left(  \int_{{\Omega}}|f|^p r^{p's} {\dd}{\m_{BH}} \right)^{1/{p'}} }{  \int_{{\Omega}} |f|^p r^{s-1} {\dd}{\m_{BH}}   }=0;
		\]
%{\color{blue} For $s=1$ this is the Heisenberg-Pauli-Weyl uncertainty principle for $s=1$; for $s = 0$ the Hydrogen uncertainty principle. }

\item \label{FunkCKN}
the generalized Caffarelli--Kohn--Nirenberg inequality fails, i.e., if $1<p<m$ and $p(n+s-1)>m(n-p)>0$,
		then
		\[
		\inf_{f\in C^\infty_0(\Omega)\backslash\{0\}}\frac{\Big( \int_{{M}}F^{*p}({\rm d} f) {\dm_{BH}}   \Big)^{1/p}\left( \int_{\Omega}  {|f|^{p'(m-1)}}{r^{p's}}{\dm_{BH}}   \right)^{1/p'}  }{ \int_{\Omega}  {|f|^{m}}{r^{s-1}}{\dm_{BH}}  }=0.
		\]
\end{enumerate}
\end{theorem}
\begin{proof}Write $\m: =\m_{BH}$
 and let $\phi$ be the Minkowski norm satisfying \eqref{phiF}.
Set
$u_\iota(x): = -[{1-\phi(x)}]^\iota$ for  $\iota\in (0,1)$.
Lemma \ref{FunkW1p} shows that $u_{\iota}\in W^{1,p}_0(\Omega,\mathsf{F},\m)$ and
\begin{align} \label{FunkF*p}
\int_\Omega \mathsf{F}^{*p}( {\dd}u_\iota) \dm &=  \frac{\iota^p n!\, \omega_n}{(\iota p+1)(\iota p+2)\cdot\cdot\cdot(\iota p+n)}
<  \frac{\iota^p n! p! \, \omega_n}{(p+n)!}.
\end{align}

\textbf{(\ref{FunkHPW})}
By \eqref{Funkd0xdx0} and \eqref{FunkIntegralEq}, a direct computation gives
\begin{align} \label{FunkHPW_1}
\int_{\Omega} {|u_\iota|^p}{r^{s-1}}{\dd}{\m}
& = n\omega_n\int_0^{+\infty} e^{- (\iota p +1) r} r^{s-1} (1 - e^{-r})^{n-1} {\dd}r  \notag\\
& \geq n\omega_n \int_{1/2}^{1} e^{- (p +1) r} r^{s-1} (1 - e^{-r})^{n-1} {\dd}r
=:n\omega_n C_1(p,s,n) > 0.
\end{align}

On the other hand, since $-p+1<s\leq 1<p<n$, a direct calculation yields
$p's = \frac{p s }{p -1} > - p > - n$. Then Lemma \ref{Lemma_integral_Funk} yields a constant $C_2(p,s,n) > 0$ such that
\begin{align} \label{FunkHPW_2}
\int_{\Omega} {|u_\iota|^p}{r^{p's}}{\dd}{\m}
%& = \int_0^1 (1-t)^{\iota a} \left[\ln\frac{1}{1-t}\right]^{b} t^{n-1} {\dd}t \notag\\
& = n\omega_n \int_0^{+\infty} r^{p's} e^{-(\iota p +1)r} (1 - e^{-r})^{n-1} {\dd}r \notag\\
& \leq  n\omega_n \int_0^{+\infty} r^{p's} e^{-r} (1 - e^{-r})^{n-1} {\dd}r\leq n\omega_n C_2(p,s,n).
\end{align}

%Hence
%\begin{equation}  \label{FunkHPW_2}
%\int_{\Omega} {|u_\iota|^p}{r^{p's}}{\dd}{\m} <  n\omega_n C_2(b,n).
%\end{equation}
%Define
% \begin{equation*}
%\mathcal {U}_{p,s}(f):=\frac{\left(  \int_{\Omega} {\mathsf{F}}^{*p}({\rm d} f){\dd}x \right)^{1/p}\left(  \int_{\Omega}|f|^p r^{p's} {\dd}x %\right)^{1/{p'}} }{  \int_{\Omega} |f|^p r^{s-1} {\dd}x   }.
%\end{equation*}
Now it follows from \eqref{FunkF*p}--\eqref{FunkHPW_2} that $\lim_{\iota\rightarrow 0^+}\mathscr{U}_{\mathbf{0},p,s}(u_{\iota})=0$,
%\[
%\lim_{\iota\rightarrow 0^+}\mathscr{U}_{\mathbf{0},p,s}(u_{\iota})\leq  \lim_{\iota\rightarrow 0^+} \frac{ \iota  [(n-1)! p!]^{\frac{1}{p}}  C_2^{\frac{1}{p'}}(b,n) }{[(p+n)!]^{\frac{1}{p}} C_1(n,p)} =0,
%\]
where $\mathscr{U}_{\mathbf{0},p,s}$ is the functional from  Definition \ref{function11}.
The same approximation argument as in the proof of Theorem \ref{MainThmBerw}/\eqref{Berw4} then yields
\begin{equation*}
\inf_{f\in C^\infty_0(\Omega)\backslash\{0\}}\mathscr{U}_{\mathbf{0},p,s} (f)\leq \liminf_{\iota \rightarrow 0^+} \mathscr{U}_{\mathbf{0},p,s} (u_{\iota})=0,
\end{equation*}
which establishes \eqref{FunkHPW}.

\smallskip
\textbf{(\ref{FunkCKN})}
The argument is parallel to the one above.
First, there exists a constant $C_3(s,m,n)>0$ with
\begin{equation} \label{FunkCKN_1}
\int_{{M}}  {|u_\iota|^{m}}{r^{s-1}}{\dm} \geq  n\omega_n C_3(s,m,n) > 0.
\end{equation}
Second, for the term with $r^{p's}$, by \eqref{FunkIntegralEq}, we have
\begin{align*}
\int_{\Omega} {|u_\iota|^{p'(m-1)}}{r^{p's}}{\dd}{\m}
%& = \int_0^1 (1-t)^{\iota a} \left[\ln\frac{1}{1-t}\right]^{b} t^{n-1} {\dd}t \notag\\
& = n\omega_n \int_0^{+\infty} r^{p's} e^{-(\iota p'(m-1) +1)r} (1 - e^{-r})^{n-1} {\dd}r
 \leq  n\omega_n \int_0^{+\infty} r^{p's} e^{-r} (1 - e^{-r})^{n-1} {\dd}r.
\end{align*}
The assumption $1<p<m$ and $p(n+s-1)>m(n-p)>0$  implies $s>1-p$, which combined with
$n> p $ gives
$p's = \frac{p s }{p -1} > - p > - n.$
Hence, by Lemma \ref{Lemma_integral_Funk}, there exists a constant $C_4(p,s,n) > 0$ such that
\begin{equation*} % \label{FunkCKN_2}
\int_{\Omega} {|u_\iota|^{p'(m-1)}}{r^{p's}}{\dd}{\m} \leq   n\omega_n C_4(p,s,n).
\end{equation*}
This combined with \eqref{FunkF*p} and \eqref{FunkCKN_1} implies  $\lim_{\iota\rightarrow 0^+}\mathscr{C}_{\mathbf{0},p,m,s}(u_{\iota})=0$,
where $\mathscr{C}_{\mathbf{0},p,m,s}$ is the functional from  Definition \ref{function11}.
%\[ \lim_{\iota\rightarrow 0^+} \frac{\left( \int_\Omega \mathsf{F}^{*p}( {\dd}u_\iota) \dm \right)^{\frac{1}{p}}  \left( \int_{\Omega} {|u_\iota|^p}{r^{p's}}{\dd}{\m} \right)^{\frac{1}{p'}}}{ \int_{\Omega} {|u_\iota|^{p'(m-1)}}{r^{p's}}{\dd}{\m} }
%\leq \lim_{\iota\rightarrow 0^+} \frac{ \iota  [(n-1)! p!]^{\frac{1}{p}}  C_4^{\frac{1}{p'}}(b,n) }{[(p+n)!]^{\frac{1}{p}} C_3(n,p)} = 0.
% \]
Thus,  (\ref{FunkCKN}) follows directly from a standard approximation argument.
\end{proof}

\begin{proof}[Proof of Theorem \ref{funkcknf}] In view of \eqref{granotwoff} and Theorem \ref{bascifunthm}/\eqref{ShenSpray}, the statement is an immediate consequence of Theorem \ref{ThmFunkInequalities}/(\ref{FunkCKN}).
\end{proof}

While both the Berwald and Funk metric spaces are Cartan--Hadamard manifolds with infinite reversibility and have many properties in common, their behavior regarding the generalized CKN inequality is strikingly different: the sharp threshold phenomenon present in the Berwald setting disappears entirely for Funk metrics.

\section{Role of $S$-Curvature}\label{RiccSboundver}

This section proves Theorem \ref{voluemcompar22}, which provides a general criterion for the failures of related functional inequalities on Finsler metric measure manifolds under a curvature domination condition.
The proof relies repeatedly on the integral decomposition \eqref{inffexprexx} and the volume comparison theorem \ref{bascivolurcompar}.

The following lemma is crucial for proving Theorem \ref{voluemcompar22}.
\begin{lemma}\label{lem_general}
Let $(M,F,\mathfrak{m})$ be an $n$-dimensional $\FMMM$ with $\mathbf{Ric} \geq -(n-1)k^2$ for some $k \geq 0$. Assume that there exists  a point $o \in M$ such that
\begin{equation}\label{generalSCONDITION}
\mathbf{S}(\nabla r) \geq (n-1)k + \frac{C}{1+r},
\end{equation}
where $r(x) = d_F(o,x)$ is the distance function from $o$, and the constant $C$ satisfies
\[
C>1  \ \text{ if } \ k>0; \qquad C \geq n  \ \text{ if } \ k=0.
\]
For $p\in (1, n)$,    let $a,b$ be two constants    satisfying
\[
a>0,\qquad \begin{cases}
-n < b< p' C, & \text{if } k > 0, \\
-n < b< p'( C-n+1), & \text{if } k = 0.
\end{cases}
\]
Define functions
$
v_\iota(x) := -e^{-\iota r^{1 + \frac{\mu}{p}}},
$ with constants $\iota, \mu$ satisfying
\[
0 < \iota < \min\left\{a^{-1}, p^{-1} \right\}, \qquad
\begin{cases}
C - 1 < \mu < C, & \text{if } k > 0, \\
C - n < \mu < C-n+1, & \text{if } k = 0.
\end{cases}
\]
Then there exist  constants $C_1:= C_1(k, n, p, C,\mu)$,  $C_2: = C_2(a,b,k,n,p,C,\mu)$ and $\xi:=\xi(b,k,p,C,\mu)$ with
\begin{align}
\int_{M} F^{*p}({\dd}v_{\iota}) \dm
&\leq
\begin{cases}
\mathscr{I}_{\m}(o)\iota^{ \frac{p(p+C-1)}{p+\mu}}C_1 ,& \text{ if }k>0,\\
 \mathscr{I}_{\m}(o)\iota^{ \frac{p(p+C-n)}{p+\mu}}C_1,   & \text{ if }k=0,
\end{cases}\label{lem_F*p_k>0=0}\\
 \int_{M} |v_\iota|^{a} r^{b} \dm    \leq \mathscr{I}_{\m}(o)\iota^{\xi}&\,C_2    \quad \text{with}\quad
 \xi> \begin{cases}
-p'\left(\frac{p+C-1}{p-\mu}\right) ,& \text{ if }k>0,\\[4pt]
 -p'\left(\frac{p+C-n}{p-\mu}\right),   & \text{ if }k=0,
\end{cases}\label{v_numer_two_estimate_k>0k=0}
\end{align}
where $\mathscr{I}_{\m}(o)$ is the integral of distortion \eqref{cocont}. In particular,
\begin{equation}\label{v_numer_estimate_both}
\lim_{\iota\rightarrow 0^+} \left( \int_{M} F^{*p}({\dd}v_{\iota}) \dm  \right)^{1/p} \left( \int_{M} |v_\iota|^{a} r^{b} \dm  \right)^{1/p'} = 0.
\end{equation}

\end{lemma}

\begin{proof} Since \eqref{v_numer_estimate_both} follows directly from \eqref{lem_F*p_k>0=0} and \eqref{v_numer_two_estimate_k>0k=0}, it suffices to establish the latter two estimates.
%We present the proof in two cases: $k > 0$ and $k = 0$.

Let $(r,y)$ be the polar coordinate system around $o$.
%From the definition of the $S$-curvature and the assumed lower bound we have
By  \eqref{Scurvaturedef}, \eqref{geommeaingofr}, and \eqref{remscurvature}, we have
\[
\tau(\gamma_y(r), \dot{\gamma}_y(r)) - \tau(o, y) = \int_0^r \mathbf{S}(\gamma_y(t), \dot{\gamma}_y(t)) \, \mathrm{d}t \geq \int_0^r \left[ (n-1)k + \frac{C}{1+t} \right] \mathrm{d}t = (n-1)kr + \ln(1+r)^C.
\]
Together with Theorem \ref{bascivolurcompar} this yields the key comparison
\begin{equation}\label{measexst_general_1_lem}
\hat{\sigma}_o(r,y)\leq  e^{-\tau(\gamma_y(r),\dot{\gamma}_y(r))}\mathfrak{s}^{n-1}_{-k^2}(r)\leq  e^{-\tau(o,y)-(n-1)kr} \mathfrak{s}^{n-1}_{-k^2}(r)\,(r+1)^{-C}, \quad \forall  r\in (0,i_y), \ y\in S_oM.
\end{equation}
The remainder of the proof is divided into two cases: $k>0$ and $k=0$. % Both cases follow a similar structure but with different parameter ranges and estimates due to the distinct lower bounds on the $S$-curvature.

\medskip

 \textbf{Case (i) $k > 0$.}
Since $\mathfrak{s}_{-k^2}(r) = \frac{\sinh(kr)}{k}$, we have
\begin{equation}\label{basckestk>0_general_1_lem}
e^{-(n-1)kr} \mathfrak{s}_{-k^2}^{n-1}(r) \leq (2k)^{1-n}.
\end{equation}
Moreover, there exists $\varepsilon_1=\varepsilon_1(k,n) \in (0,1)$ such that
\begin{equation}\label{basckestk>0_general_2_lem}
e^{-(n-1)kr} \mathfrak{s}_{-k^2}^{n-1}(r) \leq 2 r^{n-1}, \quad \forall r \in [0, \varepsilon_1].
\end{equation}

Since $F^*({\dd}r)=1$, a direct calculation together with   \eqref{inffexprexx}, \eqref{cocont}, \eqref{measexst_general_1_lem}, and \eqref{basckestk>0_general_1_lem} yields
\begin{align}
\int_{M} F^{*p}({\dd}v_{\iota}) \dm
& = \iota^p \left(1+\frac{\mu}{p}\right)^p  \int_{M} r^{\mu} e^{- \iota p r^{1+\frac{\mu}{p}}} \dm \notag\\
& \leq  \iota^p \left( 2 k  \right)^{1-n} \left(1+\frac{\mu}{p}\right)^p  \int_{S_oM}  \left( \int^{+ \infty}_0 e^{-\tau(o,y)} e^{- \iota p r^{1+\frac{\mu}{p}}} r^{\mu} (r+1)^{- C } {\dd}r  \right) {\dd}\nu_o(y) \notag \\
& \leq   \iota^p \left( 2 k  \right)^{1-n} \left(1+\frac{\mu}{p}\right)^p  \mathscr{I}_{\m}(o)  \left(   \int^{1}_0 r^{\mu} (r+1)^{- C}  {\dd}r
+ \int_1^{+\infty} e^{- \iota p r^{1+\frac{\mu}{p}}} r^{\mu} (r+1)^{- C}  {\dd}r \right) \notag \\
& \leq  \iota^p \left( 2 k  \right)^{1-n} \left(1+\frac{\mu}{p}\right)^p  \mathscr{I}_{\m}(o) \left( 1   +  \int_{1}^{+\infty}  e^{- \iota p r^{1+\frac{\mu}{p}}} r^{\mu - C}   {\dd}r \right). \label{v_numer_first0_general_1_lem}
\end{align}
The integral $\int_1^\infty e^{-\iota p r^{1 + \frac{\mu}{p}}} r^{\mu - C} \, \mathrm{d}r$ is estimated via the substitution $z = \iota p r^{1 + \frac{\mu}{p}}$.
Since $\frac{p(\mu -C ) - \mu}{p+\mu} > -1$,
\begin{align*}
\int_{1}^{+\infty} e^{- \iota p r^{1+\frac{\mu}{p}}}
r^{\mu-C}  {\dd}r
 &\leq \iota^{\frac{p(C-\mu-1)}{p+\mu}} \frac{p^{\frac{p(C-\mu)+\mu}{p+\mu}} }{p+\mu} \int_{0}^{+\infty} e^{- z}
z^{\frac{p(\mu-C) - \mu}{p+\mu}}  {\dd}z\\
&=\iota^{\frac{p(C-\mu-1)}{p+\mu}} \frac{p^{\frac{p(C-\mu)+\mu}{p+\mu}} }{p+\mu} \Gamma\left(\frac{p(\mu-C) - \mu}{p+\mu}+1\right)
 =:\iota^{\frac{p(C-\mu-1)}{p+\mu}}~ \overline{\mathcal {C}}_1(p,C,\mu).
\end{align*}
 Then it follows from \eqref{v_numer_first0_general_1_lem} and $\iota\in (0,1)$ that
\begin{align*}
 \int_{M} F^{*p}({\dd}v_{\iota}) \dm
&=  \iota^{ \frac{p(p+C-1)}{p+\mu}} \left( \iota^{\frac{p(\mu-C+1)}{p+\mu}} +  \overline{\mathcal {C}}_1(p,C,\mu) \right)   \left( 2 k  \right)^{1-n} \left(1+\frac{\mu}{p}\right)^p \mathscr{I}_{\m}(o)\notag\\
& \leq  \iota^{ \frac{p(p+C-1)}{p+\mu}} \left( 1 +  \overline{\mathcal {C}}_1(p,C,\mu) \right)   \left( 2 k  \right)^{1-n} \left(1+\frac{\mu}{p}\right)^p \mathscr{I}_{\m}(o)\\
&= :  \iota^{ \frac{p(p+C-1)}{p+\mu}} C_1(k, n, p,C,\mu) \mathscr{I}_{\m}(o) ,
\end{align*}
which proves   \eqref{lem_F*p_k>0=0}.

For  \eqref{v_numer_two_estimate_k>0k=0}, a direct calculation together with  \eqref{measexst_general_1_lem}--\eqref{basckestk>0_general_2_lem} and $b+n>0$ gives
\begin{align} \label{v_numer_second1_general_lem}
   \int_{M}  {|v_\iota|^{a}} {r^{b}} \dm
 & \leq \int_{S_oM}  \left( \int^{+\infty}_0 e^{-\tau(o,y)}e^{-\iota a r^{1+\frac{\mu}{p}} } r^{b}  e^{-(n-1)kr} \mathfrak{s}^{n-1}_{-k^2}(r) (r+1)^{- C} {\dd}r  \right) {\dd}\nu_o(y) \notag \\
&\leq  \mathscr{I}_{\m}(o) \left( 2\int^{\varepsilon_1}_0 r^{b+n-1}{\dd}r+   (2k)^{1-n} \int^1_{\varepsilon_1} r^b (r+1)^{-C}{\dd}r +  (2k)^{1-n} \int_{1}^{+\infty}   e^{-\iota a r^{1+\frac{\mu}{p}} } r^{b- C} {\dd}r  \right)\notag \\
 &\leq \mathscr{I}_{\m}(o) \left(   \frac{2}{b+n} +    \frac{\max\{1,\varepsilon_1^b \}}{(2k)^{n-1}} +  (2k)^{1-n} \int_{1}^{+\infty}   e^{-\iota a r^{1+\frac{\mu}{p}} } r^{b- C} {\dd}r  \right).
\end{align}
%Since $b + n >0,$
%we have
%\begin{align}\label{v_numer_second_0delta_general_2_lem}
%\int^{\varepsilon_1}_0 e^{-\iota a r^{1+\frac{\mu}{p}} } r^{b+n-1}{\dd}r \leq   \int^{1}_0   r^{b+n-1}{\dd}r
%= \frac{1}{b+n}.
%\end{align}
The remainder of the proof of  \eqref{v_numer_two_estimate_k>0k=0}  is divided into two subcases:  $ b \leq  C$ and $ C < b < p' C$.

\smallskip

\emph{Subcase 1.1: $b \leq  C$.}
In this subcase,
\begin{align*}
\int_{1}^{+\infty}   e^{-\iota a r^{1+\frac{\mu}{p}} } r^{b-C} {\dd}r
\leq \int_{1}^{+\infty}   e^{-\iota a r} {\dd}r
\leq   \iota^{-1} a^{-1}.
 \end{align*}
Since $\varepsilon_1=\varepsilon_1(k,n)$ depends only on $k$ and $n$,  it follows from \eqref{v_numer_second1_general_lem} and $\iota\in (0,1)$  that
\begin{align*}
 \int_{M}  |v_\iota|^{a} r^{b} \dm
  & \leq  \iota^{-1}\left(   \frac{2}{b+n} +    \frac{\max\{1,\varepsilon_1^b \}}{(2k)^{n-1}} +  \frac{1}{(2k)^{n-1}a}\right) \mathscr{I}_{\m}(o)= : \iota^{-1} C_2(a,b,k,n) \mathscr{I}_{\m}(o).
\end{align*}
In this subcase we take $\xi := -1$.  Since  $0<C-1<\mu<C$, we have
 \[ \frac{p+C-1}{p+\mu} + \frac{\xi}{p'} = \frac{p+C-1}{p+\mu} - \frac{p-1}{p} = \frac{p (C -\mu) +\mu}{p(p+\mu)} > 0,\]
 which proves  \eqref{v_numer_two_estimate_k>0k=0}.

\smallskip

\emph{Subcase 1.2: $C < b < p' C$.}
 Since
$\frac{p(b-C)-\mu}{p+\mu}  > -1$, by letting $z:= \iota a r^{1+\frac{\mu}{p}}$
we have
\begin{align}
 \int_{1}^{+\infty}   e^{-\iota a r^{1+\frac{\mu}{p}} } r^{b-C} {\dd}r
& \leq  \iota^{- \frac{p(b-C +1)}{p+\mu}}  a^{- \frac{p(b-C+1)}{p+\mu}} \left( \frac{p}{p+\mu} \right) \int^{+\infty}_0 e^{-z}  z^{ \frac{p(b-C)-\mu}{p+\mu}}   {\dd}z\notag\\
&= \iota^{- \frac{p(b-C+1)}{p+\mu}}  a^{- \frac{p(b-C+1)}{p+\mu}} \left( \frac{p}{p+\mu} \right) \Gamma\left(\frac{p(b-C)-\mu}{p+\mu} +1 \right)\notag\\
&=: \iota^{- \frac{p(b-C+1)}{p+\mu}} ~ \overline{\mathcal {C}}_2(a,b,p,C,\mu). \label{v_numer_second_deltainfty_general_1_lem}
\end{align}
Owing to $\iota\in (0,1)$, the inequalities \eqref{v_numer_second1_general_lem} and \eqref{v_numer_second_deltainfty_general_1_lem} then imply
\begin{align*}
 \int_{M}  |v_\iota|^{a} r^{b} \dm
  %& \leq  \left( \iota^{- \frac{p(b-C+1)}{p+\mu}} \bar{C}_2(a,b,p,C,\mu) + (2k)^{1-n} + (2k)^{1-n}\frac{2}{b +n} \right) \mathscr{I}_{\m}(o) \\
  & \leq \iota^{- \frac{p(b-C+1)}{p+\mu}}  \left(    \frac{2}{b+n} +    \frac{\max\{1,\varepsilon_1^b \}}{(2k)^{n-1}}  +\frac{\overline{\mathcal {C}}_2(a,b,p,C,\mu) }{(2k)^{n-1}}  \right) \mathscr{I}_{\m}(o) \\
  &=: \iota^{- \frac{p(b-C+1)}{p+\mu}} C_2(a,b,k,n,p,C,\mu) \mathscr{I}_{\m}(o).
\end{align*}
We choose $\xi := - \frac{p(b-C+1)}{p+\mu}$ for this subcase. Then
\begin{equation*}
\frac{p+C-1}{p+\mu} + \frac{\xi}{p'} =\frac{p+C-1}{p+\mu} - \frac{p(b-C+1)}{p'(p+\mu)} = \frac{(p-1)(p'C-b)}{ (p+\mu)} > 0,
\end{equation*}
 which proves \eqref{v_numer_two_estimate_k>0k=0}.

\medskip

\textbf{Case (ii) $k = 0$.} %Now the $S$-curvature bound is $\mathbf{S}(\nabla r) \geq \frac{C}{1+r}$ with $C \geq  n$.
In this case \eqref{measexst_general_1_lem} reads
\begin{equation}\label{measexst_general_lem}
\sigma(r, y) \leq e^{-\tau(o, y)} r^{n-1} (r+1)^{-C}, \quad \forall r \in (0, i_y), \; y \in S_oM.
\end{equation}

From $F^*({\dd}r)=1$, $\mu>0$, and $n-C-1<0$      along with \eqref{measexst_general_lem}, we deduce by a direct computation that
\begin{align}
\int_{M} F^{*p}({\dd}v_{\iota}) \dm
& = \iota^p \left(1+\frac{\mu}{p}\right)^p  \int_{M} r^{\mu} e^{- \iota p r^{1+\frac{\mu}{p}}} \dm \notag\\
& \leq  \iota^p \left(1+\frac{\mu}{p}\right)^p  \int_{S_oM}  \left( \int^{+ \infty}_0 e^{-\tau(o,y)} e^{- \iota p r^{1+\frac{\mu}{p}}} r^{\mu+n-1} (r+1)^{-C} {\dd}r  \right) {\dd}\nu_o(y) \notag \\
& \leq   \iota^p \left(1+\frac{\mu}{p}\right)^p  \mathscr{I}_{\m}(o)  \left(   \int^{1}_0 r^{\mu} (r+1)^{n- C -1}  {\dd}r
+ \int_1^{+\infty} e^{- \iota p r^{1+\frac{\mu}{p}}} r^{\mu} (r+1)^{n- C -1}  {\dd}r \right) \notag \\
& \leq  \iota^p \left(1+\frac{\mu}{p}\right)^p  \mathscr{I}_{\m}(o)   \left(1+    \int_{1}^{+\infty}  e^{- \iota p r^{1+\frac{\mu}{p}}} r^{\mu + n- C -1}   {\dd}r\right). \label{v_numer_first0_general_lem}
\end{align}
To estimate $\int_{1}^{+\infty}  e^{- \iota p r^{1+\frac{\mu}{p}}} r^{\mu + n- C -1} {\dd}r$, let $z:= \iota p r^{1+\frac{\mu}{p}}$.
Since $\frac{p(\mu + n-C -1) - \mu}{p+\mu} > -1$, we have
\begin{align*}
\int_{1}^{+\infty} e^{- \iota p r^{1+\frac{\mu}{p}}}
r^{\mu + n -C-1}  {\dd}r
 &\leq \iota^{\frac{p(C-n-\mu)}{p+\mu}} \frac{p^{\frac{p(C-n+1-\mu)+\mu}{p+\mu}} }{p+\mu} \int_{0}^{+\infty} e^{- z}
z^{\frac{p(\mu + n-C-1) - \mu}{p+\mu}}  {\dd}z\\
&=\iota^{\frac{p(C-n-\mu)}{p+\mu}} \frac{p^{\frac{p(C-n+1-\mu)+\mu}{p+\mu}} }{p+\mu} \Gamma\left(\frac{p(\mu + n-C-1) - \mu}{p+\mu}+1\right) \\ &=:\iota^{\frac{p(C-n-\mu)}{p+\mu}}\widetilde{\mathcal {C}}_1(n,p,C,\mu).
\end{align*}
 Then it follows from \eqref{v_numer_first0_general_lem}, $C\geq n$, and $\iota\in (0,1)$ that
\begin{align*}
\int_{M} F^{*p}({\dd}v_{\iota}) \dm
& \leq  \iota^{ \frac{p(p+C-n)}{p+\mu}} \left(1+  \widetilde{\mathcal {C}}_1(n,p,C,\mu) \right) \mathscr{I}_{\m}(o)  \left(1+\frac{\mu}{p}\right)^p \notag \\
%& \leq  \iota^{ \frac{p(p+C-n)}{p+\mu}} \left(1 +  \widetilde{\mathcal {C}}_1(n,p,C,\mu) \right) \mathscr{I}_{\m}(o)  \left(1+\frac{\mu}{p}\right)^p
& = : \iota^{ \frac{p(p+C-n)}{p+\mu}} C_1(k, n, p,C,\mu) \mathscr{I}_{\m}(o),
\end{align*}
which proves \eqref{lem_F*p_k>0=0}.

\smallskip

For \eqref{v_numer_two_estimate_k>0k=0}, together with  \eqref{measexst_general_lem} a direct calculation gives
\begin{align}
  \int_{M}  {|v_\iota|^{a}}{r^{b}} \dm
  & \leq \int_{S_oM}  \left( \int^{+ \infty}_0 e^{-\tau(o,y)}e^{-\iota a r^{1+\frac{\mu}{p}} } r^{b+n-1} (r+1)^{ - C  }   {\dd}r  \right) {\dd}\nu_o(y) \notag \\
 &\leq \mathscr{I}_{\m}(o) \left(  \int^{1}_0 e^{-\iota a r^{1+\frac{\mu}{p}} } r^{b+n-1}{\dd}r + \int_{1}^{+\infty}   e^{-\iota a r^{1+\frac{\mu}{p}} } r^{b+n - C -1} {\dd}r  \right)\notag\\
 &=  \mathscr{I}_{\m}(o) \left(  \frac{1}{b+n} + \int_{1}^{+\infty}   e^{-\iota a r^{1+\frac{\mu}{p}} } r^{b+n - C -1} {\dd}r  \right) .  \label{v_numer_second1_general}
\end{align}
%Since $b + n >0$, we have
%\begin{align}\label{v_numer_second_0delta_general_lem}
%\int^{1}_0 e^{-\iota a r^{1+\frac{\mu}{p}} } r^{b+n-1}{\dd}r \leq   \int^{1}_0   r^{b+n-1}{\dd}r
%= \frac{1}{b+n}.
%\end{align}
To estimate $\int_{1}^{+\infty}   e^{-\iota a r^{1+\frac{\mu}{p}} } r^{b+n - C -1} {\dd}r$, again we consider two subcases:

\emph{Subcase 2.1: $b \leq C-n+1$.}
In this subcase,
\begin{align*}
\int_{1}^{+\infty}   e^{-\iota a r^{1+\frac{\mu}{p}} } r^{b+n-C-1} {\dd}r
\leq \int_{1}^{+\infty}   e^{-\iota a r} {\dd}r
\leq  \iota^{-1} a^{-1}.
 \end{align*}
Together with \eqref{v_numer_second1_general} and $\iota\in (0,1)$, we obtain
\begin{align*}
  \int_{M}  {|v_\iota|^{a}}{r^{b}} \dm
 \leq \iota^{-1} \left(\frac{1}{b+n} +\frac{1}{a}   \right)  \mathscr{I}_{\m}(o) = : \iota^{-1} C_2(a,b,k,n) \mathscr{I}_{\m}(o).
\end{align*}
We then set $\xi := -1$ for this subcase. Since  $0\leq C-n < \mu < C-n +1$, we have
\[ \frac{p+C-n}{p+\mu} + \frac{\xi}{p'} = \frac{p+C-n}{p+\mu} - \frac{p-1}{p} = \frac{p(C-n+1-\mu) + \mu}{p(p+\mu)} > 0, \]
which proves \eqref{v_numer_two_estimate_k>0k=0}.

\medskip

\emph{Subcase 2.2: $C-n+1 < b < p'(C-n+1)$.}  Since
$\frac{p(b+n-C-1)-\mu}{p+\mu}  > -1$, by letting   $z:= \iota a r^{1+\frac{\mu}{p}}$
we have
\begin{align}
 \int_{1}^{+\infty}   e^{-\iota a r^{1+\frac{\mu}{p}} } r^{b+n-C-1} {\dd}r
& \leq  \iota^{- \frac{p(b+n-C)}{p+\mu}}  a^{- \frac{p(b+n-C)}{p+\mu}} \left(\frac{p}{p+\mu}\right) \int^{+\infty}_0 e^{-z}  z^{ \frac{p(b+n-C-1)-\mu}{p+\mu}}   {\dd}z\notag\\
&= \iota^{- \frac{p(b+n-C)}{p+\mu}}  a^{- \frac{p(b+n-C)}{p+\mu}} \left(\frac{p}{p+\mu}\right) \Gamma\left(\frac{p(b+n-C-1)-\mu}{p+\mu} +1 \right)\notag\\
&=:\iota^{- \frac{p(b+n-C)}{p+\mu}}   \widetilde{\mathcal {C}}_2(a,b, n, p,C,\mu). \label{v_numer_second_deltainfty_general_lem}
\end{align}
Since $\iota\in (0,1)$, the inequalities    \eqref{v_numer_second1_general} and \eqref{v_numer_second_deltainfty_general_lem} then imply
\begin{align*}%\label{v_numer_general}
  \int_{M}  {|v_\iota|^{a}}{r^{b}} \dm
%& \leq  \left( \iota^{- \frac{p(b+n-C)}{p+\mu}} \bar{C}_4(a,b, n, p,\mu) + \frac{1}{b+n} \right)     \mathscr{I}_{\m}(o) \notag\\
& \leq  \iota^{- \frac{p(b+n-C)}{p+\mu}} \left(    \frac{1}{b+n}+  \widetilde{\mathcal {C}}_2(a,b, n, p,\mu) \right)  \mathscr{I}_{\m}(o)
=: \iota^{- \frac{p(b+n-C)}{p+\mu}} C_2(a,b,k,n,p,C,\mu).
\end{align*}
In this subcase choose $\xi: = - \frac{p(b+n-C)}{p+\mu}$. Then
\begin{equation*}
\frac{p+C-n}{p+\mu} +\frac{\xi}{p'}  =\frac{p+C-n}{p+\mu} - \frac{p(b+n-C)}{p'(p+\mu)} = \frac{ (p-1)\left( p'(C-n+1)-b \right)}{p+\mu} > 0,
\end{equation*}
which proves \eqref{v_numer_two_estimate_k>0k=0}. This concludes the proof.
\end{proof}

%Now by Lemma \ref{lem_general}, the failure of the inequalities can be obtained much more efficient.

\begin{proof}[Proof of Theorem \ref{voluemcompar22}] In view of \eqref{granotwoff}, we work with $F^*(\mathrm{d}f)$ rather than $F(\nabla f)$ to prove the statements.

\textbf{(\ref{weakHardyinB})}--\textbf{(\ref{geenerCknineq})}
Choose the test functions $v_\iota(x) = -e^{-\iota r^{1 + \frac{\mu}{p}}}$  and set
\[
\begin{cases}
a:=p,\ b:= -p & \text{ for } \eqref{weakHardyinB}, \\
a:=p,\ b:= p's  & \text{ for } \eqref{generunicertpinweakcurva}, \\
a:=p'(m-1),\ b:= p's  & \text{ for } \eqref{geenerCknineq}.
\end{cases}
\]
%One checks directly that  these choices satisfy the hypotheses of Lemma \ref{lem_general}.
A direct verification shows that this construction fulfills all requirements of Lemma~\ref{lem_general}.

Arguing as in the proof of Theorem \ref{MainThmBerw}/\eqref{Berw4}, we have
\begin{align*}
\inf_{f \in C^\infty_0(M) \setminus \{0\}} \mathbf{J}(f)   \leq  \mathbf{J} (v_\iota),
\end{align*}
where  $\mathbf{J}$  is  the functional from Definition \ref{function11}, i.e.,
\[
\mathbf{J}:=
\begin{cases}
\mathscr{H}_{o,p} & \text{ for } \eqref{weakHardyinB}, \\
\mathscr{U}_{o,p,s}    & \text{ for } \eqref{generunicertpinweakcurva}, \\
\mathscr{C}_{o,p,m,s} & \text{ for } \eqref{geenerCknineq}.
\end{cases}
\]
Consequently, it suffices to prove that
\[
\lim_{\iota \to 0^+} \mathbf{J}(v_\iota) = 0.
\]
In view of \eqref{lem_F*p_k>0=0} (for \eqref{weakHardyinB}) and \eqref{v_numer_estimate_both} (for \eqref{generunicertpinweakcurva}\eqref{geenerCknineq}), the numerator of $\mathbf{J}(v_\iota)$ vanishes as $\iota\rightarrow 0^+$.
 The proof will therefore be complete once we verify that the denominator of
$\mathbf{J}(v_\iota)$ admits a uniformly positive lower bound independent of $\iota$.

To do this, let $\epsilon\in (0,1)$ be as in Remark \ref{remaep}. Since $\iota\in (0,1)$, we have
\[
|v_\iota|(x)\geq e^{-\iota    \epsilon^{1+\mu/p}}\geq e^{-     \epsilon^{1+\mu/p}}\geq e^{-1}, \qquad \forall x\in B^+_\epsilon(o).
\]
Given $\varrho>0$ and $\varsigma>-n$, the above inequality combined with the lower-bound condition in \eqref{limcaes} implies
\begin{equation}\label{standintegr_Knegative}
\int_M  |v_\iota|^\varrho r^\varsigma \dm \geq e^{-\varrho}\int_{B^+_\epsilon(o)}r^\varsigma \dm\geq \frac{1}{2e^\varrho}\int_{S_oM}\left( \int^\epsilon_0 e^{-\tau(o,y)}r^{n+\varsigma-1} {\dd}r\right) {\dd}\nu_o(y)=\frac{\mathscr{I}_{\m}(o)}{2e^\varrho}\epsilon^{n+\varsigma},
\end{equation}
where
\[
\begin{cases}
\varrho:=p,\ \varsigma:= -p & \text{ for } \eqref{weakHardyinB}, \\
\varrho:=p,\ \varsigma:= s-1  & \text{ for  } \eqref{generunicertpinweakcurva}, \\
\varrho:=m,\ \varsigma:= s-1  & \text{ for } \eqref{geenerCknineq}.
\end{cases}
\]
Consequently, the denominator of
 $\mathbf{J}(v_\iota)$ are bounded below by a positive constant independent of $\iota$, which concludes the proofs of  \eqref{weakHardyinB}--\eqref{geenerCknineq}.

\smallskip

\textbf{(\ref{generCKNin})}  Set $h:=\frac{1}{n-1}\inf_{(x,y)\in TM\backslash\{0\}}\frac{\mathbf{S}(x,y)}{F(x,y)}$. The assumption implies
$\mathbf{S}(\nabla r) \geq (n-1)h$ and $h\in (k,+\infty)$.
By   \eqref{Scurvaturedef} and  \eqref{geommeaingofr} we have
\begin{align*}
\tau(\gamma_y(r),\dot{\gamma}_y(r))-\tau(o,y)=\int^r_0 \mathbf{S}(\gamma_y(t),\dot{\gamma}_y(t)){\dd}t\geq (n-1)hr.
\end{align*}
Together with  Theorem \ref{bascivolurcompar} this yields
\begin{equation}\label{st11ong}
\hat{\sigma}_o(r,y)\leq e^{-\tau(\gamma_y(r),\dot{\gamma}_y (r))}  \mathfrak{s}_{k}^{n-1}(r)\leq e^{-\tau(o,y)-(n-1)hr}\mathfrak{s}_{k}^{n-1}(r),\quad  \forall  r\in (0,i_y), \ y\in S_oM.
\end{equation}

Now we  construct new  test functions
\[
u_\iota(x):=-\left[1+(\iota r(x))^{1+\frac{s}{p-1}}  \right]^{\frac{p-1}{p-m}}, \qquad \iota\in (0,1).
\]
By an argument analogous to the preceding one, we are reduced to showing
\[
\lim_{\iota\rightarrow0^+}\mathscr{C}_{o,p,m,s}(u_\iota)=0.
\]

Note  that the assumption implies
\begin{equation}\label{numerspcon}
\min\{n,m\}-p>0, \quad n+s-1>0, \quad s+p>1, \quad n+p's>p+p's>0.
\end{equation}
Thus,
a similar calculation to \eqref{standintegr_Knegative} combined with \eqref{numerspcon} yields
\begin{align}\label{semaikcn}
\int_{{M}} |u_\iota|^m r^{s-1} {\dm} &\geq 2^{\frac{m(p-1)}{p-m}-1} \mathscr{I}_{\m}(o)\int^\epsilon_0   r^{n+s-2}{\dd}r
= \mathscr{I}_{\m}(o) \frac{2^{\frac{p(m-1)}{p-m}}\epsilon^{n+s-1}}{(n+s-1)}  ,
\end{align}
where $\epsilon\in (0,1)$ is  taken as in Remark \ref{remaep}. Hence, the denominator of
$\mathscr{C}_{o,p,m,s}(u_\iota)$ admits a uniformly positive lower bound independent of $\iota$. It remains to show the numerator approaches to $0$ as $\iota\rightarrow 0^+$.

%Fix $\epsilon \in (0,1)$ as in Remark \ref{remaep}.  Since  $p-m<0$ and $s+p >1$, we have for $r\in (0, \epsilon)$
%\[ |u_\iota|^m = \left[1+(\iota r(x))^{1+\frac{s}{p-1}}  \right]^{\frac{m(p-1)}{p-m}} \geq 2^{\frac{m(p-1)}{p-m}}. \]
%Hence, it follows from    \eqref{limcaes} and $n+s-1>0$  that
%\begin{align}\label{semaikcn}
%\int_{{M}} |u_\iota|^m r^{s-1} {\dm} &\geq 2^{\frac{m(p-1)}{p-m}-1} \mathscr{I}_{\m}(o)\int^\epsilon_0   r^{n+s-2}{\dd}r
%= \mathscr{I}_{\m}(o) \frac{2^{\frac{p(m-1)}{p-m}}\epsilon^{n+s-1}}{(n+s-1)}  .
%\end{align}

To achieve this,
since $\mathfrak{s}^{n-1}_{-k^{2}}(r) \sim r^{n-1}$ as $r \to 0^{+}$, there exists $\varepsilon = \varepsilon(n,k) \in (0,1)$ for which
\begin{equation}\label{estkh1}
r^{p's}\mathfrak{s}^{n-1}_{-k^2}(r)\leq 2r^{n+p's-1}, \qquad \forall  r\in (0,\varepsilon).
\end{equation}
Because $\eta := h - k > 0$, a constant $\delta = \delta(n,p,s,k,h) > \varepsilon$ can be selected so that
\begin{equation}\label{estkh2}
e^{-(n-1)hr}\mathfrak{s}_{-k^2}^{n-1}(r)r^{p's}\leq  e^{-\frac{(n-1)\eta}2 r},\qquad \forall  r\in (\delta,+\infty).
\end{equation}
It now follows from \eqref{st11ong}, \eqref{numerspcon}, \eqref{estkh1}, and \eqref{estkh2} that
\begin{align}
 \int_M |u_\iota|^{p'(m-1)} r^{p's}\dm
 \leq \  &  \int_{S_oM}\left( \int^{+\infty}_0  e^{-\tau(o,y)-(n-1)hr} r^{p's}\mathfrak{s}_{-k^2}^{n-1}(r){\dd}r \right){\dd}\nu_o(y)\notag\\
\leq \ & \mathscr{I}_{\m}(o)\left(2 \int^{{\varepsilon}}_0r^{n+p's-1} {\dd}r+\int_{\varepsilon}^{\delta}    e^{-(n-1)hr}r^{p's}\mathfrak{s}_{-k^2}^{n-1}(r){\dd}r+\int_{\delta}^{+\infty}   e^{-\frac{(n-1)\eta}2 r} {\dd}r              \right)\notag\\
%\leq & \mathscr{I}_{\m}(o)\left(   \frac{2{\varepsilon}^{n+p's}}{n+p's}+\int_{\varepsilon}^{\delta}    e^{-(n-1)hr}r^{p's}\mathfrak{s}_{-k^2}^{n-1}(r){\dd}r+\frac{2 e^{-\frac{(n-1)\eta {\delta}}{2}}}{(n-1)\eta} \right)\nonumber\\
=: \ &  \mathscr{I}_{\m}(o)\,\mathcal {C}  (n,p,s,k,h,{\delta},{\varepsilon})<+\infty.\label{cknes1}
\end{align}
Moreover, a direct calculation yields
\begin{align*}
\int_M F^{*p}({\dd}u_\iota)\dm=\iota^{p+p's}\left| \frac{p-1+s}{p-m}  \right|^p \int_M |u_\iota|^{p'(m-1)} r^{p's}\dm,
\end{align*}
which together with \eqref{semaikcn} and \eqref{cknes1}  implies that the numerator of
$\mathscr{C}_{o,p,m,s}(u_\iota)$  tends to $0$ as $\iota\rightarrow 0^+$.
\end{proof}

Theorem \ref{voluemcompar22} clarifies the geometric reasons behind the phenomena on Berwald's metric space described by Theorem \ref{Berwmainth}/\eqref{CKNFIALBE4}\eqref{CKNFIALBE5} and Theorem \ref{sbeer}/\eqref{BerwCKN_i}: the complete failure of the Hardy and generalized uncertainty principles, and the non-universal failure of the generalized CKN inequality. Furthermore, a comparison of the results in \cite{HKZ,MPV,Ka,ZhaoHardy} with Theorem \ref{ThmFunkInequalities} and \eqref{uncertainfalyonB2} shows that, in the absence of finite reversibility, upper bounds on the flag and $S$-curvatures alone are insufficient to recover these inequalities. The residual validity of the CKN inequality in the Berwald case is due to its special $(\alpha,\beta)$-structure, which implies a hidden Riemannian compatibility --- a feature not present in general Finsler metrics.

\begin{remark}\label{S-weakcurvature}
Lemma \ref{lem_general}, and hence Theorem \ref{voluemcompar22}, remain valid under the weaker condition \eqref{weakSball}. We briefly outline the necessary modifications. Set \(U := \{ y \in S_o M : i_y > R \}\). The estimate \eqref{measexst_general_1_lem} must then be replaced by
\[
\hat{\sigma}_o(r, y) \le \mathcal{I}_R(y)\, e^{-(n-1)kr}\, \mathfrak{s}_{-k^2}^{\,n-1}(r)\,(r+1)^{-C}, \quad \text{ for }r \in (R, i_y),\; y \in U,
\]
where \(\mathcal{I}_R(y) := (1+R)^C e^{-\tau(\gamma_y(R),\dot{\gamma}_y(R)) + (n-1)kR}\). The forward completeness implies that \(\mathcal{I}_R\) is continuous on \(S_o M\); hence \(\mathscr{I}_R(o) := \int_{S_o M} \mathcal{I}_R(y) {\dd}\nu_o(y) < +\infty\). Consequently, for any \(f \in L^1(M,\m)\),
\begin{align*}
\int_M |f| \dm&=\int_{B^+_R(o)}|f|\dm +\int_{M\backslash B^+_R(o)}|f|\dm=\int_{B^+_R(o)}|f|\dm +\int_{U}\left(\int_R^{i_y}|f|(r,y)\,\hat{\sigma}_o(r,y){\dd}r\right){\dd}\nu_o(y)\\
&\leq \int_{B^+_R(o)}|f|\dm+ \mathscr{I}_R(o)\int_R^{+\infty}|f|(r,y)e^{-(n-1)kr} \mathfrak{s}^{n-1}_{-k^2}(r)\,(r+1)^{-C}{\dd}r.
\end{align*}
Together with Remark \ref{remaep}, this estimate allows the proof of Lemma \ref{lem_general} to be carried out with only  slight adjustments. We leave the precise formulation of the resulting statements to the interested reader.
\end{remark}

\appendix

\section{Qualitative properties of Berwald's metric}\label{axiberwmetric}

This appendix is devoted to the proof of Lemma \ref{lemmaBerwaldquartic}, following a method inspired by \cite{Masca}.
By  \eqref{defleg} and \eqref{defbasictensor}, the
Legendre transformation of Berwald's metric $\B$ is given by
\[
\mathsf{p}:=\mathfrak{L}(x,y) = g_y(y,\cdot) = \frac{1}{2}\frac{\partial({\B}^2)}{\partial y^i}(x,y) {\dd}x^i, \qquad  \forall y\in T_xM\backslash\{0\}.
\]
 It follows from \eqref{F*F} that
\[
\B(x, y)=\B^*(x, \mathsf{p}), \qquad  \forall y\in T_xM\backslash\{0\}.
\]
Thus, the process of determining the co-metric $\B^*=\B^*(x,\xi)$ is divided into two steps:
\begin{enumerate}[\rm (1)]
 \item Express $\B(x, y)$ in terms of $x$ and $\mathsf{p}$ to obtain $\B^*(x,\mathsf{p})$;
 \item Replace $\mathsf{p}$ with the a general co-vector $\xi$ and obtain  the resulting function $\B^*(x, \xi)$.
\end{enumerate}

\begin{proof}[Proof of Lemma \ref{lemmaBerwaldquartic}]
Recall Berwald's metric $\B=\B(x,y)$ admits an $(\alpha,\beta)$-representation as in \eqref{Berwaldphi}:
\begin{equation*}
\B = \frac{(\alpha+\beta)^2}{\alpha} = \alpha (1+s)^2, \quad s=\frac{\beta}{\alpha},
\end{equation*}
where $\alpha$ is a Riemannian metric and $\beta$ is a 1-form given in \eqref{exparealphbeata}.
For convenience, we introduce the notations:
\[
a_{ij} := \frac{1}{2} \frac{\partial^2 (\alpha^2)}{\partial y^i \partial y^j}, \quad b_i := \frac{\partial \beta}{\partial y^i}, \quad \mathsf{p}_i := \frac{1}{2} \frac{\partial({\B}^2)}{\partial y^i}.
\]
Then
\[
 \alpha^2 = a_{ij} y^i y^j, \quad \beta = b_i y^i, \quad \mathsf{p}=\mathsf{p}_i{\dd}x^i.
 \]
 Moreover, let $(a^{ij})$ denote  the inverse matrix of $(a_{ij})$ and set
 \[
 y_i := a_{ij} y^j, \quad b^i := a^{ij} b_j,\quad \mathsf{p}^i := a^{ij} \mathsf{p}_j.
 \]
 Consequently,
 \[
 \alpha^2 = y_i y^i, \quad \beta= b^i y_i\quad \B^2 = \mathsf{p}_i y^i = \mathsf{p}^i y_i.
 \]
Introduce the quantities
\[
 b: = \sqrt{a_{ij} b^i b^j} = \sqrt{b_i b^i} = |x|, \quad \alpha^* : = \sqrt{a_{ij} \mathsf{p}^i \mathsf{p}^j} = \sqrt{\mathsf{p}_j \mathsf{p}^j}, \quad \beta^{*} : = b^i \mathsf{p}_i=b_i\mathsf{p}^i.
 \]
 That is, $b$ and   $\alpha^*$ are the $\alpha$-norms of $\beta$ and $\mathsf{p}$ respectively, while $\beta^*$ is the $\alpha$-inner product of $\beta$ and $\mathsf{p}$.

 In the sequel, we will establish an equation involving $\B$, $x$, $\alpha^*$ and $\beta^*$, which means that $\B$ can be expressed into a function of $x$ and  $\mathsf{p}$.
 To do this, a direct computation gives
\begin{equation}\label{p_i}
\mathsf{p}_i = \frac{1}{2} \frac{\partial({\B}^2)}{\partial y^i} = \frac{2 (\alpha+ \beta)^3 b_i}{\alpha^2} + \frac{(\alpha-\beta)(\alpha+\beta)^3 y_i}{\alpha^4} =  2 \B \frac{(\alpha+ \beta) b_i}{\alpha} + \frac{(\alpha+\beta)^3 (\alpha-\beta) y_i}{\alpha^4}.
\end{equation}
Contracting \eqref{p_i} with $\mathsf{p}^i$ and $b^i$, respectively, yields
\begin{equation*}
\alpha^{*2} = 2 \B \frac{\alpha+\beta}{\alpha} \beta^{*} + \B^2 \frac{(\alpha+\beta)^3 (\alpha-\beta)}{\alpha^4}, \qquad \beta^{*} = 2 \B \frac{\alpha+\beta}{\alpha} b^2 + \B \frac{ (\alpha+ \beta) (\alpha- \beta) \beta}{\alpha^3}.
\end{equation*}
Substituting $s=\frac{\beta}{\alpha}$ into the above expressions gives
\begin{align*}
\alpha^{*2}   =2  (1+s) \beta^{*}\B +  (1+s)^3 (1-s)\B^2,\qquad
\beta^{*}   = 2  (1+s) b^2\B +   (1+ s)(1- s) s\B.
\end{align*}
Setting $t = 1+s$ transforms these relations into
\begin{align}
\label{alphastar} \B^2 t^4 - 2 \B^2 t^3 - 2 \B  \beta^{*} t + \alpha^{*2} & =0, \\
\label{betastar} \B t^3  -  3 \B t^2 + 2(1- b^2) \B t  + \beta^{*} & =0.
\end{align}
In the following  steps, our goal is to eliminate $t$ from \eqref{alphastar} and \eqref{betastar} to obtain a polynomial relation involving only
$\B$, $\alpha^*$ and $\beta^{*}$.

First, by computing $\eqref{alphastar} - \eqref{betastar} \times \B t$, we have
\begin{equation}\label{alphabeta_1}
\B^2 t^3 - 2  (1 -b^2)\B^2 t^2  - 3 \beta^{*} \B t + \alpha^{*2} =0.
\end{equation}
Next,
$\eqref{alphabeta_1} -\eqref{betastar}\times \B$ gives
\begin{equation}\label{alphabeta_2}
(2 b^2 +1) \B^2 t^2 +  (2 b^2 \B -2 \B -3 \beta^{*})\B t- \beta^{*} \B +\alpha^{*2} =0.
\end{equation}
Then, $\eqref{alphabeta_2}\times t - \eqref{betastar}\times (2 b^2 +1)\B$ yields
\begin{equation}\label{alphabeta_3}
 (8 b^2 \B + \B -3\beta^{*}) \B t^2+ \Big\{ 2 (2 b^4 - b^2  - 1 )\B^2-  \beta^{*} \B +\alpha^{*2}\Big\} t - \beta^{*} (2 b^2 +1) \B =0.
\end{equation}
Furthermore,
$\eqref{alphabeta_2}\times  (8 b^2 \B + \B -3\beta^{*}) - \eqref{alphabeta_3}\times  (2 b^2 +1)\B$ produces
\begin{equation}\label{alphabeta_4}
\begin{split}
& \Big\{ -8 b^2 (1-b^2)^2 \B^2 + 4  (1 -7 b^2)\beta^{*} \B  - (1+ 2 b^2) \alpha^{*2} + 9 \beta^{*2} \Big\} \B t
\\
& - 4 b^2 (1-b^2) \beta^{*} \B^2+(8 b^2 \alpha^{*2}+ \alpha^{*2} + 3 \beta^{*2} ) \B -3 \alpha^{*2} \beta^{*} =0.
\end{split}
\end{equation}
Define \[\Xi:=-8 b^2 (1-b^2)^2 \B^2 + 4  (1 -7 b^2)\beta^{*} \B  - (1+ 2 b^2) \alpha^{*2} + 9 \beta^{*2}.\]
From \eqref{alphabeta_4} we have
 \begin{equation}\label{alphabeta_5}
\begin{split}
 \Xi^2 \B^2 t^2 -
 \Big\{ - 4 b^2 (1-b^2) \beta^{*} \B^2 + (8 b^2 \alpha^{*2}+ \alpha^{*2} + 3 \beta^{*2} ) \B -3 \alpha^{*2} \beta^{*} \Big\}^2 =0.
\end{split}
\end{equation}
Finally, eliminating $t$ and $t^2$  via the linear combination
$\frac{-1}{(1+2|x|^2)^2}\Big\{ \eqref{alphabeta_2}\times \Xi^2 - \eqref{alphabeta_5}\times (2 b^2 +1)  -  \eqref{alphabeta_4} \times (2 b^2 \B -2 \B -3 \beta^{*})\Xi \Big\}$
and substituting $b^2 = |x|^2$ results the desired  quartic equation
\begin{equation*}
\begin{split}
& 16 |x|^2 (1-|x|^2)^2\Big\{ (1-|x|^2) \alpha^{*2} + \beta^{*2} \Big\}\B^{4}
+ 8 \Big\{ ( 10 |x|^2 -1)(1-|x|^2) \alpha^{*2}\beta^{*} + (9 |x|^2 -1)\beta^{*3} \Big\} \B^{3}
\\
&+ \Big\{  (1 -20 |x|^2 -8 |x|^4) \alpha^{*4} +6 (6|x|^2 -5)\alpha^{*2} \beta^{*2}-27 \beta^{*4} \Big\} \B^{2}
+ 12 \alpha^{*4} \beta^{*} \B - \alpha^{*6} =0,
\end{split}
\end{equation*}
which means that $\B$ is a function of $x$ and $\mathsf{p}$, i.e., $\B(x,y)=\B^*(x,\mathsf{p})$.
 Now replacing $\mathsf{p}$ with $\xi$ yields
 \eqref{Berwaldquartic}. This completes the proof.
\end{proof}

\end{document}